\documentclass[11pt,reqno]{amsart}

\usepackage{amsmath, amssymb, amsthm, mathrsfs, stmaryrd, verbatim, graphicx, wrapfig, lipsum, mathtools, mathabx, epsfig, framed}

\expandafter\def\csname opt@stmaryrd.sty\endcsname{only,shortleftarrow,shortrightarrow}

\usepackage{extpfeil}

\usepackage[foot]{amsaddr}

\usepackage[hang,flushmargin]{footmisc}

\usepackage[bookmarks, bookmarksdepth=2, colorlinks=true, linkcolor=blue, citecolor=blue, urlcolor=blue, unicode]{hyperref}

\usepackage[hyperref=true, style=alphabetic, isbn=false, url=false, giveninits=true, sorting=nyt]{biblatex}
\usepackage{mathscinet}

\DeclareLabelalphaTemplate{
	\labelelement{
		\field[final]{shorthand}
	}
	\labelelement{
		\field[ifnames=1, strwidth=3]{labelname}
		\field[ifnames=2, varwidth, strwidthmax=2]{labelname}
		\field[ifnames={3-}, varwidthlist, strwidthmax=2, names=4]{labelname}
	}
	\labelelement{
		\field[strwidth=2,strside=right]{year}
	}
}

\DeclareFieldFormat{postnote}{#1}
\DeclareFieldFormat{multipostnote}{#1}

\usepackage{tikz, tikz-cd}
\usetikzlibrary{matrix,positioning,arrows,shapes}
\tikzset{>=stealth}
\tikzcdset{arrow style=tikz, diagrams={>=stealth}}
\tikzset{rot90/.style={anchor=south, rotate=90, inner sep=.5mm}}

\usepackage{geometry}
\numberwithin{equation}{subsection}
\newtheorem{theorem}[equation]{Theorem}
\newtheorem{proposition}[equation]{Proposition}
\newtheorem{corollary}[equation]{Corollary}
\newtheorem{lemma}[equation]{Lemma}

\newtheorem{deftheorem}[equation]{Definition/Theorem}

\newtheoremstyle{named}{}{}{\itshape}{}{\bfseries}{.}{.5em}{\thmname{#3} \thmnumber{#2}}
\theoremstyle{named}

\theoremstyle{definition}
\newtheorem{definition}[equation]{Definition}

\newtheorem{example}[equation]{Example}

\newtheorem{notation}[equation]{Notation}

\theoremstyle{remark}
\newtheorem{remark}[equation]{Remark}

\newcommand{\inv}{^{-1}}
\renewcommand{\le}{\leqslant}
\renewcommand{\ge}{\geqslant}

\newcommand{\widesim}[2][1.5]{\mathrel{\overset{#2}{\scalebox{#1}[1]{$\sim$}}}}
 
\newcommand{\thereis}{\,\exists\,}

\newcommand{\surj}{\twoheadrightarrow}

\newcommand{\inj}{\hookrightarrow}

\newcommand{\bij}{\xrightarrow{\raisebox{-1.2ex}[0ex][1ex]{$\sim$}}}
\newcommand{\bijects}{\xrightarrow{\,\raisebox{-1.2ex}[0ex][1ex]{$\widesim[1]{}$\,}}}
\newcommand{\aro}{\xrightarrow{\;\;\;\;\;}}

\newcommand{\acts}{\mathrel{\reflectbox{$\righttoleftarrow$}}}
\newcommand{\racts}{\righttoleftarrow}

\usepackage{xparse}
\NewDocumentCommand{\dotimes}{t_}{\IfBooleanTF{#1}{\otimeop}{\otimes}}
\NewDocumentCommand{\otimeop}{m}{\mathbin{\mathop{\otimes}\displaylimits_{#1}}}
\NewDocumentCommand{\lotimes}{t_}{\IfBooleanTF{#1}{\lotimeop}{\otimes}}
\NewDocumentCommand{\lotimeop}{m}{\mathbin{\mathop{\overset{\mathrm{L}}{\otimes}}\displaylimits_{#1}}}

\makeatletter
\newcommand{\xmapsfrom}[2][]{%
	\ext@arrow3095\leftarrowfill@{#1}{#2}\mapsfromchar}
\makeatother

\usetikzlibrary{calc,decorations.pathmorphing,shapes}
\newcounter{sarrow}

\newextarrow{\xtwoarrows}{{20}{20}{20}{20}}
{\bigRelbar\bigRelbar{\bigtwoarrowsleft\rightarrow\rightarrow}}

\DeclareMathOperator{\Ad}{Ad}

\newcommand{\cHom}{\mathcal{H}om}

\DeclareMathOperator{\rank}{rank}
\DeclareMathOperator{\Ext}{Ext}

\DeclareMathOperator{\End}{End}

\DeclareMathOperator{\Sym}{Sym}

\DeclareMathOperator{\ad}{ad}

\DeclareMathOperator{\ind}{ind}

\DeclareMathOperator{\Mod}{Mod}

\newcommand{\BA}{\mathbb{A}}

\newcommand{\BC}{\mathbb{C}}
\newcommand{\BD}{\mathbb{D}}

\newcommand{\BP}{\mathbb{P}}
\newcommand{\BQ}{\mathbb{Q}}

\newcommand{\BZ}{\mathbb{Z}}

\newcommand{\cD}{\mathcal{D}}

\newcommand{\cF}{\mathcal{F}}

\newcommand{\cH}{\mathcal{H}}

\newcommand{\cI}{\mathcal{I}}

\newcommand{\cL}{\mathcal{L}}

\newcommand{\cM}{\mathcal{M}}

\newcommand{\cN}{\mathcal{N}}

\newcommand{\cO}{\mathcal{O}}

\newcommand{\cU}{\mathcal{U}}

\newcommand{\cV}{\mathcal{V}}

\newcommand{\cZ}{\mathcal{Z}}

\newcommand{\fb}{\mathfrak{b}}

\newcommand{\fc}{\mathfrak{c}}

\newcommand{\fg}{\mathfrak{g}}

\newcommand{\fh}{\mathfrak{h}}

\newcommand{\fl}{\mathfrak{l}}

\newcommand{\fn}{\mathfrak{n}}

\newcommand{\fp}{\mathfrak{p}}

\newcommand{\fs}{\mathfrak{s}}

\newcommand{\fu}{\mathfrak{u}}

\usepackage[normalem]{ulem}
\usepackage{multicol}
\usepackage{enumitem}
\usepackage{adjustbox}
\usepackage{tabu}

\newcommand{\dtri}[3]{%
	\begin{tikzcd}[column sep = tiny, row sep = small, ampersand replacement = \&]
		{#1} \ar[rr] \&\& {#2} \ar[dl]\\
		\& {#3} \ar[ul,"{[1]}"]
\end{tikzcd}}

\DeclareMathOperator{\ch}{ch}

\title[Kazhdan-Lusztig Algorithm for Whittaker Modules]{Kazhdan-Lusztig Algorithm for Whittaker Modules with Arbitrary Infinitesimal Characters}

\author{Qixian Zhao}
\address{Department of Mathematics, University of Utah, Salt Lake City, Utah 84112, U.S.A.}
\email{zhao@math.utah.edu}

\address{Yau Mathematical Sciences Center, Tsinghua University, Haidian District, Beijing 100084, China}
\email{zhao\_qixian@tsinghua.edu.cn}

\date{\today}

\keywords{Whittaker modules, Character formula, D-modules, Localization, Kazhdan-Lusztig polynomials}

\subjclass[2020]{17B10, 32C38}

\begin{document}
\allowdisplaybreaks

\begin{abstract}
	Let $\mathfrak{g}$ be a complex semisimple Lie algebra. We give a description of characters of irreducible Whittaker modules for $\mathfrak{g}$ with any infinitesimal character, along with a Kazhdan-Lusztig algorithm for computing them. This generalizes Mili\v ci\'c-Soergel's and Romanov's results for integral infinitesimal characters. As a special case, we recover the non-integral Kazhdan-Lusztig conjecture for Verma modules.
\end{abstract}

\maketitle
\tableofcontents

\section{Introduction}\label{sec:intro}

Let $\fg$ be a complex semisimple Lie algebra. Let $\fn \subset \fg$ be a maximal nilpotent subalgebra, and let $\cZ(\fg)$ denote the center of the enveloping algebra $\cU(\fg)$ of $\fg$. This paper studies \textit{Whittaker modules} which are finitely generated $\fg$-modules that are locally finite over both $\fn$ and $\cZ(\fg)$. They originate from the study of Whittaker functionals for reductive groups which were first considered by Jacquet \cite{Jacquet:Whittaker} to study principal series representations of Chevalley groups. In the context of group representations, a Whittaker functional is, roughly speaking, a linear functional on a smooth representation of a reductive group $G$ (over a local field) that transforms according to a one dimensional representation of a unipotent subgroup $N$. The Whittaker modules we consider are natural analogous (for Lie algebras) of representations generated by a Whittaker functional.

Kostant first studied Whittaker modules in his beautiful paper \cite{Kostant:Whittaker}. He showed that when the $\fn$-character $\eta$ is non-degenerate (see \textsection\ref{subsec:Wh_prelim} for the definition of non-degeneracy) any cyclic Whittaker modules that is generated by an $\eta$-eigenvector and admits an infinitesimal characters is irreducible, and the category of Whittaker modules $\cN_\eta$ on which $\fn$ acts by $\eta$ (defined in \textsection\ref{subsec:Wh_prelim}) is equivalent to the category of finite dimensional $\cZ(\fg)$-modules. Later Mili\v ci\'c-Soergel \cite[\textsection 5]{Milicic-Soergel:Whittaker_geometric} showed that the subcategory $\cN_{\theta,\eta}$ of modules on which $\cZ(\fg)$ acts by a fixed infinitesimal character $\chi_\theta$ (defined at the beginning of \textsection\ref{sec:prelim}) is semisimple. It is then natural to ask for a description of the category $\cN_{\theta,\eta}$ in the degenerate case and, in particular, a description of composition series of cyclic modules. Towards this direction, McDowell \cite{McDowell:Whittaker} constructed and studied \textit{standard Whittaker modules} (Definition \ref{def:std_Wh_mods}), which are analogs of Verma modules. Using algebraic methods, Mili\v ci\'c and Soergel later showed that cyclic modules are filtered by standard modules (in fact a direct sum of standard modules in nice cases), and composition series of standard modules are described when the infinitesimal character $\chi_\theta$ is integral and regular \cite{Milicic-Soergel:Whittaker_algebraic}. Here integrality is a usual assumption and is the ``basic case'' compared to general infinitesimal characters.

It was observed by Mili\v ci\'c-Soergel \cite{Milicic-Soergel:Whittaker_geometric} that Kostant's result in the non-degenerate case has an alternative proof based on the localization theory of Beilinson-Bernstein \cite{Beilinson-Bernstein:Localization}, and the degenerate case should be also solvable using localization, similar to the solution of the Kazhdan-Lusztig conjecture for Verma modules. However, the latter is based on the decomposition theorem for perverse sheaves (or equivalently regular holonomic $\cD$-modules) due to Beilinson-Bernstein-Deligne \cite{Beilinson-Bernstein-Deligne:Decomposition} which does not apply to localizations of Whittaker modules \--- they have irregular singularities. Therefore, the argument for Verma modules was not translated to Whittaker modules until a decomposition theorem for general holonomic $\cD$-modules was proven by Mochizuki \cite{Mochizuki:Decomp}. Based on Mochizuki's result, Romanov proved an algorithm for computing composition series of standard Whittaker modules with integral regular infinitesimal characters \cite{Romanov:Whittaker}. Along the way, she developed a character theory for $\cN_{\theta,\eta}$ and computed characters of standard modules. 

Despite the success of the geometric methods, the results of Mili\v ci\'c-Soergel and Romanov were not extended beyond integral regular infinitesimal characters. This paper fulfills the gap. Namely, we prove a Kazhdan-Lusztig algorithm for Whittaker modules for regular infinitesimal characters, and deduce from the algorithm a character formula for irreducible Whittaker modules for arbitrary infinitesimal characters. In particular, we recover the non-integral Kazhdan-Lusztig conjecture for Verma modules.

\subsection{The character formula}\label{subsec:main_results}

To state the character formula, let us introduce more notations.  We choose a Cartan subalgebra $\fh$ of $\fg$ normalizing $\fn$ (so that $\fb = \fh + \fn$ is a Borel subalgebra). Let $G$ be a connected algebraic group over $\BC$ with Lie algebra $\fg$ and write $N$ for the unipotent subgroup with Lie algebra $\fn$. We write $\Sigma \supset \Sigma^+ \supset \Pi$ and $W$ for the set of roots of $(\fg,\fh)$, the set of positive roots determined by $\fn$, the set of simple roots, and the Weyl group of $\Sigma$, respectively. We set
\begin{equation}\label{eqn:defn_of_Theta_intro}
	\Theta = \{ \alpha \in \Pi \mid \eta \text{ is nonzero on the $\alpha$-root space in } \fn \}
\end{equation}
and let $\lambda \in \fh^*$. We use a subscript $\Theta$ (resp. $\lambda$) on objects to denote corresponding subobjects that are defined by $\Theta$ (resp. integral to $\lambda$). So $\Sigma_\Theta$ is the subsystem of $\Sigma$ generated by $\Theta$; the parabolic subgroup $W_\Theta$ is the Weyl group of $\Sigma_\Theta$, embedded as a subgroup of $W$; the integral root system $\Sigma_\lambda$ consists of those roots $\alpha \in \Sigma$ so that $\alpha^\vee(\lambda) \in \BZ$, where $\alpha^\vee$ is the coroot of $\alpha$; the set of positive integral roots $\Sigma_\lambda^+$ is defined to be $\Sigma_\lambda \cap \Sigma^+$ and $\Pi_\lambda \subseteq \Sigma_\lambda^+$ is the corresponding set of simple roots (which may not be simple in $\Sigma^+$); the integral Weyl group $W_\lambda$ is the Weyl group of $\Sigma_\lambda$, which can be embedded in $W$ as $\{w \in W \mid w \lambda - \lambda \in \BZ \cdot \Sigma \}$. We write $\theta$ for the Weyl group orbit of $\lambda$.

Let us fix a $\lambda$ that is antidominant regular with respect to $\Sigma^+$. This means $\alpha^\vee(\lambda)$ is not a non-negative integer for all $\alpha \in \Sigma^+$. The Grothendieck group of the category $\cN_{\theta,\eta}$ has two natural basis: one is given by McDowell's standard Whittaker modules $M(w^C \lambda, \eta)$, and another is given by irreducible quotients $L(w^C \lambda,\eta)$ of the standard modules, both labeled by right $W_\Theta$-cosets $C$ in $W$. Here, $w^C$ denotes the unique longest element in $C$ under the Bruhat order. Romanov defined a character map $\ch$ on objects of $\cN_{\theta,\eta}$ that factors through the Grothendieck group \cite[\textsection 2.2]{Romanov:Whittaker}. We aim to express the character of $L(w^C \lambda,\eta)$ in terms of the characters of standard modules $M(w^D \lambda,\eta)$ (the latter were computed by Romanov). These facts about Whittaker modules will be recalled in \textsection\ref{subsec:Wh_prelim}. 

The precise expression of the character involves combinatorial data extracted from double cosets $W_\Theta \backslash W / W_\lambda$. Each double coset $W_\Theta u W_\lambda$ contains a unique shortest element $u$ with respect to Bruhat order (Corollary \ref{thm:cross-section_db_coset}). We can then take the intersections of $u W_\lambda$ with various right $W_\Theta$-cosets in $W_\Theta u W_\lambda$. This produces a partition of $u W_\lambda$. Left-translating back into $W_\lambda$, we obtain a partition of $W_\lambda$, which coincides with the partition given by right $W_{\lambda,\Theta(u,\lambda)}$-cosets of $W_\lambda$ (Proposition \ref{lem:int_Whittaker_model}). Here, $W_{\lambda,\Theta(u,\lambda)}$ is a parabolic subgroup of $W_\lambda$ corresponding to the subset of simple roots $\Theta(u,\lambda) = u\inv \Sigma_\Theta \cap \Pi_\lambda \subseteq \Pi_\lambda$. We thus obtain a map from the set of right $W_\Theta$-cosets in $W_\Theta u W_\lambda$ to the set of right $W_{\lambda,\Theta(u,\lambda)}$-cosets in $W_\lambda$, i.e. a map
\begin{equation}\label{eqn:defn_of_(-)|_lambda_intro}
	(-)|_\lambda: W_\Theta \backslash W_\Theta u W_\lambda \to W_{\lambda,\Theta(u,\lambda)} \backslash W_\lambda
\end{equation}
(Notation \ref{not:right_coset_partition}). Recall that there is a partial order $\le$ on $W_\Theta \backslash W$ inherited from the restriction of Bruhat order to the set of the longest element in each coset (see \textsection\ref{subsec:WTheta_prelim}). We denote the partial order on $W_{\lambda,\Theta(u,\lambda)} \backslash W_\lambda$ by $\le_{u,\lambda}$. 

The double cosets reflect the block decomposition of $\cN_{\theta,\eta}$ (here a ``block'' means an indecomposable direct summand of $\cN_{\theta,\eta}$). On the level of character formula, $\ch M(w^D \lambda,\eta)$ ($\ch$ denotes the character map) appears in $\ch L(w^C \lambda,\eta)$ only if $D$ and $C$ are in the same double coset $W_\Theta u W_\lambda$ and $D|_\lambda \le_{u,\lambda} C|_\lambda$ (for which we will simply write $D \le_{u,\lambda} C$). The precise coefficient of $\ch M(w^D \lambda,\eta)$ is described by Whittaker Kazhdan-Lusztig polynomials. For a triple $(W,\Pi,\Theta)$, Whittaker Kazhdan-Lusztig polynomials are polynomials $P_{CD} \in \BZ[q]$ labeled by pairs $(C,D)$ of right $W_\Theta$-cosets with $D \le C$ (Definition/Theorem \ref{def:parabolic_KL_poly_Theta}). These polynomials compute (at $q = -1$) the character formula of irreducible Whittaker modules for integral infinitesimal characters. Applied to the triple $(W_\lambda, \Pi_\lambda, \Theta(u,\lambda))$ and the pair $(C|_\lambda,D|_\lambda)$, we obtain polynomials $P_{CD}^{u,\lambda} = P_{C|_\lambda, D|_\lambda}^{u,\lambda}$ (Definition/Theorem \ref{def:parabolic_KL_poly}, or see \ref{enum:WKL_basis_expansion} and \ref{enum:WKL_basis_U} in \textsection\ref{subsec:geom_idea}).

\begin{theorem}[Character formula: regular case]\label{thm:multiplicity_intro}
	Let $\lambda$ be antidominant regular. For any $C \in W_\Theta \backslash W$, let $W_\Theta u W_\lambda$ be the double coset containing $C$, where $u$ is the unique shortest element in this double coset. Then
	\begin{equation*}
		\ch L(w^C\lambda,\eta) = \ch M(w^C\lambda,\eta) +
		\sum_{\substack{D \in W_\Theta \backslash W_\Theta u W_\lambda\\%
				D <_{u,\lambda} C}} 
		P_{CD}^{u,\lambda}(-1) \ch M(w^D \lambda,\eta),
	\end{equation*}
\end{theorem}

This appears as Theorem \ref{thm:multiplicity} below. We also extend this to singular $\lambda$ in Theorem \ref{thm:multiplicity_singular}.  At the special case $\eta = 0$, we recover the non-integral Kazhdan-Lusztig conjecture for Verma modules (Corollary \ref{thm:Verma_multiplicity}). 

The above formula follows from an algorithm (so called Kazhdan-Lusztig algorithm), namely Theorem \ref{thm:KL_alg}. The proof of the algorithm is done by studying (weakly) equivariant $\cD$-modules.
We postpone the statement of the algorithm to the next subsection.

\subsection{The Kazhdan-Lusztig algorithms and an outline of proof}\label{subsec:geom_idea}

Before discussing the extension to non-integral infinitesimal characters, let us first discuss Romanov's work in the integral case. Her argument is in the same spirit as the algorithm for highest weight modules which we now recall. For the backgrounds of localization theory relevant to Whittaker modules, we refer the readers to \textsection\ref{subsec:geom_prelim}.

For a fixed integral regular infinitesimal character $\chi_\theta$, the Grothendieck group $K\cN_{\theta,0}$ of the highest weight category $\cN_{\theta,0}$ is free abelian with two natural bases labeled by $W$: one is given by irreducible objects $L_w$, and the other given by Verma modules $M_w$. The original Kazhdan-Lusztig conjecture is a description of the change of basis matrix. The strategy of the proof of the conjecture is to study $\cD$-modules on the flag variety $X$ corresponding to highest weight modules. To say more detail, recall that any infinitesimal character $\chi_\theta: \cZ(\fg) \to \BC$ is determined by a Weyl group orbit $\theta$ in $\fh^*$ (we recall the definition of $\chi_\theta$ in \textsection\ref{sec:prelim}). Let $\lambda \in \theta$ be an antidominant element with respect to roots in $\Sigma^+$. Beilinson-Bernstein's localization theory \cite{Beilinson-Bernstein:Localization} gives us an equivalence of categories
\begin{equation*}
	\cN_{\theta,0} \cong \Mod_{coh}(\cD_\lambda,N).
\end{equation*}
Here, $\cD_\lambda$ is a certain twisted sheaf of differential operators on the flag variety $X$ of $\fg$, and the category $\Mod_{coh}(\cD_\lambda,N)$ is the category of $N$-equivariant coherent $\cD_\lambda$-modules on $X$. The images $\cL_w$ and $\cM_w$ of $L_w$ and $M_w$ have natural geometric meanings. We call $\cM_w$ the \textit{costandard modules}. If we write $[\cL_w]$ and $[\cM_w]$ for their classes in the Grothendieck group $K \Mod_{coh}(\cD_\lambda,N)$, the Kazhdan-Lusztig conjecture is now the problem of expressing $[\cL_w]$ as a linear combination of $[\cM_v]$ in the Grothendieck group.

To relate this problem with the combinatorics in \cite{Kazhdan-Lusztig:Hecke_Alg}, one aims to build a comparision map $\nu'$ that fits into the commutative diagram
\begin{equation*}
	\begin{tikzcd}
		\Mod_{coh}(\cD_\lambda,N) \ar[d, "{[-]}"'] \ar[r, "\nu'"]
		& \cH \ar[d, "q=-1"]\\
		K\Mod_{coh}(\cD_\lambda,N) \ar[r, "\cong"] 
		& \BZ[W]
	\end{tikzcd}.
\end{equation*}
In this diagram, the objects $\cH$ and $\BZ[W]$ are the Hecke algebra and the group algebra of $W$, respectively, and the bottom map sends $[\cM_w]$ to the basis in $\BZ[W]$ labeled by $w$. Moreover, the right regular action of $\cH$ on the top right corner should lift to an $\cH$ ``action'' on the $\cM_w$ and $\cL_w$ in $\Mod_{coh}(\cD_\lambda,N)$. Once this diagram is constructed, then $[\cL_w] = \nu'(\cL_w)|_{q=-1}$ by commutativity of the diagram, and $\nu'(\cL_w)$ can be computed by studying the $\cH$-action.

In further detail, recall that the Hecke algebra $\cH$ has an underlying free $\BZ[q^{\pm1}]$-module structure with two bases labeled by $W$: the defining basis $\{\delta_w\}$ and the Kazhdan-Lusztig basis $\{C_w\}$ \cite{Kazhdan-Lusztig:Hecke_Alg}. The Kazhdan-Lusztig basis is characterized by three conditions: 
\begin{enumerate}[label=(KL.\arabic*)]
	\item the expansion of $C_w$ in terms of the $\delta_v$ basis elements involve only those with $v \le w$, the coefficient of $\delta_w$ is $1$, and the coefficient of $\delta_v$ for $v < w$ is a polynomial $P_{wv}(q)$ with no constant term; \label{enum:KL_basis_expansion}
	\item the product $C_w C_s$, where $s$ is a simple reflection so that $ws > w$, is a $\BZ$-linear combination of $C_v$ with $v \le ws$;	\label{enum:KL_basis_U}
	\item $C_s = \delta_s + q$
\end{enumerate}
(after some normalizations, the first two conditions are (1.1.b) and (2.3.b) of \cite{Kazhdan-Lusztig:Hecke_Alg}, respectively). Here $<$ and $\le$ are the Bruhat order on $W$. These conditions inductively determine the Kazhdan-Lusztig basis and provide a recursive algorithm for computing it. The coefficients $P_{wv}$ of $\delta_v$ are the famous \textit{Kazhdan-Lusztig polynomials}. The Kazhdan-Lusztig conjecture predicts that the coefficients of the Verma modules in the irreducible modules in the Grothendieck group are given by Kazhdan-Lusztig polynomials evaluated at $-1$ (or at $1$, depending on the normalization). In view of the above diagram, proving the conjecture amounts to constructing $\nu'$ so that $\nu'(\cM_w) = \delta_w$ and $\nu'(\cL_w) = C_w$. 

To this end, we define the map $\nu'$ by sending a $\cD_\lambda$-module $\cF$ to a linear combination of $\delta_v$ where the coefficient of $\delta_v$ is the generating function (in variable $q$) of the pullback of $\BD \cF$ to the Schubert cell $C(v)$:
\begin{equation*}
	\nu'(\cF) = \sum_{w \in W} \big( \chi_q i_w^! \BD \cF \big) \delta_w
\end{equation*}
(the map $\chi_q$ is defined in (\ref{eqn:defn_of_chi_q})). Here $\BD$ is the duality functor of holonomic $\cD$-modules. With this definition, the map $\nu'$ sends $\cM_v$ to $\delta_v$, and $\nu'(\cL_w)$ automatically satisfies condition \ref{enum:KL_basis_expansion} for support reason. Moreover, multiplication by $C_s$ on $\delta_w$ for a simple reflection $s$ lifts on $\cM_w$ to the ``push-pull'' operation along the natural map $X \to X_s$ to the type-$s$ partial flag variety (we call this operation the \textit{$U$-functor} since it agrees with the functor $U$ defined by Vogan \cite[Definition 3.8]{Vogan:irred_char_I}). The condition \ref{enum:KL_basis_U} is proven by an induction on $\ell(w)$ by showing the same lifting for irreducibles, using the Decomposition Theorem of Beilinson-Bernstein-Deligne \cite{Beilinson-Bernstein-Deligne:Decomposition} for regular holonomic $\cD$-modules (or perverse sheaves). This proves $\nu'(\cL_w) = C_w$ and hence the Kazhdan-Lusztig conjecture. A detailed argument following these lines can be found in Mili\v ci\'c's unpublished notes \cite[Chapter 5]{Milicic:Localization}. Since the character map on highest weight modules factors through the Grothendieck group, one can write down characters of irreducible modules in terms of characters of Verma modules, and the latter can be easily computed.  

The argument we just described naturally extends to parabolic highest weight categories corresponding to a subset $\Theta$ of simple roots and with regular integral infinitesimal characters. Two bases of the Grothendieck group are now given by parabolic Verma modules and their irreducible quotients, both labeled by right $W_\Theta$-cosets. The map $\nu'$ is now defined by pulling back to orbits of a parabolic subgroup $P_\Theta$ of type $\Theta$, and the image of the comparison map $\nu'$ is now replaced by a smaller $\cH$-module. The Kazhdan-Lusztig polynomials are then replaced by \textit{parabolic Kazhdan-Lusztig polynomials}, which form a subset of the ordinary Kazhdan-Lusztig polynomials.

In the case of Whittaker modules with integral regular infinitesimal characters, we still have two bases of the Grothendieck group labeled by right $W_\Theta$-cosets: the standard Whittaker modules defined by McDowell and their irreducible quotients. Localizations of Whittaker modules now land into the category $\Mod_{coh}(\cD_\lambda,N,\eta)$ of \textit{twisted Harish-Chandra sheaves} (defined in \textsection\ref{subsec:geom_prelim}). By the work of Mili\v ci\'c-Soergel \cite{Milicic-Soergel:Whittaker_algebraic}, the category of Whittaker modules is equivalent to the highest weight category with a singular infinitesimal character. The latter is known to be Koszul dual to a parabolic highest weight category with an integral regular infinitesimal character by the work of Beilinson-Ginzburg-Soergel \cite{Beilinson-Ginzburg-Soergel:Koszul_duality}. Therefore, the Kazhdan-Lusztig polynomials of Whittaker modules (what Romanov called \textit{Whittaker Kazhdan-Lusztig polynomials}) are expected to be dual to parabolic Kazhdan-Lusztig polynomials. More precisely, if we define $\Theta$ as in (\ref{eqn:defn_of_Theta_intro}), then the Whittaker category $\cN_{\theta,\eta}$ is expected to be dual to the parabolic highest weight category determined by $\Theta$. A starting point towards proving this would be a Kazhdan-Lusztig algorithm Whittaker modules. However, the $\cD$-modules in this situation are no longer regular holonomic (merely holonomic). Therefore a decomposition theorem for general holonomic modules is needed in order for the same argument to work. This is proven by Mochizuki \cite{Mochizuki:Decomp}. Romanov then adapted the strategy for highest weight modules to the case of Whittaker modules in her thesis (later published as \cite{Romanov:Whittaker}) and obtained a Kazhdan-Lusztig algorithm. Together with the character theory she developed, her work implies a character formula for irreducible Whittaker modules. The comparison map $\nu'$ in the highest weight setting now becomes a map
\begin{equation*}
	\Mod_{coh}(\cD_\lambda,N,\eta) \xrightarrow{\nu'} \cH_\Theta
\end{equation*}
defined by pulling back $\BD\cF$ to Schubert cells of the form $C(w^C)$, where $\cH_\Theta$ is an $\cH$-module whose underlying $\BZ[q^{\pm1}]$-module structure is free with basis elements labeled by $W_\Theta \backslash W$. This $\cH$-module structure defines a Kazhdan-Lusztig basis of $\cH_\Theta$, whose elements coincide with the images of irreducible $\cD$-modules under $\nu'$. 

The work of this paper generalizes Romanov's algorithm to arbitrary infinitesimal characters. There are two extra complications compared to Romanov's situation. First, although (co)standard and irreducible Whittaker modules are still parameterized by $W_\Theta \backslash W$, now our category is a direct sum of smaller blocks, and different blocks have different sizes. On the other hand, the parabolic highest weight category can have fewer blocks, so the duality mentioned in the preceding paragraph fails. Nevertheless, one can expect the blocks to be parameterized by Weyl group data involving both $W_\Theta$ and $W_\lambda$. Indeed, it turns out that blocks are parameterized by double cosets $W_\Theta \backslash W / W_\lambda$, and the polynomials for each block turn out to be the same as (integral) Whittaker Kazhdan-Lusztig polynomials related to the integral Weyl group $W_\lambda$.

The second complication is that the ``push-pull'' operation along $X \to X_s$ does not exist when $\lambda$ is non-integral to $s$ \--- there is no sheaf of twisted differential operators on $X_s$ that pulls back to $\cD_\lambda$. As a result, induction on $\ell(w)$ cannot proceed as before. To remedy this, we use the \textit{intertwining functor} $I_s$ (defined in \textsection\ref{sec:geom}) for non-integral $s$ in place of the $U$-functor. It is an equivalence of categories between $\cD_\lambda$-modules and $\cD_{s\lambda}$-modules. This allows us to increase $\ell(w)$ and retain induction hypotheses. This idea of proof is suggested to the author by Mili\v ci\'c. 

We can now state our algorithm. We refer to \textsection\ref{subsec:KL_poly} for the precise definitions of the Hecke-theoretic objects appearing below. We fix a character $\eta: \fn \to \BC$ and define a subset $\Theta$ of simple roots from $\eta$ as in (\ref{eqn:defn_of_Theta_intro}). For each $\lambda$ (not necessarily antidominant) and each $C \in W_\Theta \backslash W$, we write $\cM(w^C,\lambda,\eta)$ for the costandard $\cD$-module and $\cL(w^C,\lambda,\eta)$ for its irreducible quotient (these are defined in \textsection \ref{subsec:geom_prelim}). In the case where $\lambda$ is antidominant regular, they are localizations of the standard Whittaker module $M(w^C \lambda,\eta)$ and the irreducible Whittaker module $L(w^C \lambda,\eta)$, respectively. We let  $\cH_\Theta$ be the free $\BZ[q^{\pm1}]$-module with basis $\{\delta_C\}_{C \in W_\Theta \backslash W}$ and define a map $\nu'$ similar to the highest weight case
\begin{equation*}
	\nu': \Mod(\cD_\lambda,N,\eta) \to \cH_\Theta,\quad
	\cF \mapsto \sum_{C \in W_\Theta \backslash W} \big( \chi_q i_{w^C}^! \BD \cF \big) \delta_C
\end{equation*}
($\chi_q$ is defined in (\ref{eqn:defn_of_chi_q}); in the body of the paper we instead work with $\nu = \nu' \circ \BD$ (\ref{eqn:defn_of_nu}) instead of $\nu'$ for technical simplicities). It fits into the commutative diagram
\begin{equation*}
	\begin{tikzcd}
		\Mod_{coh}(\cD_\lambda,N,\eta) \ar[d, "{[-]}"'] \ar[r, "\nu'"]
		& \cH_\Theta \ar[d, "q=-1"] \ar[r, "(-)|_\lambda"] 
		& {\displaystyle \bigoplus_{W_\Theta u W_\lambda} \cH_{\Theta(u,\lambda)} } \ar[d, "q=-1"]\\
		K\Mod_{coh}(\cD_\lambda,N,\eta) \ar[r, "\cong"] 
		& \BZ[W_\Theta \backslash W] \ar[r, "(-)|_\lambda"] 
		& {\displaystyle \bigoplus_{W_\Theta u W_\lambda} \BZ[W_{\lambda,\Theta(u,\lambda)} \backslash W_\lambda] }
	\end{tikzcd}.
\end{equation*}
Here $\BZ[W_\Theta \backslash W]$ is the $\BZ$-module with basis $\{\delta_C\}_{C \in W_\Theta \backslash W}$, and the first horizontal map at the bottom sends $[\cM(w^C,\lambda,\eta)]$ to $\delta_C$. The modules $\cH_{\Theta(u,\lambda)}$ and $\BZ[W_{\lambda,\Theta(u,\lambda)} \backslash W_\lambda]$ are defined similarly but their basis elements are instead labeled by $W_{\lambda,\Theta(u,\lambda)} \backslash W_\lambda$. The map $(-)|_\lambda$ is defined on basis elements analogous to (\ref{eqn:defn_of_(-)|_lambda_intro}). Each $\cH_{\Theta(u,\lambda)}$ is a module over the Hecke algebra $\cH_\lambda = \cH(W_\lambda)$ of the integral Weyl group $W_\lambda$ (as in Romanov's work in the integral case). Thus each $\alpha \in \Pi_\lambda$ defines an operator $T_\alpha^{u,\lambda}$ on $\cH_{\Theta(u,\lambda)}$ representing the multiplication of the Kazhdan-Lusztig basis element $C_{\lambda,s_\alpha} \in \cH_\lambda$ corresponding to the simple reflection $s_\alpha$. Romanov's main result \cite[Theorem 11]{Romanov:Whittaker}, interpreted combinatorially and applied to $\cH_{\Theta(u,\lambda)}$, says that the operators $T_\alpha^{u,\lambda}$ inductively define a Kazhdan-Lusztig basis of $\cH_{\Theta(u,\lambda)}$ in a similar fashion as the conditions \ref{enum:KL_basis_expansion} and \ref{enum:KL_basis_U}. More precisely, the Kazhdan-Lusztig basis $\{\psi_{u,\lambda}(F)\}_{F \in W_{\lambda,\Theta(u,\lambda)} \backslash W_\lambda}$ of $\cH_{\Theta(u,\lambda)}$ is the unique basis over $\BZ[q^{\pm1}]$ such that 
\begin{enumerate}[label=(W.\arabic*)]
	\item \label{enum:WKL_basis_expansion}
	$\psi_{u,\lambda}(F) = \delta_F + \sum_{G <_{u,\lambda} F} P_{FG}^{u,\lambda} \; \delta_G$ for some $P_{FG}^{u,\lambda} \in q \BZ[q]$; and 
	
	\item \label{enum:WKL_basis_U}
	if $F$ is not the shortest right coset, there exist $\alpha \in \Pi_\lambda$ and $c_G \in \BZ$ such that $F s_\alpha <_{u,\lambda} F$ and 
	\begin{equation*}
		T_\alpha^{u,\lambda}(\psi_{u,\lambda}(F s_\alpha)) = \sum_{G \le_{u,\lambda} F} c_G \; \psi_{u,\lambda}(G)
	\end{equation*}
\end{enumerate}
(Definition/Theorem \ref{def:parabolic_KL_poly}). We can still formally consider an $\cH$-module structure on $\cH_\Theta$ as in the integral case and define operators $T_\alpha: \cH_\Theta \to \cH_\Theta$ for simple roots $\alpha$. When $\alpha$ is integral to $\lambda$, $T_\alpha$ is the combinatorial incarnation of the $U$-functor. It preserves the decomposition $(-)|_\lambda$ and restricts to $T_\alpha^{u,\lambda}$ on each $\cH_{\Theta(u,\lambda)}$. When a simple root $\beta$ is non-integral, we will instead consider the endomorphism $(-) \cdot s_\beta$ on $\cH_\Theta$ given by $\delta_C \cdot s_\beta = \delta_{C s_\beta}$, which represents the action of the intertwining functor $I_{s_\beta}$.

Here is our (slightly rephrased) algorithm.

\begin{theorem}[Kazhdan-Lusztig algorithm]
	Fix a character $\eta: \fn \to \BC$. For any $\lambda$ and any $C \in W_\Theta \backslash W$, write $W_\Theta u W_\lambda$ for the double coset containing $C$, where $u$ is the unique shortest element in this double coset. Then
	\begin{enumerate}[label=(A.\arabic*)]
		\item \label{enum:WKL_expansion}
		There exist polynomials $P_{CD}^{u,\lambda} \in q\BZ[q]$ so that 
		\begin{equation*}
			\nu'(\cL(w^C,\lambda,\eta)) = \nu'(\cM(w^C,\lambda,\eta))
			+ \sum_{\substack{D \in W_\Theta \backslash W_\Theta u W_\lambda\\%
					D <_{u,\lambda} C}} 
			P_{CD}^{u,\lambda} \; \nu' (\cM(w^D, \lambda,\eta)).
		\end{equation*}
		
		\item \label{enum:WKL_U}
		For any integral simple root $\alpha$ such that $C s_\alpha < C$, there exist integers $c_D$ depending on $C$, $D$, and $s_\alpha$, such that
		\begin{equation*}
			T_\alpha(\nu'(\cL(w^{Cs_\alpha},\lambda,\eta))) = 
			\sum_{\substack{D \in W_\Theta \backslash W_\Theta u W_\lambda \\ D \le_{u,\lambda} C}}
			c_D \;\nu'(\cL(w^D,\lambda,\eta)).
		\end{equation*}
		
		\item \label{enum:WKL_I}
		For any non-integral simple root $\beta$ such that $C s_\beta < C$, 
		\begin{equation*}
			\nu'(\cL(w^C,\lambda,\eta)) \cdot s_\beta = \nu'(\cL(w^{C s_\beta}, s_\beta \lambda, \eta)).
		\end{equation*}
		
		\item \label{enum:WKL_integral_model} 
		$\nu'(\cL(w^C,\lambda,\eta))|_\lambda$ is a Kazhdan-Lusztig basis element of $\cH_{\Theta(u,\lambda)}$.
	\end{enumerate}
\end{theorem}

This appears as Theorem \ref{thm:KL_alg} below.  The character formula \ref{thm:multiplicity_intro} follows by taking \ref{enum:WKL_expansion} and \ref{enum:WKL_integral_model} for antidominant regular $\lambda$, descending to the Grothendieck group by specializing at $q=-1$, passing through Beilinson-Bernstein localization, and applying the character map. 

The proof of the algorithm is an induction on the length $\ell(w^C)$. The proofs of \ref{enum:WKL_expansion} and \ref{enum:WKL_U} are analogous to the proofs of \ref{enum:KL_basis_expansion} and \ref{enum:KL_basis_U}, respectively. \ref{enum:WKL_I} reflects the action of non-integral intertwining functor $I_{s_\beta}$. In fact, the following diagram commutes
\begin{equation}
	\begin{tikzcd}
		\Mod_{coh}(\cD_\lambda,N,\eta) \ar[d, "I_{s_\beta}"'] \ar[r, "\nu'"]
		& \cH_\Theta \ar[d, "(-) \cdot s_\beta"'] \ar[r, "(-)|_\lambda"] 
		& {\displaystyle \bigoplus \cH_{\Theta(u,\lambda)} } \ar[d, "s_\beta \cdot (-) \cdot s_\beta"]\\
		\Mod_{coh}(\cD_{s_\beta\lambda},N,\eta)  \ar[r, "\nu'"]
		& \cH_\Theta  \ar[r, "(-)|_\lambda"] 
		& {\displaystyle \bigoplus \cH_{\Theta(r,s_\beta \lambda)} } 
	\end{tikzcd}
\end{equation}
(Proposition \ref{lem:Is_right_coset}, Corollary \ref{lem:I_on_std}, and Proposition \ref{thm:Is_pullback}; we only prove the commutativity of this diagram for irreducible $\cD$-modules, but extension to other modules is straightforward). The push-pull operation together with non-integral intertwining functors allows the induction argument to run. In the actual proof, one prove \ref{enum:WKL_U} and \ref{enum:WKL_I} first at each inductive step and use them two prove the remaining statements.

The remaining technical difficulty lies in the proof of \ref{enum:WKL_integral_model}. It requires us to find $\alpha \in \Pi_\lambda$ so that $C s_\alpha <_{u,\lambda} C$ and \ref{enum:WKL_basis_U} holds. If $\alpha$ can be chosen to be also simple in $\Sigma^+$, then \ref{enum:WKL_basis_U} simply follows from \ref{enum:WKL_U}. But there are examples where this cannot be done. The strategy then is to apply non-integral intertwining functors so that $\alpha$ becomes simple in both the integral Weyl group and in $W$, and that $C$ is translated to a coset of smaller length so that \ref{enum:WKL_U} holds by induction assumption. \ref{enum:WKL_basis_U} is obtained by translating \ref{enum:WKL_U} back via inverse intertwining functors. The existence of such a chain of intertwining functors is guaranteed by Lemma \ref{lem:decrease_of_length}.

In the special case $\eta = 0$, our argument gives a new proof of the non-integral Kazhdan-Lusztig conjecture for Verma modules. The non-integral intertwining functors we use do not seem to have analogue in existing approaches using perverse sheaves (for example Lusztig's proof \cite[Chapter 1]{Lusztig:Char_finite_field}) and Soergel modules (\cite{Soergel:V}). Since these functors are equivalences on the category of quasi-coherent $\cD_\lambda$-modules, we believe they should have applications outside of the current context. We will expand on this in Remark \ref{rmk:Verma_case_old_proofs}.


\subsection{Outline of the paper}

The paper is organized as follows. In \textsection \ref{sec:prelim} we present preliminaries on Whittaker modules and their localizations. The following section \textsection\ref{sec:db_coset} is devoted to studying the structure of left $W_\lambda$-cosets and double $(W_\Theta,W_\lambda)$-cosets in the Weyl group. In \textsection \ref{sec:geom} we study the effect of non-integral intertwining functors on irreducible $\cD$-modules. In \textsection \ref{sec:KL} we state and prove the main algorithm. The character formula is established in \textsection \ref{sec:character_formula}. Lastly, in \textsection \ref{sec:examples}, we provide an example on the $A_3$ root system.

\subsection{Acknowledgements}

I would like to thank Dragan Mili\v ci\'c for suggesting on the algorithm, as well as for his mentorship, guidance, and many helpful discussions. I thank Anna Romanov for some enlightening conversations. I thank the referee for their valuable suggestions and for pointing out a mistake in \textsection\ref{sec:examples}. The results in this paper are based on the results of the author's Ph.D. thesis \cite{Zhao:thesis} at the University of Utah.
\section{Preliminarlies}\label{sec:prelim}

In this section we fix some notations and present necessary facts on Whittaker modules and their localizations without proof. 

Let us start by recalling some notations. We have fixed in \textsection\ref{subsec:main_results} a complex semisimple Lie algebra $\fg$ over $\BC$, a maximal nilpotent subalgebra $\fn$, and a Cartan subalgebra $\fh$ normalizing $\fn$. The capital letters $G$, $N$, and $H$ denote the corresponding algebraic groups. The sets $\Sigma \supset \Sigma^+ \supseteq \Pi \supseteq \Theta$ denote the root system of $(\fg,\fh)$, the set of positive roots as roots in $\fn$, and the set of simple roots, respectively. We let $\eta: \fn \to \BC$ be a Lie algebra character and let
\begin{equation*}
	\Theta = \{ \alpha \in \Pi \mid \eta \text{ is nonzero on the $\alpha$-root space in } \fn \}.
\end{equation*}
The half sum of positive roots is denoted by $\rho$. The Weyl group is denoted by $W$. Let $\lambda \in \fh^*$ and $\theta$ is the $W$-orbit of $\lambda$. A subscript $\Theta$ or $\lambda$ denotes corresponding subobjects defined by $\Theta$ or are integral to $\lambda$. The capital letters $C,D,E,F$ will denote right $W_\Theta$-cosets in $W$ or in $W_\lambda$ (except $C(w)$ will denote a Schubert cell, and $C_w$, $C_D$, etc. will denote the Kazhdan-Lusztig basis elements in various Hecke algebra modules). We will write $w^C$ for the unique longest element in $C$.

Let $\cU(\fg)$ be the enveloping algebra of $\fg$ and write $\cZ(\fg)$ for the center of $\cU(\fg)$. We write $\xi: \cZ(\fg) \to \Sym(\fh)^W$ for the Harish-Chandra isomorphism (this is the map $\gamma \circ \varphi|_{\cZ(\fg)}$ in \cite[Theorem 7.4.5]{Dixmier:Enveloping_Alg}). The composition
\begin{equation*}
	\chi_\theta: \cZ(\fg) \xrightarrow{\xi} \Sym(\fh)^W \inj \Sym(\fh) \xrightarrow{\lambda} \BC
\end{equation*}
only depends on the Weyl group orbit $\theta$ of $\lambda$. We will refer to $\chi_\theta$ as an \textit{infinitesimal character}. We let $\cU(\fg)_\theta := \cU(\fg)/\langle \ker \chi_\theta \rangle$. When the Lie algebra is understood, we often write $\cU_\theta$ for $\cU(\fg)_\theta$. The weight $\lambda$ is said to be \textit{regular} if $\alpha^\vee(\lambda) \neq 0$ for all roots $\alpha$, \textit{antidominant} if $\alpha^\vee(\lambda)$ is not a non-negative integer for all $\alpha \in \Sigma^+$, and \textit{integral} if $\alpha^\vee(\lambda) \in \BZ$ for all $\alpha \in \Sigma$. We say $\chi_\theta$ is regular if $\lambda$ is.

\subsection{Preliminaries on Whittaker modules}\label{subsec:Wh_prelim}

The category of \textbf{Whittaker modules}, denoted by $\cN$, is the full subcategory of all $\fg$-modules consisting of those that are finitely generated over $\fg$, locally finite over $\fn$, and locally finite over $\cZ(\fg)$. Here we say a module over a $\BC$-algebra is locally finite if every element generates a finite dimensional subspace. We write $\cN_\theta$ (resp. $\cN_\eta$) for the full subcategory of $\cN$ consisting of objects with infinitesimal character $\chi_\theta$ (resp. on which $\xi - \eta(\xi)$ acts locally nilpotently for all $\xi \in \fn$). Set $\cN_{\theta,\eta} = \cN_\theta \cap \cN_\eta$. Every object of $\cN$ has finite length (\cite[Theorem 2.7(c)]{McDowell:Whittaker}, \cite[Theorem 2.6(1)]{Milicic-Soergel:Whittaker_algebraic}). By local finiteness over $\cZ(\fg)$ and $\cU(\fn)$, $\cN$ decomposes as a direct sum of various subcategories $\cN_{\theta,\eta}$. So each irreducible object land in a single $\cN_{\theta,\eta}$. 

To give more description of the category $\cN_{\theta,\eta}$, let us first consider the non-degenerate case. So assume $\eta$ is non-degenerate, i.e. $\Theta = \Pi$. Consider the cyclic module
\begin{equation*}
	Y_\fg(\lambda,\eta) := \cU(\fg)_\theta \dotimes_{\cU(\fn)} \BC_\eta,
\end{equation*}
i.e. a module generated by a single vector on which $\fn$ acts by $\eta$ and $\cZ(\fg)$ acts by $\chi_\theta$. Kostant showed that $Y_\fg(\lambda,\eta)$ is irreducible, and is the unique irreducible object in the semisimple category $\cN_{\theta,\eta}$ \cite[Theorem A]{Kostant:Whittaker} (see \cite[Theorem 5.6]{Milicic-Soergel:Whittaker_geometric} for a geometric proof). The category $\cN_\eta$ is equivalent to the category of finite dimensional $\cZ(\fg)$-modules \cite[Theorem 5.9]{Milicic-Soergel:Whittaker_geometric}. 

Now suppose $\eta$ is general. We can define $Y_\fg(\lambda,\eta)$ in the same way as above, but it will be potentially reducible. Instead, we look at the parabolic subalgebra $\fp_\Theta \supset \fh + \fn$ defined by $\Theta$. We take its $\ad \fh$-stable Levi decomposition $\fp_\Theta = \fl_\Theta + \fu_\Theta$, write $\fn_\Theta = \fl_\Theta \cap \fn$ (so that $\fn = \fn_\Theta + \fu_\Theta$), and let $\rho_\Theta$ for the half sum of $\fh$-roots in $\fn_\Theta$. Then the restriction $\eta|_{\fn_\Theta}$ is non-degenerate by construction, so the cyclic $\fl$-module
\begin{equation*}
	Y_\fl(\lambda-\rho+\rho_\Theta,\eta) = \cU(\fl)_{\lambda-\rho+\rho_\Theta} \dotimes_{\cU(\fn_\Theta)} \BC_\eta
\end{equation*}
is irreducible. The following definition is due to McDowell \cite[Proposition 2.4]{McDowell:Whittaker} (see also \cite[\textsection 2]{Milicic-Soergel:Whittaker_algebraic}; our notation is the closest to the one in \cite[Definition 2]{Romanov:Whittaker})

\begin{definition}\label{def:std_Wh_mods}
	The \textbf{standard Whittaker module} is the module parabolically induced from $Y_\fl(\lambda-\rho+\rho_\Theta,\eta)$:
	\begin{equation*}
		M(\lambda,\eta) = \cU(\fg) \dotimes_{\cU(\fp_\Theta)} Y_\fl(\lambda-\rho+\rho_\Theta,\eta).
	\end{equation*}
\end{definition}
When $\eta$ is non-degenerate, $M(\lambda,\eta) = Y_\fg(\lambda,\eta)$. When $\eta = 0$, these are just Verma modules. 

The standard Whittaker modules share similar properties with Verma modules. McDowell showed that each $M(\lambda,\eta)$ is in $\cN_{\theta,\eta}$ and admits a unique irreducible quotient $L(\lambda,\eta)$. Moreover, $M(\lambda,\eta) = M(\lambda',\eta)$ if and only if $W_\Theta \lambda = W_\Theta \lambda'$, and the same holds for irreducibles. These facts are contained in \cite{McDowell:Whittaker} Proposition 2.4, Theorem 2.5, and Theorem 2.9 and are reproved in \cite[\textsection 2]{Milicic-Soergel:Whittaker_algebraic}. In particular, if we fix an antidominant $\lambda$ and write $W^\lambda$ for the stabilizer of $\lambda$ in $W$, then standard objects and irreducible objects in $\cN_{\theta,\eta}$ are parameterized by double cosets $W_\Theta \backslash W / W^\lambda$ where $W_\Theta z W^\lambda$ corresponds to $M(z \lambda,\eta)$ and $L(z \lambda,\eta)$. If $\lambda$ is regular, then standards and irreducibles are parameterized by $W_\Theta \backslash W$. In accordance with the geometric setup in \textsection\ref{subsec:geom_prelim}, we will write $M(w^C \lambda,\eta)$ and $L(w^C \lambda,\eta)$ for the modules parameterized by $C \in W_\Theta \backslash W$, where $w^C$ is the unique longest element in $C$.

Using the standard modules, McDowell showed that each cyclic module $Y_\fg$ also lands inside a single $\cN_{\theta,\eta}$ \cite[Theorem 2.5]{McDowell:Whittaker}. Later Mili\v ci\'c-Soergel showed that in fact cyclic modules are filtered by standard modules (and is a direct sum if the infinitesimal character is regular) \cite[Corollary 2.5]{Milicic-Soergel:Whittaker_algebraic}.

By mimicking the construction for Verma modules, Romanov developed in \cite[\textsection 2.2]{Romanov:Whittaker} a character theory for $\cN_{\theta,\eta}$. She defines a map $\ch$ on objects of $\cN_{\theta,\eta}$ that factors through and is injective on the Grothendieck group $K \cN_{\theta,\eta}$. The characters of standard Whittaker modules are computed explicitly in \cite[Equation (2)]{Romanov:Whittaker}. In Romanov's paper, the character theory is mainly used to match global sections of costandard $\cD$-modules (which will be defined in \textsection\ref{subsec:geom_prelim}) with standard Whittaker modules. Although our main results are stated in terms of the character map, they are in fact statements of the Grothendieck group, and we will not use any other property of the character map. 

Nevertheless, let us briefly describe the shape of this character theory. Let $\fh^\Theta$ be the center of $\fl_\Theta$, let $\fs_\Theta = [\fl_\Theta,\fl_\Theta]$ be the semisimple part of $\fl_\Theta$, and let $\fh_\Theta = \fs_\Theta \cap \fh$ be a Cartan in $\fs_\Theta$, so that $\fh = \fh_\Theta \oplus \fh^\Theta$. Since $\eta$ is non-degenerate on $\fn_\Theta = \fs_\Theta \cap \fn$, the category $\cN(\fs_\Theta)_\eta$ of Whittaker modules of $\fs_\Theta$ with generalized $\fn_\Theta$-character $\eta$ is equivalent to the category of finite dimensional $\cZ(\fs_\Theta)$-modules. The Grothendieck group $K \cN(\fs_\Theta)_\eta$ is therefore free abelian with a basis given by maximal ideals in $\cZ(\fs_\Theta)$, which is in bijection with the set of $\fn_\Theta$-regular dominant integral weights of $\fh_\Theta$. For an object $U \in \cN(\fs_\Theta)_\eta$, we write $[U]$ for its class in $K \cN(\fs_\Theta)_\eta$.

Any object $V$ in $\cN_\eta$ is necessarily locally $\fh^\Theta$-finite. Hence $V$ can be decomposed into a direct sum of generalized $\fh^\Theta$-weight spaces $V^{{\mu}}$, $ \mu \in (\fh^\Theta)^*$. It can be shown that each one of these is an $\fs_\Theta$-module living in $\cN(\fs_\Theta)_\eta$ \cite[Theorem 4]{Romanov:Whittaker}. The character map is defined by
\begin{equation}
	\ch: \operatorname{Obj} \cN_{\theta,\eta} \aro K \cN(\fs_\Theta)_\eta \dotimes_\BZ \BZ[[(\fh^\Theta)^*]],\quad
	V \mapsto \sum_{  \mu \in (\fh^\Theta)^*} [V^{  \mu}|_{\fs_\Theta}] e^{  \mu},
\end{equation}
where $\BZ[[(\fh^\Theta)^*]]$ is the group of power series in $e^{  \mu}$, $  \mu \in (\fh^\Theta)^*$. The characters of standard modules are easily computed, and is a linear combination with partition functions as coefficients, similar to Verma modules. See \cite[Equation (2)]{Romanov:Whittaker} for details.

\subsection{Localization of Whittaker modules}\label{subsec:geom_prelim}

In this subsection we describe the localization framework related to Whittaker modules. References for facts below include \cite{Beilinson-Bernstein:Localization}, \cite{Jantzen_Conj}, \cite{Milicic-Soergel:Whittaker_geometric}, \cite{Milicic:Localization}, \cite{Romanov:Whittaker}.

Let $X$ be the flag variety of $\fg$, the variety of Borel subalgebras of $\fg$, with the natural $G$-action. The sheaf of ordinary (algebraic) differential operators $\cD_X$ is the subsheaf of $\cHom_\BC(\cO_X,\cO_X)$ generated by multiplications of functions and actions of vector fields. The natural action of $G$ on the space of functions on $X$ can be differentiated, which assigns each element in $\fg$ a vector field on $X$, whence a map $\fg \to \cD_X$.

More generally, for each $\lambda \in \fh^*$, Beilinson-Bernstein constructed in \cite{Beilinson-Bernstein:Localization} a twisted sheaf of differential operators $\cD_\lambda$ on $X$ together with a map $\fg \to \cD_\lambda$ that induces an isomorphism $\cU_\theta \cong \Gamma(X,\cD_\lambda)$ (recall that $\theta$ is the Weyl group orbit of $\lambda$). Here $\cD_\lambda$ is a sheaf of $\BC$-algebras that is locally isomorphic to $\cD_X$. We use the parametrization of these sheaves as in \cite[Chapter 2 \textsection 1]{Milicic:Localization}, under which $\cD_X = \cD_{-\rho}$. If $\lambda$ is antidominant and regular, Beilinson and Bernstein showed that taking global sections on $X$ is an equivalence of categories
\begin{equation}
	\Gamma(X,-): \Mod_{qc}(\cD_\lambda) \cong \Mod(\cU_\theta)
\end{equation}
between the category of quasi-coherent $\cD_\lambda$-modules and the category of $\cU_\theta$-modules, and a quasi-inverse is given by the \textit{localization} functor $\cD_\lambda \dotimes_{\cU_\theta} -$ \cite{Beilinson-Bernstein:Localization}. If $\lambda$ is only antidominant but not regular, $\Gamma(X,-)$ is still exact, but some $\cD_\lambda$-modules can have zero global section, and $\Gamma(X,-)$ factors through an equivalence between $\Mod(\cU_\theta)$ and a quotient of $\Mod_{qc}(\cD_\lambda)$. The subcategory $\cN_{\theta,\eta}$ of $\Mod(\cU_\theta)$ corresponds, under the above equivalence of categories, to the subcategory $\Mod_{coh}(\cD_\lambda,N,\eta)$ consisting of \textbf{$\eta$-twisted Harish-Chandra sheaves} (or $\eta$-twisted sheaves for short). This is the full subcategory of all coherent $\cD_\lambda$-modules consisting of those $\cV$ such that 
\begin{itemize}
	\item $\cV$ is an $N$-equivariant $\cO_X$-module,
	\item the action map $\cD_\lambda \dotimes \cV \to \cV$ of $\cD_\lambda$ on $\cV$ is $N$-equivariant, and
	\item for all $n \in \fn$, the equation $\pi(\xi) = \mu(\xi) + \eta(\xi)$ holds in $\End_\BC(\cV)$, where $\pi$ is the action of $\fn$ induced by $\fn \subset \fg \to \cD_\lambda \acts \cV$, and $\mu$ is the action given by the differential of the $N$-equivariant structure on $\cV$.
\end{itemize}

Any $\eta$-twisted Harish-Chandra sheaf is automatically holonomic \cite[Lemma 1.1]{Milicic-Soergel:Whittaker_geometric}. Holonomic modules share very nice properties (see \cite[Definition 2.3.6]{HTT} for the definition of holonomicity and \cite[Chapter 3]{HTT} for its properties). They have finite length. They are preserved by direct images and inverse images along morphisms of smooth algebraic varieties. They admit a duality operation 
\begin{equation*}
	\BD: \Mod_{hol}(\cD_\lambda) \cong \Mod_{hol}(\cD_\lambda{}^{op}) \cong \Mod_{hol}(\cD_{-\lambda})
\end{equation*}
where $\cD_\lambda{}^{op}$ denotes the opposite algebra of $\cD_\lambda$ (the last equality follows from the identification $\cD_\lambda{}^{op} \cong \cD_{-\lambda}$, see \cite[\textsection A.2 pp.325]{Hecht-Milicic-Schmid-Wolf:Localization1} or \cite[pp.44, No.9 Example 3]{Beilinson-Bernstein:Subrep}). For a morphism $f$ between smooth varieties, we denote direct images of holonomic $\cD$-modules by $f_+$, $f_!$ and inverse images by $f^+$, $f^!$. Here $f_+$ agrees with the definition in \cite[VI.5]{Borel:D-mods} and with $\int_f$ in \cite[\textsection 1.5]{HTT}. The functor $f_!$ is obtained by conjugating $f_+$ by holonomic duality $\BD$, which agrees with $\int_{f!}$ in \cite[Definition 3.2.3]{HTT}. The pullback $f^!$ agrees with the one defined in \cite[VI.4]{Borel:D-mods} and with $f^\dagger$ in \cite[\textsection 1.5]{HTT}. When $f$ is a closed immersion of a smooth subvariety, $H^0 f^! \cV$ consists of sections of $\cV$ supported in the subvariety. The functor $Lf^+$ is a shift of $f^!$ by the relative dimension ($Lf^*$ in \cite[\textsection 1.5]{HTT}); forgetting the $\cD$-module structures, $f^+ := H^0 Lf^+$ agrees with the usual $\cO$-module inverse image $f^*$. All $\eta$-twisted Harish-Chandra sheaves are functorial with respect to all these operations.

Let $C(w)$, $w \in W$ be the Schubert cells (i.e. $N$-orbits) on $X$, with inclusion maps $i_w: C(w) \to X$. There exist nonzero $\eta$-twisted Harish-Chandra sheaves on $C(w)$ if and only if $w = w^C$ is the longest element in the right $W_\Theta$-coset that contains it. If this is the case, the category $\Mod_{coh}(\cD_{C(w^C)},N,\eta)$ is semisimple, in which the unique irreducible object, denoted by $\cO_{C(w^C)}^\eta$, has $\cO_{C(w^C)}$ as the underlying structure of an $N$-equivariant $\cO_{C(w^C)}$-module, but with an $\eta$-twisted $\cD_{C(w^C)}$-action (these results are contained in \cite[\textsection 3 and \textsection 4]{Milicic-Soergel:Whittaker_geometric}). We call the $\cD$-module direct images
\begin{equation*}
	\cI(w^C,\lambda,\eta) = i_{w^C+} \cO_{C(w^C)}^\eta,\qquad
	\cM(w^C,\lambda,\eta) = i_{w^C!} \cO_{C(w^C)}^\eta
\end{equation*}
the \textbf{standard module} and the \textbf{costandard module}, respectively. The standard module $\cI(w^C,\lambda,\eta)$ contains a unique irreducible submodule
\begin{equation*}
	\cL(w^C,\lambda,\eta),
\end{equation*}
and $\cL(w^C,\lambda,\eta)$ is the unique irreducible quotient of $\cM(w^C,\lambda,\eta)$. These exhaust all irreducible objects in $\Mod_{coh}(\cD_\lambda,N,\eta)$ (\cite[\textsection 3.4]{HTT}, \cite[\textsection 4]{Milicic-Soergel:Whittaker_geometric}). Romanov \cite[Theorem 9]{Romanov:Whittaker} showed (using the character theory she developed) that if $\lambda$ is antidominant, Beilinson-Bernstein's equivalence sends $\cM(w^C,\lambda,\eta)$ to $M(w^C\lambda,\eta)$ and $\cL(w^C,\lambda,\eta)$ to either $L(w^C\lambda,\eta)$ or $0$. If $\lambda$ is furthermore regular, $\cL(w^C,\lambda,\eta)$ is always sent to $L(w^C\lambda,\eta)$. This allows us to study Whittaker modules using geometry on $X$.

In practice, we work with $\cI(w^C,\lambda,\eta)$ rather than $\cM(w^C,\lambda,\eta)$ because $f_+$ is more natural in the $\cD$-module theory than $f_!$. The holonomic duality $\BD$ sends $\cM(w^C,\lambda,\eta)$ to $\cI(w^C,-\lambda,\eta)$ (because $f_! = \BD \circ f_+ \circ \BD$) and hence sends the unique irreducible quotient $\cL(w^C,\lambda,\eta)$ of $\cM(w^C,\lambda,\eta)$ to the unique irreducible submodule $\cL(w^C,-\lambda,\eta)$ of $\cI(w^C,-\lambda,\eta)$ \cite[Proposition 3.4.3]{HTT}. So we have the following flowchart
\begin{equation*}
	\cN_{\theta,\eta} \xrightarrow{\cD_\lambda\dotimes_{\cU_\theta}-} \Mod_{coh}(\cD_\lambda,N,\eta) \xrightarrow{\BD} \Mod_{coh}(\cD_{-\lambda},N,\eta),
\end{equation*}
\begin{equation*} 
	L(w^C\lambda,\eta) \mapsto \cL(w^C,\lambda,\eta) \mapsto \cL(w^C,-\lambda,\eta),
\end{equation*}
\begin{equation*}
	M(w^C\lambda,\eta) \mapsto \cM(w^C,\lambda,\eta) \mapsto \cI(w^C,-\lambda,\eta).
\end{equation*}
Because of the finite length property, the set of irreducible objects form a basis for the Grothendieck group $K \Mod_{coh}(\cD_{-\lambda},N,\eta)$. A standard argument using pullback-pushforward adjunctions shows that the set of standard modules also form a basis for $K \Mod_{coh}(\cD_{-\lambda},N,\eta)$. Therefore, our goal of finding coefficients of $\ch M(w^D\lambda,\eta)$ in $\ch L(w^C\lambda,\eta)$ is the same as finding the change of bases matrix from the $\cL$ basis to the $\cI$ basis.
\section{Double cosets in the Weyl group}\label{sec:db_coset}

In this section, we collect some known results on the integral root subsystem, and examine the structure of double $(W_\Theta, W_\lambda)$-cosets in $W$. Most results on here are either known or not hard. We include most of the proofs for completeness. We will refer to the example in \textsection\ref{sec:examples} when combinatorial objects are introduced. 

In \textsection\ref{subsec:Bruhat_W/Wlambda} we define a cross-section of $W/W_\lambda$ and examine the restriction of Bruhat order to each coset. The next subsection \textsection\ref{subsec:WTheta_prelim} sets notations and collects some known facts on $W_\Theta \backslash W$. In \textsection\ref{subsec:db_coset_xsec}, we construct a cross-section $A_{\Theta,\lambda}$ of $W_\Theta \backslash W / W_\lambda$ consisting of the unique shortest elements in each double coset (Corollary \ref{thm:cross-section_db_coset}). Next, we show in \textsection\ref{subsec:int_model} that, if one looks at the partition of $W_\Theta \backslash W$ given by double cosets $W_\Theta \backslash W / W_\lambda$, then each block in this partition corresponds to a right coset in $W_\lambda$ of a parabolic subgroup of $W_\lambda$, called the ``integral model'' for this double coset. As mentioned in \textsection\ref{sec:intro}, the Whittaker Kazhdan-Lusztig polynomials for $(W_\lambda,\Pi_\lambda)$ with respect to this parabolic subgroup describe the multiplicities of Whittaker modules indexed by right $W_\Theta$-cosets in this double coset. Lastly, in \textsection\ref{subsec:lem_induction}, we prove a lemma which enables a key induction step in \textsection\ref{subsec:(4)}.

We refer the readers to \textsection\ref{sec:prelim} for the definitions of the root system theoretic objects $\Sigma$, $\Sigma^+$, $\Pi$, $W$, and their variants defined by $\Theta$ and $\lambda$.

\subsection{Left $W_\lambda$-cosets and Bruhat order}\label{subsec:Bruhat_W/Wlambda}

For any $u \in W$, define the set
\begin{equation*}
	\Sigma_u^+ = \{ \alpha \in \Sigma^+ \mid u \alpha \in - \Sigma^+\} = \Sigma^+ \cap (-u\inv \Sigma^+),
\end{equation*}
i.e. the set of positive roots $\alpha$ so that $u\alpha$ is not positive. Write
\begin{equation*}
	A_\lambda = \{ u \in W \mid \Sigma_u^+ \cap \Sigma_\lambda = \varnothing\}.
\end{equation*}
The following is well-known.

\begin{lemma}
	The set $A_\lambda$ is a cross-section of $W/W_\lambda$. 
\end{lemma}

\begin{proof}
	This proof is copied verbatim from Mili\v ci\'c's unpublished notes. Observe that
	\begin{align*}
		\Sigma_u^+ \cap \Sigma_\lambda 
		&= \Sigma^+ \cap (-u\inv \Sigma^+) \cap \Sigma_\lambda && (\text{by definition of } \Sigma_u^+)\\
		&= (\Sigma^+ \cap \Sigma_\lambda) \cap (-u\inv \Sigma^+) && (\text{rearranging terms})\\
		&= \Sigma_\lambda^+ \cap (-u\inv \Sigma^+) && (\text{by definition of } \Sigma_\lambda^+).
	\end{align*}
	Hence $\Sigma_u^+ \cap \Sigma_\lambda = \varnothing \iff \Sigma_\lambda^+ \subseteq u\inv \Sigma^+$, and
	\begin{equation}\label{eqn:A_lambda_alt}
		A_\lambda = \{u \in W \mid \Sigma_\lambda^+ \subseteq u\inv \Sigma^+\}.
	\end{equation}
	
	We first show that any left $W_\lambda$-coset has a representative in $A_\lambda$. Take any $w \in W$. Then $w\inv \Sigma^+$ is a set of positive roots in $\Sigma$. Hence $\Sigma_\lambda \cap w\inv \Sigma^+$ is a set of positive roots in $\Sigma_\lambda$. So there is an element $v \in W_\lambda$ such that $v(\Sigma_\lambda \cap w\inv \Sigma^+) = \Sigma_\lambda^+$, or equivalently $\Sigma_\lambda \cap vw\inv \Sigma^+ = \Sigma_\lambda^+$ (because $v \Sigma_\lambda = \Sigma_\lambda$). In particular $\Sigma_\lambda^+ \subseteq vw\inv \Sigma^+$, and hence $(v w\inv) = wv\inv \in A_\lambda$ by the above alternative description of $A_\lambda$. As a result $w \in A_\lambda v \subseteq A_\lambda W_\lambda$. This shows $W = A_\lambda W_\lambda$, and any left $W_\lambda$-coset has a representative in $A_\lambda$.
	
	Now suppose $u_1,u_2 \in A_\lambda$ are in the same left $W_\lambda$-coset, i.e. there is $v \in W_\lambda$ with $u_1 = u_2 v$. This implies
	\begin{align*}
		\Sigma_\lambda^+ 
		&= \Sigma_\lambda \cap u_1\inv \Sigma^+ && (\text{since } u_1 \in A_\lambda \text{ and because of (\ref{eqn:A_lambda_alt})})\\
		&= \Sigma_\lambda \cap v\inv u_2\inv \Sigma^+  && (\text{since } u_1 = u_2 v)\\
		&= v\inv ( \Sigma_\lambda \cap u_2\inv \Sigma^+) && (\text{using } v\inv \Sigma_\lambda = \Sigma_\lambda \text{ and factoring $v\inv$ out})\\
		&= v\inv \Sigma_\lambda^+ && (\text{since } u_2 \in A_\lambda \text{ and because of (\ref{eqn:A_lambda_alt})}).
	\end{align*}
	Since $W_\lambda$ acts simply transitively on the set of sets of positive roots of $\Sigma_\lambda$, we have $v = 1$ and $u_1 = u_2$. Thus $A_\lambda$ is a cross-section of $W/W_\lambda$.
\end{proof}

The set $\Sigma_\lambda$, $W_\lambda$ and $A_\lambda$ satisfy the following elementary properties. 

\begin{lemma}\label{lem:non-int_refl_subsys}
	Let $\beta$ be a simple root and let $u \in W$. Write $s_\beta$ for the reflection of $\beta$. Then
	\begin{enumerate}[label=(\alph*)]
		\item $u \Sigma_\lambda = \Sigma_{u \lambda}$;
		\item if $u \in A_\lambda$, $u \Sigma_\lambda^+ = \Sigma_{u \lambda}^+$;
		\item if $u \in A_\lambda$, $u \Pi_\lambda = \Pi_{u \lambda}$;
		\item $u W_\lambda u\inv = W_{u \lambda}$;
		\item if $s_\beta \in A_\lambda$ and $u \in A_\lambda$, $us_\beta \in A_{s_\beta\lambda}$.
		\item $s_\beta \in A_\lambda$ if and only if $\beta \in \Pi-\Pi_\lambda$.
	\end{enumerate}
\end{lemma}

\begin{proof}
	(a): for any $\alpha \in \Sigma_\lambda$, $(u \alpha)^\vee (u\lambda) = \alpha^\vee( u\inv u \lambda) = \alpha^\vee(\lambda) \in \BZ$. Hence $u \alpha \in \Sigma_{u\lambda}$ and $u \Sigma_\lambda \subseteq \Sigma_{u\lambda}$ by the definition of $\Sigma_{u\lambda}$. Since both sets have the same size, equality holds. 
	
	(b): from (\ref{eqn:A_lambda_alt}), we know $u \Sigma_\lambda^+ \subseteq \Sigma^+$. Hence 
	\begin{equation*}
		u \Sigma_\lambda^+ 
		= u \Sigma_\lambda \cap \Sigma^+ 
		= \Sigma_{u \lambda} \cap \Sigma^+ 
		= \Sigma_{u \lambda}^+.
	\end{equation*}
	
	(c): elements in $\Pi_\lambda$ and $\Pi_{u\lambda}$ can be characterized by not being sums of other elements of $\Sigma_\lambda^+$ and $\Sigma_{u \lambda}^+$, respectively. Since $u: \Sigma_\lambda^+ \to \Sigma_{u \lambda}^+$ commutes with sums, it must send $\Pi_\lambda$ to $\Pi_{u \lambda}$.
	
	(d): for any $w \in W_\lambda$, 
	\begin{equation*}
		(u w u\inv) u \lambda - u \lambda 
		= u (w \lambda - \lambda)
		\in u (\BZ \cdot \Sigma) = \BZ \cdot \Sigma.
	\end{equation*}
	Hence $u W_\lambda u\inv \subseteq W_{u \lambda}$ by definition of $W_{u\lambda}$. Since both sides have the same size, equality holds.
	
	(e): observe
	\begin{align*}
		\Sigma_{us_\beta}^+ \cap \Sigma_{s_\beta\lambda}\
		&= (\Sigma^+ \cap (-(us_\beta)\inv \Sigma^+)) \cap \Sigma_{s_\beta\lambda} &&(\text{by definition of } \Sigma_{us_\beta}^+)\\
		&= (\Sigma^+ \cap \Sigma_{s_\beta \lambda}) \cap (-(us_\beta)\inv \Sigma^+) && (\text{rearranging terms})\\
		&= \Sigma_{s_\beta\lambda}^+ \cap (-(us_\beta)\inv \Sigma^+) && (\text{by definition of } \Sigma_{s_\beta \lambda}^+)\\
		&= s_\beta \Sigma_\lambda^+ \cap (-(us_\beta)\inv \Sigma^+) && (\text{by part (b)}).
	\end{align*}
	Hence
	\begin{align*}
		us_\beta \in A_{s_\beta \lambda}
		&\iff \Sigma_{us_\beta}^+ \cap \Sigma_{s_\beta\lambda}  = \varnothing && (\text{by definition of } A_{s_\beta\lambda})\\
		&\iff s_\beta \Sigma_\lambda^+ \cap (-(us_\beta)\inv \Sigma^+) = \varnothing && (\text{by the above observation})\\
		&\iff s_\beta \Sigma_\lambda^+ \subseteq (us_\beta)\inv \Sigma^+ = s_\beta u\inv \Sigma^+\\
		&\iff u\Sigma_\lambda^+ \subseteq \Sigma^+ && (\text{multiplying both sides by } us_\beta)\\
		&\iff u \in A_\lambda && (\text{by (\ref{eqn:A_lambda_alt})})
	\end{align*}
	which is true by assumption.
	
	(f): if $s_\beta \in A_\lambda$, then since $A_\lambda$ is a cross-section of $W/W_\lambda$ and $1$ is already in $W_\lambda$, we must have $s_\beta \not\in W_\lambda$. Hence
	\begin{equation*}
		-\beta^\vee(\lambda)\beta = (\lambda - \beta^\vee(\lambda)\beta) - \lambda = s_\beta \lambda - \lambda \not\in \BZ \cdot \Sigma,
	\end{equation*}
	and $\beta^\vee(\lambda) \not\in \BZ$. Therefore $\beta \not\in \Pi_\lambda$. On the other direction, suppose $\beta \not\in \Pi_\lambda$. Since the only positive root moved out of $\Sigma^+$ by $s_\beta$ is $\beta$, and $\beta$ is not in $\Sigma_\lambda$, we see that $s_\beta \Sigma_\lambda^+ \subseteq \Sigma^+$. This implies $s_\beta \in A_\lambda$ by (\ref{eqn:A_lambda_alt}).
\end{proof}

In particular, (c) and (d) imply that conjugation by $u \in A_\lambda$ sends simple reflections in $W_\lambda$ to simple reflections in $W_{u \lambda}$. This implies:

\begin{corollary}\label{lem:Bruhat_order_conj}
	Let $u \in A_\lambda$. Then conjugation by $u$ is an isomorphism of posets
	\begin{equation*}
		(W_\lambda,\le_\lambda) \bijects (W_{u\lambda},\le_{u\lambda}).
	\end{equation*}
\end{corollary}

We want to show that $A_\lambda$ consists of unique shortest elements in left cosets. We in fact have a stronger statement: left multiplication by an element in $A_\lambda$ is a map from $W_\lambda$ to $W$ that preserves the Bruhat orders. 

\begin{lemma}\label{lem:Bruhat_order_pre}
	Let $w$, $s_\alpha \in W_\lambda$ with $\alpha \in \Sigma_\lambda^+$, and let $u \in A_\lambda$. Suppose $s_\alpha w <_\lambda w$. Then $u s_\alpha w < u w$.
\end{lemma}

\begin{proof}
	Consider the projection $\fh^* \surj \operatorname{span} \Sigma_\lambda$ along the subspace $\bigcap_{\alpha \in \Sigma_\lambda} \ker \alpha$. For an element $\mu \in \fh^*$, we write $\bar \mu$ for its image under this projection.
	
	An inequality in $W$ with respect to Bruhat order can be checked by a regular antidominant integral weight. That is, if $\mu$ is such a weight in $\fh^*$, then $u s_\alpha w < u w$ if and only if $u s_\alpha w \mu < u w \mu$, where the second inequality means that $u w \mu - u s_\alpha w \mu$ is nonzero and is a non-negative sum of simple roots. Similarly, $s_\alpha w <_\lambda w$ if and only if $s_\alpha w \bar \mu <_\lambda w \bar \mu$. 
	
	Therefore, if we write $\nu = \mu-\bar\mu$,
	\begin{align*}
		s_\alpha w <_\lambda w
		&\iff s_\alpha w \bar\mu <_\lambda w \bar\mu\\
		&\iff s_\alpha w \bar\mu + \sum_{\alpha_i \in \Pi_\lambda} a_i \alpha_i = w \bar\mu \text{ for some } a_i \in \BZ_{\ge 0} \text{ not all zero}\\
		&\iff s_\alpha w \bar\mu + \nu + \sum_{\alpha_i \in \Pi_\lambda} a_i \alpha_i = w \bar\mu + \nu \text{ for some } a_i \in \BZ_{\ge 0} \text{ not all zero}\\
		&\iff s_\alpha w \mu + \sum_{\alpha_i \in \Pi_\lambda} a_i \alpha_i = w \mu \text{ for some } a_i \in \BZ_{\ge 0} \text{ not all zero}
	\end{align*} 
	where the last step is because $\nu$ is annihilated by all coroots in $\Sigma_\lambda^\vee$.
	Applying $u$ to both sides we get
	\begin{equation*}
		u s_\alpha w \mu + \sum_{\alpha_i \in \Pi_\lambda} a_i u \alpha_i = u w \mu \text{ for some } a_i \ge 0 \text{ not all zero}
	\end{equation*}
	which implies $u s_\alpha w \mu < u w \mu$, by the fact that $u \alpha_i \in u \Sigma_\lambda^+ \subseteq \Sigma^+$. Thus $u s_\alpha w < uw$ as desired.
\end{proof}

\begin{corollary}\label{lem:Bruhat_order}
	Let $v,w \in W_\lambda$ and $v \le_\lambda w$. Then for any $u \in A_\lambda$, $uv \le uw$.
\end{corollary}

\begin{proof}
	If equality holds, then the statement is trivial. Otherwise, by the definition of Bruhat order, there exist $\alpha_1,\ldots,\alpha_k \in \Sigma_\lambda^+$ such that 
	\begin{equation*}
		v = s_{\alpha_k} \cdots s_{\alpha_1} w 
		<_\lambda \cdots <_\lambda
		s_{\alpha_1} w <_\lambda w.
	\end{equation*} 
	Apply \ref{lem:Bruhat_order_pre} to each inequality, we see
	\begin{equation*}
		uv = u s_{\alpha_k} \cdots s_{\alpha_1} w 
		< \cdots <
		u s_{\alpha_1} w < u w
	\end{equation*}
	as desired.
\end{proof}

\begin{corollary}\label{lem:u_smallest}
	For any $u \in A_\lambda$, $u$ is the unique shortest element in $u W_\lambda$ with respect to the restriction of Bruhat order to $u A_\lambda$.	
\end{corollary}

\begin{example}
	In the example in \textsection\ref{sec:examples}, $A_\lambda = \{1, s_\alpha, s_\beta, s_\gamma s_\beta\}$. The left $W_\lambda$-cosets are $W_\lambda$ (whose elements are crossed-out in (\ref{diag:W(A3)})), $s_\alpha W_\lambda$, $s_\beta W_\lambda$, and $s_\gamma s_\beta W_\lambda$ (whose elements are underlined in (\ref{diag:W(A3)})).
\end{example}

The next lemma is analogous to a similar statement for parabolic subgroups (Lemma \ref{lem:Theta_prelim}), which we will need in a few occasions. The proof is a standard argument using the lifting property \cite[2.2.7]{Bjorner-Brenti:Coxeter}.

\begin{lemma}\label{lem:u_in_db_coset}
	Let $\alpha \in \Pi$, and $u \in A_\lambda$. Then either $s_\alpha u \in A_\lambda$, or $s_\alpha u \in u W_\lambda$.
\end{lemma}

\begin{proof}
	Suppose $s_\alpha u \not\in u W_\lambda$. Then $s_\alpha u$ is in a different left $W_\lambda$-coset, i.e. $s_\alpha u \in r W_\lambda$ for some $r \in A_\lambda$ with $r \neq u$. So there exists some $v \in W_\lambda$ such that $s_\alpha u = r v$. We need to show that $v = 1$. 
	
	Write $w_1 \vartriangleleft w_2$ when $w_1 < w_2$ and $\ell(w_1) = \ell(w_2) -1$. From the relation $s_\alpha u = r v$, either $rv \vartriangleleft u$ or $rv \vartriangleright u$. Also $s_\alpha uv\inv = r$, so either $r \vartriangleleft uv\inv$ or $r \vartriangleright uv\inv$. From \ref{lem:u_smallest}, we also know $r \le rv$ and $u \le uv\inv$. We have the following four possibilities.
	
	\begin{enumerate}[label=(\alph*)]
		\item 
		$\begin{tikzcd}[row sep = small, column sep = small]
			r \ar[r, phantom, "\vartriangleright" description] & uv\inv \ar[d, phantom,  "\rotatebox{-90}{$\ge$}" description]\\
			rv \ar[u, phantom,  "\rotatebox{90}{$\ge$}" description] & u \ar[l, phantom,  "\vartriangleleft" description]
		\end{tikzcd}$
		is impossible since it implies $u > u$.
		
		\item 
		$\begin{tikzcd}[row sep = small, column sep = small]
			r \ar[r, phantom, "\vartriangleright" description] & uv\inv \ar[d, phantom,  "\rotatebox{-90}{$\ge$}" description]\\
			rv \ar[u, phantom,  "\rotatebox{90}{$\ge$}" description] & u \ar[l, phantom,  "\vartriangleright" description]
		\end{tikzcd}$.
		If $rv > r$, then from $rv > r \vartriangleright uv\inv \ge u$ we see that $\ell(rv) \ge \ell(u) + 2$, which violates $rv \vartriangleright u$. Therefore we must have $rv = r$ and hence $v = 1$.
		
		\item 
		$\begin{tikzcd}[row sep = small, column sep = small]
			r \ar[r, phantom, "\vartriangleleft" description] & uv\inv \ar[d, phantom,  "\rotatebox{-90}{$\ge$}" description]\\
			rv \ar[u, phantom,  "\rotatebox{90}{$\ge$}" description] & u \ar[l, phantom,  "\vartriangleleft" description]
		\end{tikzcd}$.
		Same argument as in (b) shows that $v = 1$.
		
		\item 
		$\begin{tikzcd}[row sep = small, column sep = small]
			r \ar[r, phantom, "\vartriangleleft" description] & uv\inv \ar[d, phantom,  "\rotatebox{-90}{$\ge$}" description]\\
			rv \ar[u, phantom,  "\rotatebox{90}{$\ge$}" description] & u \ar[l, phantom,  "\vartriangleright" description]
		\end{tikzcd}$.
		Let $k = \ell(uv\inv) - \ell(u)$. Then 
		\begin{align*}
			\ell(rv) 
			&\ge\ell(r) \\
			&= \ell(uv\inv) -1\\
			&= \ell(u) + k-1\\
			&= \ell(rv) -1 + k-1\\
			&= \ell(rv) + k-2
		\end{align*}
		and $0 \le k \le 2$. If $k = 2$, then $\ell(r) = \ell(rv)$ and $v = 1$. If $k = 0$, then $\ell(u) = \ell(uv\inv)$ and $v = 1$. Suppose $k = 1$. Applying the lifting property twice, we see $r \le u$ and $u \le r$. So $r = u$, contradicting our assumption for $r$. Therefore we must have $v = 1$.
	\end{enumerate}
	Thus $v = 1$ in all cases, as desired.
\end{proof}

\subsection{Notations and preliminaries on $W_\Theta \backslash W$}\label{subsec:WTheta_prelim}

We recall some well-known facts of right $W_\Theta$-cosets and partial orders. Details of these facts can be found in \cite[Chapter 6]{Milicic:Localization}. See also \cite[\textsection 2.4 and \textsection 2.5]{Bjorner-Brenti:Coxeter} for proofs of analogous results for left $W_\Theta$-cosets.

Write $w_\Theta \in W_\Theta$ for the longest element. The set
\begin{equation*}
	{}^\Theta W = \{ w \in W \mid w\inv \Theta \subseteq - \Sigma^+\}
\end{equation*}
is a cross-section of $W_\Theta \backslash W$ consisting of the longest elements in each coset, and
\begin{equation}\label{eqn:w_Theta_Theta_W}
	w_\Theta {}^\Theta W = \{w \in W \mid w\inv \Theta \subseteq \Sigma^+\}
\end{equation}
is a cross-section consisting of the shortest elements in each coset. For a right $W_\Theta$-coset $C$, we write $w^C$ for the corresponding element in ${}^\Theta W$. The restriction of Bruhat order on the set ${}^\Theta W$ defines a partial order $\le$ on $W_\Theta \backslash W$. We will use the phrase ``the length of $C$'' to refer to the length of the element $w^C$. If ${\Theta(u,\lambda)}$ is a subset of $\Pi_\lambda$ defining the parabolic subgroup $W_{\lambda,\Theta(u,\lambda)} \subseteq W_\lambda$, we write ``$\le_{u,\lambda}$'' for the partial order on $W_{\lambda,{\Theta(u,\lambda)}} \backslash W_\lambda$.

\begin{example}
	In the example in \textsection\ref{sec:examples} (specifically in (\ref{diag:W(A3)})), Weyl group elements are grouped together based on the partition given by $W_\Theta \backslash W$. Elements that are contained in ${}^\Theta W$ are surrounded by a shape.
\end{example}

The following facts will be used throughout this section.

\begin{lemma}\label{lem:W_Theta_permutes_roots}
	Any element in $W_\Theta$ permutes positive roots outside $\Sigma_\Theta^+$, that is, it permutes the set $\Sigma^+ - \Sigma_\Theta^+$.
\end{lemma}

\begin{proof}
	It suffices to prove the statement for a simple reflection $s_\beta \in W_\Theta$. Since $s_\beta$ permutes $\Sigma^+ - \{\pm\beta\}$ and also permutes $\Sigma - \Sigma_\Theta$, it permutes $(\Sigma^+ - \{\pm\beta\}) \cap (\Sigma - \Sigma_\Theta)$ which equals $\Sigma^+ - \Sigma_\Theta^+$.
\end{proof}

\begin{lemma}\label{lem:Theta_prelim}
	Let $C$ be a right $W_\Theta$-coset and $\alpha \in \Pi$. Then exactly one of the following happens.
	\begin{enumerate}[label=(\alph*)]
		\item $C s_\alpha > C$. In this case $w^{C s_\alpha} = w^C s_\alpha$, and for any $w \in C$, $w s_\alpha > w$.
		\item $C s_\alpha = C$.
		\item $C s_\alpha < C$. In this case $w^{C s_\alpha} = w^C s_\alpha$, and for any $w \in C$, $w s_\alpha < w$.
	\end{enumerate}
	Moreover, the identity coset $W_\Theta$ is the only right $W_\Theta$-coset $C$ such that $C s_\alpha \ge C$ for all $\alpha \in \Pi$.
\end{lemma}

\begin{proof}
	With minor modifications, the results in \textsection\ref{subsec:Bruhat_W/Wlambda} can be translated to the case where we replace $\Sigma_\lambda$ by $\Sigma_\Theta$, left $W_\lambda$-cosets by right $W_\Theta$-cosets, and $A_\lambda$ by ${}^\Theta W$. Under these replacements, Lemmas \ref{lem:Bruhat_order} and \ref{lem:u_in_db_coset} say the following: 
	\begin{itemize}
		\item[(i)] Let $v$, $w \in W_\Theta$ and $v \le w$. Then for any $C \in W_\Theta \backslash W$, $v w^C \ge w w^C$.
		\item[(ii)] Let $\alpha \in \Pi$ and $C \in W_\Theta \backslash W$. Then either $w^C s_\alpha = w^{C s_\alpha}$ or $w^C s_\alpha \in C$.
	\end{itemize} 
	The second case of (ii) ($w^C s_\alpha \in C$) corresponds to (b). Suppose we are in the first case of (ii), that is $w^C s_\alpha = w^{C s_\alpha}$. Then $C s_\alpha \neq C$, otherwise $w^C s_\alpha = w^{C s_\alpha} = w^C$ is impossible. Suppose $Cs_\alpha > C$, i.e. $w^{C s_\alpha} > w^C$. We want to show that $w s_\alpha > w$ for any $w \in C$.
	
	Since $C = W_\Theta w^C$, we can write $w = v w^C$ for some unique $v \in W_\Theta$. We will do induction on $\ell(v)$. For the base case $\ell(v) = 1$, $v = s_\beta$ is simple. (i) implies $s_\beta w^{C s_\alpha} > w^{C s_\alpha}$ and $s_\beta w^C > w^C$. Combined with the assumption $w^{Cs_\alpha} > w^C$, we obtain
	\begin{equation*}
		\begin{tikzcd}[row sep = small, column sep = small]
			s_\beta w^{C s_\alpha} \ar[r, phantom, "<" description] & w^{C s_\alpha} \ar[d, phantom,  "\rotatebox{-90}{$>$}" description]\\
			s_\beta w^C  & w^C \ar[l, phantom,  "<" description]
		\end{tikzcd}.
	\end{equation*}
	If $s_\beta w^C s_\alpha < s_\beta w^C$, then the chain of inequalities 
	\begin{equation*}
		\begin{tikzcd}[row sep = small, column sep = small]
			s_\beta w^{C s_\alpha} & w^{C s_\alpha} \ar[d, phantom,  "\rotatebox{-90}{$>$}" description]\\
			s_\beta w^C \ar[u, phantom,  "\rotatebox{90}{$>$}" description] & w^C \ar[l, phantom,  "<" description]
		\end{tikzcd}
	\end{equation*}
	would imply that $s_\beta w^{Cs_\alpha}$ and $w^{Cs_\alpha}$ have length difference $\ge 3$, which is impossible since their lengths only differ by $1$. Therefore $w s_\alpha = s_\beta w^C s_\alpha > s_\beta w^C = w$. This establishes the base case.
	
	Now suppose $v = s_\beta r>r$ for some $s_\beta,r \in W_\Theta$, with $s_\beta$ simple. Induction hypothesis says $r w^{Cs_\alpha} > r w^C$. Invoking (i) again, we obtain the following inequalities
	\begin{equation*}
		\begin{tikzcd}[row sep = small, column sep = small]
			s_\beta r w^{C s_\alpha} \ar[r, phantom, "<" description] & r w^{C s_\alpha} \ar[d, phantom,  "\rotatebox{-90}{$>$}" description]\\
			s_\beta r w^C  & r w^C \ar[l, phantom,  "<" description]
		\end{tikzcd}.
	\end{equation*}
	Arguing similarly as in the base case, it is impossible to have $s_\beta r w^{C s_\alpha} < s_\beta r w^C$. Therefore $w s_\alpha = s_\beta r w^{C s_\alpha} > s_\beta r w^C = w$. This proves the additional claim in case (a). An identical argument establishes the claim in case (c). 
	
	It remains to prove the last statement. Suppose $C$ satisfies $C s_\alpha \ge C$ for all $\alpha \in \Pi$. Let $w \in C$ be an element of minimal length. If $w \neq 1$, then $w>1$, and there is a simple reflection $s_\alpha$ with $w s_\alpha < w$. By minimality of $w$, $w s_\alpha$ is in a different coset. This forces us to be in case (c), which contradicts $C s_\alpha \ge C$. So $w = 1$ and $C = W_\Theta$.
\end{proof}

This immediately implies

\begin{corollary}\label{lem:Bruhat_order_rcoset}
	Let $C$, $D$ be two right $W_\Theta$-cosets. Let $v \in D$, $w \in C$. If $v \le w$, then $D \le C$.
\end{corollary}

\begin{proof}
	We can choose simple roots $\alpha_i$ so that 
	\begin{equation*}
		w = v s_{\alpha_k} \cdots s_{\alpha_2} s_{\alpha_1} \ge 
		v s_{\alpha_k} \cdots s_{\alpha_2}  \ge  \cdots \ge v.
	\end{equation*}
	By the lemma, this implies
	\begin{equation*}
		C \ge C s_{\alpha_1} \ge \cdots \ge C s_{\alpha_1} \cdots s_{\alpha_k} = D. \qedhere
	\end{equation*}
\end{proof}

\subsection{A cross-section of $W_\Theta \backslash W / W_\lambda$}\label{subsec:db_coset_xsec}

Define the set 
\begin{equation*}
	A_{\Theta,\lambda} = A_\lambda \cap (w_\Theta {}^\Theta W).
\end{equation*} 
We will show (in Corollary \ref{thm:cross-section_db_coset}) that this is a cross-section of $W_\Theta \backslash W / W_\lambda$ consisting of the unique shortest elements in each double coset. Later results, as well as the main theorem of the paper, will often be formulated using this set.

\begin{lemma}\label{lem:u_in_same_right_coset}
	Let $u$, $r \in A_\lambda$. Suppose $u$ and $r$ are in the same $(W_\Theta,W_\lambda)$-coset. Then $u$ and $r$ are contained in the same right $W_\Theta$-coset.
\end{lemma}

\begin{proof}
	The case $u = r$ is trivial. Assume $u \neq r$. By assumption, $r = wuv\inv$ for some $w \in W_\Theta$ and $v \in W_\lambda$. We will do induction on $\ell(w)$. 
	
	Consider the case $\ell(w) = 1$. Then $w = s_\alpha$ for some $\alpha \in \Theta$, and $s_\alpha u = r v$. By \ref{lem:u_in_db_coset} $v = 1$. Hence $r = s_\alpha u \in W_\Theta u$ which is in the same right $W_\Theta$-coset as $u$. 
	
	Consider $\ell(w) = k > 1$. Write $w = s_\alpha w'$ for some $\alpha \in \Theta$ and some $w' \in W_\Theta$ so that $\ell(w) = \ell(w') + 1$. Then $w' u = (s_\alpha r) v$. We have two possibilities.
	
	\begin{enumerate}[label=(\alph*)]
		\item $s_\alpha r \in r W_\lambda$. Then there exists $v' \in W_\lambda$ such that $s_\alpha r = r v'$, so $w' u (v' v)\inv = r$. Since $\ell(w') \le k-1$, by the induction assumption, $u$ and $r$ are in the same right $W_\Theta$-coset.
		
		\item $s_\alpha r \not\in r W_\lambda$. Then by \ref{lem:u_in_db_coset}, $s_\alpha r \in A_\lambda$. From the equation $w' u v\inv = s_\alpha r$, $\ell(w') \le k-1$ and the induction assumption, we see that $u$ and $s_\alpha r$ are in the same right $W_\Theta$-coset. Since $s_\alpha r$ and $r$ are in the same right $W_\Theta$-coset, so are $u$ and $r$. \qedhere
	\end{enumerate}
\end{proof}

\begin{proposition}\label{lem:unique_smallest_right_coset}
	Consider any double coset $W_\Theta w W_\lambda$ in $W$.
	\begin{enumerate}[label=(\alph*)]
		\item $W_\Theta w W_\lambda$ contains a unique smallest right $W_\Theta$-coset $C$, in the sense that $C \le C'$ for any $C' \in W_\Theta \backslash W_\Theta w W_\lambda$.
		\item $A_\lambda \cap (W_\Theta w W_\lambda) \subseteq C$.
		\item The unique shortest element in $C$ is in $A_\lambda$.
	\end{enumerate}
\end{proposition}

\begin{proof}
	By \ref{lem:u_in_same_right_coset}, there exists a right $W_\Theta$-coset $C$, contained in $W_\Theta w W_\lambda$, such that $A_\lambda \cap (W_\Theta w W_\lambda) \subseteq C$. Let $y$ be the unique shortest element in $C$. Say $y \in u W_\lambda$ for some $u \in A_\lambda$. Then $u \le y$ by \ref{lem:u_smallest}. If $y \neq u$, we will have $u < y$, and hence by minimality of $y$, $u$ is in a different right $W_\Theta$-coset than $y$, contradicting the construction of $C$. Hence we must have $y = u$, i.e. the unique shortest element in $C$ is in $A_\lambda$. Lastly, for any other right $W_\Theta$-coset $C'$ in $W_\Theta w W_\lambda$, let $y' $ be its unique shortest element. $y'$ is contained in one of the left $W_\lambda$-cosets, say $y' \in u' W_\lambda$ for some $u' \in A_\lambda$. Then $u' \le y'$ by \ref{lem:u_smallest}. Also $u' \neq y'$ (otherwise $C' \ni y' = u' \in C$ which would imply $C = C'$). Hence $u' < y'$. Therefore $C < C'$ by \ref{lem:Bruhat_order_rcoset}. Thus $C$ is the unique smallest right $W_\Theta$-coset in $W_\Theta w W_\lambda$.
\end{proof}

\begin{corollary}\label{thm:cross-section_db_coset}
	The set $A_{\Theta,\lambda} = A_\lambda \cap (w_\Theta {}^\Theta W)$ is a cross-section of $W_\Theta \backslash W / W_\lambda$ consisting of the unique shortest elements in each double coset. For each $u \in A_{\Theta,\lambda}$, $W_\Theta u$ is the unique smallest right $W_\Theta$-coset  in $W_\Theta u W_\lambda$.
\end{corollary}

\begin{proof}
	Take any double coset $W_\Theta w W_\lambda$. By \ref{lem:unique_smallest_right_coset}(c), if we take the shortest element $u$ in the smallest right $W_\Theta$-coset in this double coset, then $u \in A_\lambda$. Hence $u \in A_\lambda \cap (w_\Theta {}^\Theta W)$. On the other hand, by \ref{lem:unique_smallest_right_coset}(b), any other right $W_\Theta$-coset in $W_\Theta w W_\lambda$ has empty intersection with $A_\lambda$. Therefore $A_\lambda \cap (w_\Theta {}^\Theta W) \cap W_\Theta w W_\lambda = \{u\}$. This shows that $A_\lambda \cap (w_\Theta {}^\Theta W)$ is a cross-section.
\end{proof}

\begin{example}
	In the example in \textsection\ref{sec:examples}, $A_{\Theta,\lambda} = \{1, s_\gamma s_\beta\}$. There are two $(W_\Theta,W_\lambda)$-cosets: $W_\Theta W_\lambda$ and $W_\Theta s_\gamma s_\beta W_\lambda$. Elements in $W_\Theta s_\gamma s_\beta W_\lambda$ are underlined in (\ref{diag:W(A3)}).
\end{example}

\subsection{Integral models}\label{subsec:int_model}

By results of the previous subsection, for each double coset $W_\Theta u W_\lambda$ one can choose $u$ to be in $A_{\Theta,\lambda}$. Then $u W_\lambda$ is contained in $W_\Theta u W_\lambda$ and it intersects with different right $W_\Theta$-cosets. It will turn out that these intersections produce a parabolic subgroup in $W_\lambda$, and the Whittaker Kazhdan-Lusztig polynomials that arise determine the coefficients in the character formula.

\begin{lemma}
	Let $u \in A_{\Theta,\lambda}$. Then $\Sigma_\Theta \cap \Pi_{u\lambda}$ is a set of simple roots for the root system $\Sigma_\Theta \cap \Sigma_{u\lambda}$.
\end{lemma}

The intersection $\Sigma_\Theta \cap \Sigma_{u\lambda}$ is a root system because both $\Sigma_\Theta$ and $\Sigma_{u\lambda}$ are root systems.

\begin{proof}
	Let $\beta \in \Sigma_\Theta \cap \Sigma_{u\lambda}$. Write $\beta$ as a $\BZ_{\ge 0}$-linear combination in terms of reflections of roots in $\Pi_{u\lambda}$. If one of the summands is from $\Pi_{u\lambda} - \Sigma_\Theta$, then writing $\beta$ as a sum of reflections of roots in $\Pi$, there is a summand that comes from $\Pi - \Theta$. This implies $\beta \not\in \Sigma_\Theta$, a contradiction. Hence $\beta$ is a sum of reflections of roots from $\Sigma_\Theta \cap \Pi_{u\lambda}$. Therefore $\Sigma_\Theta \cap \Pi_{u \lambda}$ spans $\Sigma_\Theta \cap \Sigma_{u\lambda}$. Since $\Sigma_\Theta \cap \Pi_{u\lambda}$ is a subset of simple roots in $\Sigma_{u\lambda}$, roots contained in $\Sigma_\Theta \cap \Pi_{u \lambda}$ remain simple in $\Sigma_\Theta \cap \Sigma_{u\lambda}^+$. Thus $\Sigma_\Theta \cap \Pi_{u \lambda}$ is a set of simple roots for $\Sigma_\Theta \cap \Sigma_{u\lambda}$.
\end{proof}

\sloppy Write $W_{u\lambda, \Sigma_\Theta \cap \Pi_{u \lambda}}$ for the parabolic subgroup of $W_{u\lambda}$ corresponding to $\Sigma_\Theta \cap \Pi_{u \lambda}$. Then $W_{u\lambda, \Sigma_\Theta \cap \Pi_{u \lambda}}$ is the Weyl group of $\Sigma_\Theta \cap \Sigma_{u\lambda}$ and is a subgroup of $W_\Theta \cap W_{u\lambda}$.

\begin{proposition}\label{lem:int_Whittaker_model}
	For any $u \in A_{\Theta,\lambda}$, $W_\Theta \cap W_{u\lambda} = W_{u \lambda, \Sigma_\Theta \cap \Pi_{u \lambda}}$. In particular, $W_\Theta \cap W_{u\lambda}$ is a parabolic subgroup of $W_{u \lambda}$.
\end{proposition}

\begin{proof}
	The subgroup $W_{u \lambda, \Sigma_\Theta \cap \Pi_{u \lambda}}$ is certainly contained in $W_\Theta \cap W_{u\lambda}$. Let $w \in W_\Theta \cap W_{u \lambda}$. Being in $W_\Theta$, $w$ permutes roots in $\Sigma_\Theta$; being in $W_{u\lambda}$, $w$ permutes roots in $\Sigma_{u\lambda}$. Hence $w$ permutes roots in $\Sigma_\Theta \cap \Sigma_{u\lambda}$, and it sends the set $\Sigma^+ \cap (\Sigma_\Theta \cap \Sigma_{u \lambda})$ of positive roots in $\Sigma_\Theta \cap \Sigma_{u \lambda}$ to another set of positive roots $w \Sigma^+ \cap (\Sigma_\Theta \cap \Sigma_{u \lambda})$. Since $W_{u\lambda, \Sigma_\Theta \cap \Pi_{u \lambda}}$ is the Weyl group of $\Sigma_\Theta \cap \Sigma_{u\lambda}$, there exists a unique element $v \in W_{u\lambda, \Sigma_\Theta \cap \Pi_{u \lambda}}$ that sends $w \Sigma^+ \cap (\Sigma_\Theta \cap \Sigma_{u \lambda})$ back to $\Sigma^+ \cap (\Sigma_\Theta \cap \Sigma_{u \lambda})$. Hence $vw$ permutes $\Sigma^+ \cap (\Sigma_\Theta \cap \Sigma_{u \lambda}) = \Sigma_{u\lambda}^+ \cap \Sigma_\Theta^+$. On the other hand, since $vw \in W_\Theta$, by \ref{lem:W_Theta_permutes_roots} it permutes $\Sigma^+ - \Sigma_\Theta^+$; $vw$ is also in $W_{u \lambda}$, so it permutes $\Sigma_{u\lambda}$. Hence, it permutes $(\Sigma^+ - \Sigma_\Theta^+) \cap \Sigma_{u\lambda} = \Sigma_{u\lambda}^+ - \Sigma_\Theta^+$. As a result, $vw$ permutes
	\begin{equation*}
		\big( \Sigma_{u\lambda}^+ \cap \Sigma_\Theta^+ \big) \cup \big( \Sigma_{u\lambda}^+ - \Sigma_\Theta^+ \big) 
		= \Sigma_{u\lambda}^+.
	\end{equation*}
	Since $W_{u\lambda}$ acts simply transitively on the set of sets of positive roots in $\Sigma_{u\lambda}$, we must have $vw = 1$. Therefore $w = v\inv \in W_{u\lambda,\Sigma_\Theta \cap \Pi_{u \lambda}}$. Thus $W_\Theta \cap W_{u\lambda} = W_{u \lambda, \Sigma_\Theta \cap \Pi_{u \lambda}}$, as desired.
\end{proof}

For $u \in A_{\Theta,\lambda}$, write 
\begin{equation*}
	\Theta(u,\lambda) = u\inv (\Sigma_\Theta \cap \Pi_{u\lambda}) = u\inv \Sigma_\Theta \cap \Pi_\lambda.
\end{equation*}
Since $\Sigma_\Theta \cap \Pi_{u\lambda}$ is a subset of simple roots in $\Sigma_{u\lambda}$, ${\Theta(u,\lambda)}$ is a subset of simple roots in $u\inv \Sigma_{u \lambda} = \Sigma_\lambda$. Write $W_{\lambda,{\Theta(u,\lambda)}}$ for the parabolic subgroup of $W_\lambda$ corresponding to ${\Theta(u,\lambda)}$. 

\begin{example}
	Recall that $A_{\Theta,\lambda} = \{1,s_\gamma s_\beta\}$ in the example in \textsection\ref{sec:examples}. There, 
	\begin{align*}
		\Theta(1,\lambda) &= \{\alpha+\beta\}, & W_{\lambda,\Theta(1,\lambda)} &= \{1, s_{\alpha+\beta}\},\\
		\Theta(s_\gamma s_\beta,\lambda) &= \Pi_\lambda,& 
		W_{\lambda,\Theta(s_\gamma s_\beta,\lambda)} &= W_\lambda,
	\end{align*}
	where $W_\lambda$ is the type $A_2$ Weyl group generated by $s_{\alpha+\beta}$ and $s_\gamma$.
\end{example}

\begin{proposition}\label{thm:int_Whittaker_model}
	Let $u \in A_{\Theta,\lambda}$. The left-multiplication-by-$u$ map
	\begin{equation*}
		W_\lambda \bijects u W_\lambda
	\end{equation*}
	induces bijections
	\begin{equation*}
		\begin{array}{ccccc}
			W_{\lambda,\Theta(u,\lambda)} \backslash W_\lambda &\bijects& \big\{C \cap u W_\lambda \mid C \in W_\Theta \backslash W_\Theta u W_\lambda \big\} &\bijects& W_\Theta \backslash W_\Theta u W_\lambda\\
			W_{\lambda,\Theta(u,\lambda)} v &\mapsto& u W_{\lambda,\Theta(u,\lambda)} v = W_\Theta uv \cap u W_\lambda &\mapsto& W_\Theta uv.
		\end{array}		
	\end{equation*}
	
	Moreover, this map preserves the partial orders on cosets: if $C',D' \in W_{\lambda,{\Theta(u,\lambda)}} \backslash W_\lambda$ are sent to $C \cap u W_\lambda$ and $D \cap u W_\lambda$, respectively, then $D' \le_{u,\lambda} C'$ implies $D \le C$.
\end{proposition}

\begin{proof}
	Consider the smallest right $W_\Theta$-coset $W_\Theta u$ of $W_\Theta u W_\lambda$. 
	\begin{align*}
		W_\Theta u \cap u W_\lambda 
		&= (W_\Theta \cap u W_\lambda u\inv) u\\
		&= (W_\Theta \cap W_{u\lambda} )u\\
		&= W_{u\lambda, \Sigma_\Theta \cap \Pi_{u \lambda}} u\\
		&= W_{u\lambda, u{\Theta(u,\lambda)}} u\\
		&= (u W_{\lambda,{\Theta(u,\lambda)}} u\inv) u\\
		&= u W_{\lambda ,{\Theta(u,\lambda)}}.
	\end{align*}
	Hence left multiplication by $u$ sends $W_{\lambda ,{\Theta(u,\lambda)}}$ to $W_\Theta u \cap u W_\lambda$. Since left multiplication by $u$ commutes with right multiplication by elements of $W_\lambda$, it sends right $W_{\lambda,{\Theta(u,\lambda)}}$-cosets in $W_\lambda$ to right $W_\lambda$-translates of $W_\Theta u \cap u W_\lambda$, which gives us $C \cap u W_\lambda$ for various right $W_\Theta$-cosets $C$ in $W_\Theta u W_\lambda$. Moreover, any right $W_\Theta$-coset $C$ in $W_\Theta u W_\lambda$ is obtained as a right $W_\lambda$-translation of $W_\Theta u$, hence the intersection $C \cap u W_\lambda$ is necessarily the image of a right $W_{\lambda,{\Theta(u,\lambda)}}$-coset.
	
	To show that this map is order preserving, take two right $W_{\lambda,{\Theta(u,\lambda)}}$-cosets $C'$ and $D'$ such that $D' \le_{u,\lambda} C'$. This means that the $\le_\lambda$-longest elements $v^{D'}$, $v^{C'}$ of $D'$ and $C'$ satisfy $v^{D'} \le_\lambda v^{C'}$. Since left multiplication by $u$ preserves Bruhat orders (Corollary \ref{lem:Bruhat_order}), $u v^{D'} \le u v^{C'}$. Therefore $D \le C$ by \ref{lem:Bruhat_order_rcoset}.
\end{proof}

If we take the union of the maps in Proposition \ref{thm:int_Whittaker_model} for each $u$, we obtain the following bijection.

\begin{corollary}\label{lem:right_coset_partition}
	As $u$ ranges over $A_{\Theta,\lambda}$, left multiplication by $W_\Theta u$ defines a bijection
	\begin{equation*}
		\ind_\lambda: \bigcup_{u \in A_{\Theta,\lambda}} W_{\lambda, {\Theta(u,\lambda)}} \backslash W_\lambda \bijects W_\Theta \backslash W,\quad
		W_{\lambda, {\Theta(u,\lambda)}} v  \mapsto W_\Theta uv
	\end{equation*}
	which is order-preserving when restricted to each $W_{\lambda, {\Theta(u,\lambda)}} \backslash W_\lambda$ and commutes with right multiplication by $W_\lambda$. The image of $W_{\lambda, {\Theta(u,\lambda)}} \backslash W_\lambda$ is equal to $W_\Theta \backslash W_\Theta u W_\lambda$.
\end{corollary}

\begin{notation}\label{not:right_coset_partition}
	We will write 
	\begin{equation*}
		(-)|_\lambda: W_\Theta \backslash W \to \bigcup_{u \in A_{\Theta,\lambda}} W_{\lambda, {\Theta(u,\lambda)}} \backslash W_\lambda	
	\end{equation*}
	for the inverse map. If $C$ and $D$ are both sent to $W_{\lambda,\Theta(u,\lambda)} \backslash W_\lambda$, we will write $C \le_{u,\lambda} D$ for $C|_\lambda \le_{u,\lambda} D|_\lambda$.	
\end{notation}

The map $(-)|_\lambda$ plays an important role towards our goal. As explained in the introduction, standard and irreducible Whittaker modules in $\cN_{\theta,\eta}$ are parameterized by $W_\Theta \backslash W$, but compared to the integral case, $\cN_{\theta,\eta}$ is divided into smaller blocks. The map $(-)|_\lambda$ reflects this division: on the level of standard/irreducible modules, modules that correspond to various cosets $C$ in the same $(W_\Theta,W_\lambda)$-coset are in the same block, and each block looks like an integral Whittaker category (at least on the level of standard and irreducible modules) modeled by $W_{\lambda, {\Theta(u,\lambda)}} \backslash W_\lambda$.

\begin{example}
	Let us look at the double coset $W_\Theta W_\lambda$ in the example in \textsection\ref{sec:examples}. There are three right $W_\Theta$-cosets in $W_\Theta W_\lambda$. The longest elements in these cosets are $s_\alpha s_\beta s_\alpha$, $s_\alpha s_\beta s_\alpha s_\gamma$, and $w_0$, respectively. In $W_\lambda$, there are also three right $W_{\lambda,\Theta(1,\lambda)}$-cosets whose longest elements are $s_{\alpha+\beta}$, $s_{\alpha+\beta}s_\gamma$, and $s_{\alpha+\beta+\gamma}$, respectively. The map 
	\begin{equation*}
		(-)|_\lambda: W_\Theta \backslash W_\Theta W_\lambda \bijects W_{\lambda,\Theta(1,\lambda)} \backslash W_\lambda
	\end{equation*}
	is visualized in (\ref{diag:(-)|_lambda_A3}).
\end{example}

We also need to understand how $(-)_\lambda$ behaves under right multiplication by a non-integral simple reflection. This reflects the effect of non-integral intertwining functors which will be defined in \textsection\ref{sec:geom} and will be used in the algorithm. Roughly speaking, right multiplication by a non-integral simple reflection translates $(W_\Theta,W_\lambda)$-coset structures to $(W_\Theta,W_{s_\beta\lambda})$-coset structures, while conjugation by the same reflection translates right $W_{\lambda,\Theta(u,\lambda)}$-coset structures in $W_\lambda$ to $W_{s_\beta \lambda, \Theta(r,s_\beta \lambda)}$-coset structures in $W_{s_\beta \lambda}$.

\begin{lemma}\label{lem:Is_pres_lowest_db_coset}
	Let $u \in A_{\Theta,\lambda}$, $\beta \in \Pi - \Pi_\lambda$. Then $W_\Theta u s_\beta$ is the smallest right $W_\Theta$-coset in $W_\Theta u s_\beta W_{s_\beta \lambda} = W_\Theta u W_\lambda s_\beta$.
\end{lemma}

\begin{proof}
	By Lemma \ref{lem:non-int_refl_subsys}(e)(f), $u s_\beta \in A_{s_\beta \lambda}$. Proposition \ref{lem:unique_smallest_right_coset} says that elements in $A_{s_\beta\lambda}$ are concentrated on the smallest right $W_\Theta$-cosets. So the right coset $W_\Theta (u s_\beta)$ containing $u s_\beta$ must be the smallest in the double coset $W_\Theta( u s_\beta) W_{s_\beta \lambda}$ containing $u s_\beta$. This proves the lemma. The final identification simply follows from $s_\beta W_{s_\beta \lambda} s_\beta = W_\lambda$ by \ref{lem:non-int_refl_subsys}(d).
\end{proof}

Rephrasing slightly and using \ref{lem:unique_smallest_right_coset} again, we get

\begin{corollary}
	Let $u \in A_{\Theta,\lambda}$, $\beta \in \Pi - \Pi_\lambda$. If $r$ denotes the unique element in $ A_{\Theta,s_\beta \lambda} \cap W_\Theta us_\beta W_{s_\beta \lambda}$, then $W_\Theta r = W_\Theta u s_\beta$.
\end{corollary}


\begin{proposition}\label{lem:Is_right_coset}
	Let $u \in A_{\Theta,\lambda}$, $\beta \in \Pi - \Pi_\lambda$. Let $r$ be the unique element in $ A_{\Theta,s_\beta \lambda} \cap W_\Theta u s_\beta W_{s_\beta \lambda}$. Then conjugation by $s_\beta$ is a bijection
	\begin{equation*}
		s_\beta (-) s_\beta : W_{s_\beta \lambda, \Theta(r,s_\beta \lambda)} \backslash W_{s_\beta \lambda} \bijects
		W_{\lambda,\Theta(u,\lambda)} \backslash W_\lambda
	\end{equation*}
	that preserves the partial orders on right cosets. Moreover, the following diagram commutes
	\begin{equation}\label{eqn:Is_right_coset}
		\begin{tikzcd}
			W_{s_\beta \lambda, \Theta(r,s_\beta \lambda)} \backslash W_{s_\beta \lambda} \ar[d,"\ind_{s_\beta\lambda}"'] \ar[r,"s_\beta (-) s_\beta"] &
			W_{\lambda,\Theta(u,\lambda)} \backslash W_\lambda \ar[d,"\ind_\lambda"]\\
			W_\Theta \backslash W \ar[r," (-) s_\beta"] & W_\Theta \backslash W
		\end{tikzcd}.
	\end{equation}
	In particular, for any $C$, $D \in W_\Theta \backslash W_\Theta r W_{s_\beta \lambda}$, 
	\begin{equation*}
		D \le_{r,s_\beta \lambda}  C \iff D s_\beta \le_{u,\lambda} C s_\beta.
	\end{equation*}
\end{proposition}

\begin{proof}
	By \ref{lem:Bruhat_order_conj}, conjugation by $s_\beta$ is an isomorphism of groups and posets between $(W_{s_\beta \lambda},\le_{s_\beta\lambda})$ and $(W_\lambda,\le_\lambda)$.	By the preceding corollary, there exists $w \in W_\Theta$ such that $w r = u s_\beta$. Therefore
	\begin{align*}
		s_\beta \Theta(u,\lambda)
		&= s_\beta (u\inv \Sigma_\Theta \cap \Pi_\lambda)\\
		&= (us_\beta)\inv \Sigma_\Theta \cap s_\beta \Pi_\lambda\\
		&= (wr)\inv \Sigma_\Theta \cap \Pi_{s_\beta \lambda}\\
		&= r{}\inv (w\inv \Sigma_\Theta) \cap \Pi_{s_\beta \lambda}\\
		&= r{}\inv \Sigma_\Theta \cap \Pi_{s_\beta \lambda}\\
		&= \Theta(r,s_\beta \lambda).
	\end{align*}
	Hence conjugation by $s_\beta$ sends $W_{s_\beta \lambda, \Theta(r,s_\beta \lambda)}$ to $W_{\lambda,\Theta(u,\lambda)}$ and therefore induces a bijection from $W_{s_\beta \lambda,\Theta(r,s_\beta \lambda)} \backslash W_{s_\beta \lambda}$ to $W_{\lambda,\Theta(u,\lambda)} \backslash W_\lambda$. Furthermore, since conjugation by $s_\beta$ preserves Bruhat orders, it also preserves the partial orders on right cosets.
	
	To check that the diagram commutes, take any $D' \in W_{s_\beta \lambda, \Theta(r,s_\beta \lambda)} \backslash W_{s_\beta \lambda}$. Along the top-right path, $D'$ is sent to 
	\begin{equation*} 
		W_\Theta u \cdot s_\beta D' s_\beta = W_\Theta w r D' s_\beta = W_\Theta r D' s_\beta,
	\end{equation*}
	which agrees with the image along the bottom-left path.
\end{proof}

\subsection{A technical lemma}\label{subsec:lem_induction}

In the last part of this section, we prove a technical lemma that will be used in \textsection\ref{subsec:(4)} in induction process. 

\begin{proposition}\label{lem:decrease_of_length}
	Let $u \in A_{\Theta,\lambda}$ and $C \in W_\Theta \backslash W_\Theta u W_\lambda$. Suppose $C \neq W_\Theta u$. Then there exist $\alpha \in \Pi_\lambda$, $s \ge 0$ and $\beta_1,\ldots,\beta_s \in \Pi$ such that, writing $z_0 = 1$, $z_i = s_{\beta_1} \cdots s_{\beta_i}$ and $z = z_s$, the following conditions hold:
	\begin{enumerate}[label=(\alph*)]
		\item for any $0 \le i \le s-1$, $\beta_{i+1}$ is non-integral to $z_i\inv \lambda$;
		\item $z\inv \alpha \in \Pi \cap \Pi_{z\inv \lambda}$;
		\item $C s_\alpha <_{u,\lambda} C$;
		\item if $s > 0$, $Cz < C$;
		\item $C s_\alpha z = C z s_{z\inv \alpha} < C z$.
	\end{enumerate}
\end{proposition}

This proposition is used in showing that the $q$-polynomials defined geometrically (by taking higher inverse images to Schubert cells; see \textsection\ref{subsec:geom_idea} for an explanation of the idea) agree with the Whittaker Kazhdan-Lusztig polynomials for the triple $(W_\lambda, \Pi_\lambda, \Theta(u,\lambda))$. This is a proof by induction in the length of $C$. As mentioned in \textsection\ref{subsec:geom_idea} (specifically the condition \ref{enum:WKL_basis_U}), the Kazhdan-Lusztig basis elements $C_C = \psi_{u,\lambda}(C|_\lambda)$ of $\cH_{\Theta(u,\lambda)}$ are partly characterized by properties of the product $C_C C_s = T_\alpha^{u,\lambda}(\psi_{u,\lambda}(C|_\lambda))$. If the simple reflection $s \in W_\lambda$ happens to be simple in $W$, then multiplication by $C_s$ on $C_C$ lifts to the geometric $U$-functor (push-pull along $X \to X_s$). However, if $s$ is not simple in $W$, no such $U$-functor exists. The strategy in this situation is to use non-integral intertwining functors to translate everything so that $s$ becomes simple in both the integral Weyl group and in $W$. On the $W_\lambda$ level, these non-integral intertwining functors correspond to applying conjugations $s_{\beta_i}(-)s_{\beta_i}$ by non-integral simple reflections so that $s \in W_\lambda$ is translated to $(s_{\beta_1} \cdots s_{\beta_s})\inv s (s_{\beta_1} \cdots s_{\beta_s})$ which is simple in $W_{s_{\beta_s} \cdots s_{\beta_1} \lambda}$. On the $W$ level, they correspond to right multiplication on $C$ by $s_{\beta_1} \cdots s_{\beta_s}$. Also, one needs to ensures that the length of $C$ decreases after these non-integral reflections in order to apply the induction hypothesis on $C$. The existence of such a chain of non-integral reflections is guaranteed by the proposition.

\begin{proof}
	Since $C \neq W_\Theta u$, in particular $C \neq W_\Theta$, there exists a simple reflection $s_\gamma$ such that $C s_\gamma < C$.
	
	If there exists $\alpha \in \Pi \cap \Pi_\lambda$ such that $C s_\alpha < C$, then this $\alpha$ together with $s = 0$ satisfies the requirement: (a) and (d) are void, while (b) and (e) are true by construction. We need to verify (c). Since $s_\alpha$ is simple in $(W_\lambda, \Pi_\lambda)$, we have three mutually exclusive possibilities: $C s_\alpha <_{u,\lambda} C$, $C s_\alpha = C$, or $C s_\alpha >_{u,\lambda} C$. Since the map $\ind_\lambda$ preserves the partial order, they imply $C s_\alpha < C$, $C s_\alpha = C$ and $C s_\alpha > C$, respectively. By our choice of $\alpha$, the last two possibilities cannot happen. Hence we must have $C s_\alpha <_{u,\lambda} C$ and (c) holds.
	
	Suppose such $\alpha$ does not exist. Then any simple reflection that decreases the length of $C$ via right multiplication must be non-integral to $\lambda$. Let $s_{\beta_1}$, $\beta_1 \in \Pi - \Pi_\lambda$, be one of those. If there exists $\alpha' \in \Pi \cap \Pi_{s_{\beta_1} \lambda}$ with $C s_{\beta_1} s_{\alpha'} < C s_{\beta_1}$, we claim that $\alpha := s_{\beta_1} \alpha' \in s_{\beta_1} \Pi_{s_{\beta_1} \lambda} = \Pi_\lambda$, $s = 1$ and $\beta_1$ satisfy our requirements. (a) and (d) follows by our choice of $s_{\beta_1}$, (e) follows from the conditions on $\alpha'$. For (b),
	\begin{equation*}
		z\inv \alpha = s_{\beta_1} s_{\beta_1} \alpha' = \alpha' \in \Pi \cap \Pi_{s_{\beta_1} \lambda}
	\end{equation*}
	by definition of $z$ and $\alpha'$. For (c), arguing in the same way, we only need to rule out $C s_\alpha \ge C$, which would imply $\ell(C) - \ell(C s_\alpha s_{\beta_1}) \in \{-2,-1,0,1\}$. On the other hand,
	\begin{align*}
		C > C s_{\beta_1} > C s_{\beta_1} s_{\alpha'} 
		= C s_{\beta_1} s_{(s_{\beta_1} \alpha)}
		= C s_{\beta_1} s_{\beta_1} s_\alpha s_{\beta_1}
		= C s_\alpha s_{\beta_1}.
	\end{align*}
	So $\ell(C) - \ell(C s_\alpha s_{\beta_1}) \ge 2$ and (c) holds.
	
	If such $\alpha'$ does not exist, then we can find $\beta_2,\ldots,\beta_s \in \Pi$ such that $C z_{i+1} < C z_i$ for all $1 \le i \le s-1$ until we get to a point where there exists $\alpha'' \in \Pi \cap \Pi_{z\inv \lambda}$ with $C z s_{\alpha''} < C z$ (termination of this process is proven in the next paragraph). We claim that $\alpha := z\alpha'' \in z \Pi_{z\inv \lambda} = \Pi_\lambda$, $s$ and $\beta_1,\ldots,\beta_s$ satisfy our requirements. The verification is essentially the same as in the previous case. (a), (b), (d) and (e) are satisfied by our choice of $\beta_i$ and $\alpha''$. For (c), we have an inequality
	\begin{equation}\label{eqn:decrease_of_length_step_a}
		C z > C z s_{\alpha''} = C z s_{z\inv \alpha} = C z z\inv s_\alpha z = C s_\alpha z
	\end{equation}
	where $\ell(w^{Cz}) = \ell(w^C z) = \ell(w^C) - s$. Also $w^{C z s_{\alpha''}} = w^C z s_{\alpha''} = w^C s_\alpha z$. Hence
	\begin{equation*}
		\ell(w^C s_\alpha) 
		= \ell(w^C s_\alpha z z\inv) 
		= \ell(w^{C s_\alpha z} z\inv) 
		\le \ell(w^{C s_\alpha z}) +s
		= \ell(w^{Cz}) -1 +s
		= \ell(w^C) -1 < \ell(w^C).
	\end{equation*}
	This rules out $C s_\alpha \ge C$ and (c) is thus verified.
	
	Lastly, let us show that this process of finding $\alpha''$ must terminate no later than when we get to $\ell(w^{Cz}) = \ell(w_\Theta)+1$. That is, we show that when $\ell(w^{Cz}) = \ell(w_\Theta)+1$, such an $\alpha''$ must exist. The condition $\ell(w^{Cz}) = \ell(w_\Theta)+1$ implies $C z = W_\Theta s_\gamma > W_\Theta$ for some simple reflection $s_\gamma$. If $\gamma \in \Pi - \Pi_{z\inv \lambda}$, then $s_\gamma \in A_{z\inv \lambda}$. Also, since $W_\Theta s_\gamma > W_\Theta$, any element of $W_\Theta s_\gamma$ must have length $\ge 1$. Hence $s_\gamma$ is the shortest element of $W_\Theta s_\gamma$, i.e. $s_\gamma \in w_\Theta {}^\Theta W$. Therefore $s_\gamma \in A_{z\inv \lambda} \cap (w_\Theta {}^\Theta W) = A_{\Theta,z\inv \lambda}$. Since $C = W_\Theta s_\gamma z\inv$, by (repeatedly applying) \ref{lem:Is_pres_lowest_db_coset}, we see that $C$ is the smallest right $W_\Theta$-coset in the $(W_\Theta,W_\lambda)$-coset containing it, that is, $C = W_\Theta u$. This contradicts our assumption on $C$. Therefore $\gamma \in \Pi \cap \Pi_\lambda$, and $\alpha'' = \gamma$ satisfies our requirement for $\alpha''$. Thus the process terminates.
\end{proof}
\section{Non-integral intertwining functors}\label{sec:geom}

In this section, we give the definition of non-integral intertwining functors and show that they translate the Kazhdan-Lusztig polynomials for our Whittaker modules. Readers can review \textsection\ref{subsec:geom_prelim} for the basic geometric setup and related notations. In the rest of the paper, we will use facts about $\cD$-modules without citing references, including the distinguished triangle for immersions of a smooth closed subvariety and its complement (also known as the distinguished triangle for local cohomology), the base change theorem for $\cD$-modules, and Kashiwara's equivalence of categories for closed immersions. These facts are contained in \cite{Borel:D-mods}, IV.8.3, 8.4 and 7.11, respectively.

For any $w \in W$, let $Z_w$ denote the subset of $X \times X$ consists of pairs $(x,y)$ such that $\fb_x$ and $\fb_y$ are in relative position $w$. This means that for any common Cartan subalgebra $\fc$ and any representative of $w$ in $N_G(\fc)$ (also denoted by $w$), $\fb_x = \Ad w(\fb_y)$. If $w$ is fixed, we write
\begin{equation*}
	X \xleftarrow{p_1} Z_w \xrightarrow{p_2} X
\end{equation*}
for the two projections. For an integral weight $\mu \in \fh^*$, write $\cO_X(\mu)$ for the $G$-equivariant line bundle on $X$ where the $\fb$-action on the geometric fiber at $x_\fb \in X$ (the point on $X$ that corresponds to $\fb$) is given by $\mu$.

\begin{definition}
	For $w \in W$ and $\lambda \in \fh^*$, the \textbf{intertwining functor} $LI_w$ is defined to be 
	\begin{align*}
		LI_w : D^b(\cD_\lambda) &\to D^b(\cD_{w\lambda}),\notag\\
		\cF^\bullet &\mapsto p_{1+} \big( p_1^* \cO_X(\rho-w\rho) \dotimes_{\cO_{Z_w}} p_2^+ \cF^\bullet \big)\notag\\
		&\qquad \cong \cO_X(\rho-w\rho) \dotimes_{\cO_X} p_{1+} p_2^+ \cF^\bullet.
	\end{align*}
	Write $I_w$ for $H^0 LI_w$. It is shown in \cite[L.3]{Milicic:Localization} that $LI_w$ is the left derived functor of $I_w$.
\end{definition}

For properties of intertwining functors readers can refer to \textit{loc. cit.} The main property we will use is 

\begin{theorem}[{\cite[Chapter 3 Corollary 3.22]{Milicic:Localization}}]\label{lem:non-int_I}
	If $\beta \in \Pi - \Pi_\lambda$, then $I_{s_\beta}$ is an equivalence of categories
	\begin{equation*}
		I_{s_\beta} : \Mod_{qc}(\cD_\lambda) \cong \Mod_{qc}(\cD_{s_\beta\lambda})
	\end{equation*}
	whose quasi-inverse is $I_{s_\beta}$. 
\end{theorem}

To use these functors for our purpose, we need to compute the action of intertwining functors on standard and irreducible modules. Romanov computed the following result for $Cs_\beta > C$. The main ingredients of the proof there are base change formula and projection formula for $\cD$-modules.

\begin{proposition}[{\cite[\textsection 3.4 Proposition 5]{Romanov:Whittaker}}]\label{lem:I_moves_std_general}
	Let $\beta \in \Pi$ and $C \in W_\Theta \backslash W$ such that $C s_\beta > C$. Then for any $\lambda \in \fh^*$,
	\begin{equation*}
		LI_{s_\beta} \cI(w^C,\lambda,\eta) = \cI(w^C s_\beta, s_\beta \lambda,\eta).
	\end{equation*}
\end{proposition}

Combined with \ref{lem:non-int_I} we get

\begin{corollary}\label{lem:I_moves_std}
	Let $\beta \in \Pi - \Pi_\lambda$ and $C \in W_\Theta \backslash W$ such that $C s_\beta \neq C$. Then 
	\begin{equation*}
		I_{s_\beta} \cI(w^C,\lambda,\eta) = \cI(w^C s_\beta, s_\beta \lambda, \eta).
	\end{equation*}
\end{corollary}

\begin{proof}
	Suppose $C s_\beta > C$, then the statement follows from \ref{lem:I_moves_std_general}. But since $I_{s_\beta}$ is an equivalence of categories with inverse $I_{s_\beta}$,
	\begin{equation*}
		\cI(w^C,\lambda,\eta) = I_{s_\beta} \cI(w^C s_\beta, s_\beta \lambda, \eta).
	\end{equation*}
\end{proof}

It remains to consider the case $C s_\beta = C$. 

For a simple root $\beta$, write $X_\beta$ for the partial flag variety of type $\beta$, and write $p_\beta: X \to X_\beta$ for the natural projection. This is a Zariski-local $\BP^1$-fibration. $x$ and $y$ are contained in the same $p_\beta$-fiber (i.e. $p_\beta(x) = p_\beta(y)$) if and only if $\fb_x$ and $\fb_y$ are in relative position $1$ or $s_\beta$.

\begin{lemma}\label{lem:S}
	Let $C \in W_\Theta \backslash W$ and $\beta \in \Pi$ such that $C s_\beta = C$. Set
	\begin{equation*}
		S = \{ (x,y) \in C(w^C) \times C(w^C) \mid \fb_x \text{ and } \fb_y \text{ are in relative position } s_\beta\} \subset Z_{s_\beta}.
	\end{equation*}
	Write $C(w^C) \xleftarrow{p_1|_S} S \xrightarrow{p_2|_S} C(w^C)$	for the projections. Then
	\begin{equation*}
		(p_1|_S)_+ (p_2|_S)^+ \cO_{C(w^C)}^\eta = \cO_{C(w^C)}^\eta.
	\end{equation*}
\end{lemma}

\begin{proof}
	For convenience, write $w = w^C$, $p_1 = p_1|_S$ and $p_2 = p_2|_S$. Set
	\begin{equation*}
		S' = C(w) \times_{p_\beta(C(w))} C(w) = \{ (x,y) \in C(w) \times C(w) \mid p_\beta(x) = p_\beta(y) \}.
	\end{equation*}
	Then $S \subset S \cup \Delta_{C(w)} = S' \subset Z_{s_\beta}$, where $\Delta_{C(w)}$ denotes the diagonal. Write $C(w) \xleftarrow{q_1} S' \xrightarrow{q_2} C(w)$	for the projections, and $\Delta_{C(w)} \xrightarrow{i_\Delta} S' \xleftarrow{i_S} S$ for the inclusions. Then $i_\Delta$ is a closed immersion with relative dimension $1$, and $i_S$ is open. We have the following diagram
	\begin{equation}\label{diag:S}
		\begin{tikzcd}
			S \ar[dr, "i_S", crossing over] \ar[ddr, bend right, "p_1"'] \ar[rrd, bend left, "p_2"] &[-5ex] \\[-2ex]
			& S' \ar[d, "q_1"] \ar[r, "q_2"] & C(w) \ar[d, "p_\beta"]\\
			& C(w) \ar[r, "p_\beta"] & p_\beta(C(w))
		\end{tikzcd}
	\end{equation}
	where the bottom-right square is Cartesian.
	
	Applying the triangle for local cohomology to $q_2^+ \cO_{C(w)}^\eta$, we get
	\begin{equation*}
		\dtri{ i_{\Delta+} i_\Delta^! q_2^+ \cO_{C(w)}^\eta }%
		{ q_2^+ \cO_{C(w)}^\eta }%
		{ i_{S+} i_S^+ q_2^+ \cO_{C(w)}^\eta}.
	\end{equation*}
	Applying $q_{1+}$, we get 
	\begin{equation*}
		\dtri{ q_{1+} i_{\Delta+} i_\Delta^! q_2^+ \cO_{C(w)}^\eta }%
		{ q_{1+} q_2^+ \cO_{C(w)}^\eta }%
		{ q_{1+} i_{S+} i_S^+ q_2^+ \cO_{C(w)}^\eta}.
	\end{equation*}
	Applying base change to the bottom-right square in (\ref{diag:S}), $q_{1+} q_2^+ \cO_{C(w)} \cong p_\beta^+ p_{\beta+} \cO_{C(w)}^\eta$. Here $p_{\beta+} \cO_{C(w)}^\eta$ is an $\eta$-twisted Harish-Chandra sheaf on $p_\beta(C(w))$. But $p_\beta(C(w))$ is isomorphic to $C(w s_\beta)$ as an $N$-variety via $p_\beta$, and since $w s_\beta$ is not the longest element in $W_\Theta w s_\beta = W_\Theta w$, we know there is no $\eta$-twisted Harish-Chandra sheaf on $C(w s_\beta)$ except $0$. Hence $p_{\beta+} \cO_{C(w)}^\eta = 0$ and thus $q_{1+} q_2^+ \cO_{C(w)} = 0$. As a result,
	\begin{equation*}
		q_{1+} i_{S+} i_S^+ q_2^+ \cO_{C(w)}^\eta = q_{1+} i_{\Delta+} i_\Delta^! q_2^+ \cO_{C(w)}^\eta [1].
	\end{equation*}
	The left side simplifies to $p_{1+} p_2^+ \cO_{C(w)}^\eta$. For the right side, $q_{1+} i_{\Delta+} = (q_1 \circ i_\Delta)_+$ and $q_1 \circ i_\Delta$ is the projection $\Delta_{C(w)} \to C(w)$ along the first coordinate which is an $N$-equivariant isomorphism. Moreover, 
	\begin{align*}
		i_\Delta^! q_2^+ \cO_{C(w)}^\eta [1]
		&= i_\Delta^+ q_2^+ \cO_{C(w)}^\eta [1] [-1]\\
		&= (q_2 \circ i_\Delta)^+ \cO_{C(w)}^\eta,
	\end{align*}
	and $q_2 \circ i_\Delta$ is the projection $\Delta_{C(w)} \to C(w)$ along the second coordinate, also an $N$-equivariant isomorphism. Thus
	\begin{equation*}
		p_{1+} p_2^+ \cO_{C(w)}^\eta
		= (q_1 \circ i_\Delta)_+ (q_2 \circ i_\Delta)^+ \cO_{C(w)}^\eta 
		= \cO_{C(w)}^\eta. \qedhere
	\end{equation*}
\end{proof}

\begin{lemma}\label{lem:two_cells}
	Let $s_\beta \in \Pi$ and $C \in W_\Theta \backslash W$ such that $C s_\beta = C$. Write $\iota: C(w^C) \inj C(w^C) \cup C(w^C s_\beta)$ for the inclusion. Then for any $\cF \in \Mod_{coh}(\cD_{C(w^C) \cup C(w^C s_\beta)},N,\eta)$, 
	\begin{equation*}
		\cF = \iota_+ \iota^! \cF = ( \iota_{+} \cO_{C(w^C)}^\eta )^{\oplus \rank \iota^! \cF}
	\end{equation*}
	where $\rank$ stands for the rank as a free $\cO$-module.
\end{lemma}

\begin{proof}
	Write $w = w^C$. The assumption implies that $w s_\beta \in C$, $w s_\beta < w$, and that $C(w)$ and $C(w s_\beta)$ are open and closed in $C(w) \cup C(w s_\beta)$, respectively.
	
	Since the category of $\eta$-twisted Harish-Chandra sheaves on $C(w)$ is semisimple, $\iota^! \cF$ is a direct sum of copies of $\cO_{C(w)}^\eta$. This implies the second equality. For the first equality, adjunction gives a map
	\begin{equation}\label{eqn:iota_w_adjunction}
		\cF \to \iota_{+} \iota^! \cF
	\end{equation}
	whose kernel and cokernel are supported on $C(w s_\beta)$, which are equal to direct images of $\eta$-twisted Harish-Chandra sheaves on $C(w s_\beta)$ by Kashiwara's equivalence. But $w s_\beta$ is not the longest element in $C$, so there is no such module on $C(w s_\beta)$ except zero. Hence (\ref{eqn:iota_w_adjunction}) is an isomorphism, which establishes the first equality.
\end{proof}

\begin{proposition}\label{lem:I_fixes_std}
	Let $C \in W_\Theta \backslash W$, $\beta \in \Pi$ such that $C s_\beta = C$. Then
	\begin{equation*}
		LI_{s_\beta} \cI(w^C,\lambda,\eta) = \cI(w^C, s_\beta \lambda,\eta).
	\end{equation*}
\end{proposition}

\begin{proof}
	Write $w = w^C$. Let
	\begin{equation*}
		F = Z_{s_\beta} \times_{p_2,X,i_w} C(w) = \{ (x,y) \in X \times C(w) \mid \fb_x \text{ and } \fb_y \text{ are in relative position } s_\beta \}.
	\end{equation*}
	and let $S$ be as in \ref{lem:S}. Then $S$ is a subvariety of $F$. It's easy to see that
	\begin{equation*}
		p_1(F) = \{x \in X \mid \thereis y \in C(w) \text{ such that } \fb_x \text{ and } \fb_y \text{ are in relative position } s_\beta \} = C(w) \cup C(w s_\beta).
	\end{equation*}	
	So we have the following diagram
	\begin{equation}\label{diag:SF}
		\begin{tikzcd}
			&[-5ex] S \ar[dl,"p_1|_S"'] \ar[dr, "a_S"] & [-4ex]\\
			C(w) \ar[dr, "\iota_w"] \ar[ddrr, bend right, distance=3cm, "i_w"'] & & F \ar[dl, "p_1|_F"'] \ar[dr, "i_F"'] \ar[rr, equal] & & F \ar[dl, "i_F"] \ar[dr, "p_2|_F"]\\
			& C(w) \cup C(w s_\beta) \ar[dr, "j_w"] && Z_{s_\beta} \ar[dl, "p_1"'] \ar[dr, "p_2"] && C(w) \ar[dl, "i_w"']\\
			&& X && X
		\end{tikzcd}
	\end{equation}
	The right-most square is Cartesian by definition of $F$. The top-left square is also Cartesian, i.e. $S$ is the preimage of $C(w)$ along $p_1: F \to C(w) \cup C(w s_\beta)$. By definition of intertwining functors and base change,
	\begin{align}
		LI_{s_\beta} \cI(w,\lambda,\eta)
		&= \cO_X(\rho - s_\beta \rho) \dotimes_{\cO_X} p_{1+} p_2^+ i_{w+} \cO_{C(w)}^\eta \nonumber\\
		&= \cO_X(\rho - s_\beta \rho) \dotimes_{\cO_X} p_{1+} i_{F+} (p_2|_F)^+ \cO_{C(w)}^\eta \nonumber\\
		&= \cO_X(\rho - s_\beta \rho) \dotimes_{\cO_X} j_{w+} (p_1|_F)_+ (p_2|_F)^+ \cO_{C(w)}^\eta. \label{eqn:LI(std)_step_a}
	\end{align}	
	We claim that $(p_1|_F)_+ (p_2|_F)^+ \cO_{C(w)}^\eta = \iota_{w+} \cO_{C(w)}^\eta$.
	By \ref{lem:two_cells}, 
	\begin{equation*}
		(p_1|_F)_+ (p_2|_F)^+ \cO_{C(w)}^\eta = \iota_{w+} \iota_w^! (p_1|_F)_+ (p_2|_F)^+ \cO_{C(w)}^\eta.
	\end{equation*}		
	Apply base change using the top-left square in (\ref{diag:SF}), 
	\begin{equation*}
		\iota_w^! (p_1|_F)_+ (p_2|_F)^+ \cO_{C(w)}^\eta 
		= (p_1|_S)_+ a_S^! (p_2|_F)^+ \cO_{C(w)}^\eta
		= (p_1|_S)_+ a_S^+ (p_2|_F)^+ \cO_{C(w)}^\eta
	\end{equation*}
	Note that $p_2|_F \circ a_S = p_2|_S$. Hence, by \ref{lem:S}, the sheaf in the above equation is isomorphic to $\cO_{C(w)}^\eta$. Therefore
	\begin{equation*}
		(p_1|_F)_+ (p_2|_F)^+ \cO_{C(w)}^\eta = \iota_{w+} \iota_w^! (p_1|_F)_+ (p_2|_F)^+ \cO_{C(w)}^\eta = \iota_{w+} \cO_{C(w)}^\eta
	\end{equation*}
	as claimed. As a result,
	\begin{align*}
		(\ref{eqn:LI(std)_step_a}) 
		&= \cO_X(\rho - s_\beta \rho) \dotimes_{\cO_X} j_{w+} \iota_{w+} \cO_{C(w)}^\eta\\
		&= \cO_X(\rho - s_\beta \rho) \dotimes_{\cO_X} i_{w+} \cO_{C(w)}^\eta\\
		&= \cI(w, s_\beta \lambda, \eta)	
	\end{align*}
	which proves the proposition.
\end{proof}

\begin{corollary}\label{lem:I_on_std}
	Let $\beta \in \Pi - \Pi_\lambda$. Let $C \in W_\Theta \backslash W$. Then
	\begin{align*}
		I_{s_\beta} \cI(w^C,\lambda,\eta) &= \cI(w^{C s_\beta}, s_\beta\lambda, \eta),\\
		I_{s_\beta} \cL(w^C,\lambda,\eta) &= \cL(w^{C s_\beta}, s_\beta\lambda, \eta)
	\end{align*}
	(note that we have $w^{C s_\beta}$ instead of $w^C s_\beta$ on the right hand sides).
\end{corollary}

\begin{proof}
	The statement about standard modules is the combination of \ref{lem:I_moves_std} and \ref{lem:I_fixes_std}. Since $I_{s_\beta}$ is an equivalence of categories, it must send the unique irreducible submodule of $\cI(w^C,\lambda,\eta)$ to the unique irreducible submodule of $\cI(w^{C s_\beta}, s_\beta\lambda, \eta)$, i.e. it must send $\cL(w^C,\lambda,\eta)$ to $\cL(w^{C s_\beta}, s_\beta\lambda, \eta)$.
\end{proof}

Next, we show that non-integral intertwining functors also preserves pullback of irreducible modules to strata. 

\begin{proposition}\label{thm:Is_pullback}
	Let $\beta \in \Pi - \Pi_\lambda$, $C,D \in W_\Theta \backslash W$ and $p \in \BZ$. Then
	\begin{equation*}
		\rank H^p i_{w^D}^! \cL(w^C, \lambda,\eta) = \rank H^p i_{w^{D s_\beta}}^! \cL(w^{C s_\beta}, s_\beta \lambda, \eta).
	\end{equation*}
\end{proposition}

The proof we give below uses the same tools as in the previous proposition. There is an alternative proof which we briefly mention. One shows that $\rank H^p i_{w^D}^! \cL(w^C, \lambda,\eta)$ is the same as the dimension of the $p$-th $\cD_\lambda$-module $\Ext$ group of $\cM(w^C,\lambda,\eta)$ and $\cL(w^C,\lambda,\eta)$ using facts on derived categories of highest weight categories (Brown-Romanov \cite[Theorem 7.2]{Brown-Romanov:Whittaker-Verma-pairing} showed that $\Mod_{coh}(\cD_\lambda,N,\eta)$ is a highest weight category). The proposition follows from the fact that $I_{s_\beta}$ is an equivalence of categories and induces an isomorphism on $\Ext$-groups.

\begin{proof}
	Write $w = w^D$. 
	
	There are two cases, $D s_\beta \neq D$ or $D s_\beta = D$. Consider the first case. Assume $D s_\beta < D$. Then $w^{D s_\beta} = w^D s_\beta = w s_\beta$. Let
	\begin{equation*}
		F = C(w s_\beta) \times_{i_{w s_\beta}, X, p_1} Z_{s_\beta} = \{ (x,y) \in C(w s_\beta) \times X \mid \fb_x \text{ and } \fb_y \text{ are in relative position } s_\beta \}.
	\end{equation*}
	Then the second projection $p_2|_F: F \to X$ induces an isomorphism of $F$ onto $C(w)$, and we have the following commuting diagram
	\begin{equation*}
		\begin{tikzcd}
			&[-1ex] F \ar[dl, "p_1|_F"'] \ar[dr, "i_F"'] \ar[rr, equal] & & F \ar[dl,"i_F"] \ar[dr, "p_2|_F", "\cong"']\\
			C(w s_\beta) \ar[dr, "i_{w s_\beta}"'] && Z_{s_\beta} \ar[dl, "p_1"'] \ar[dr, "p_2"] && C(w) \ar[dl, "i_w"]\\
			& X && X
		\end{tikzcd}
	\end{equation*}
	where the left square is Cartesian. Using base change, 
	\begin{align}
		\rank H^p i_{w^{D s_\beta}}^! \cL(w^{C s_\beta} ,s_\beta \lambda, \eta) 
		&= \rank H^p i_{w s_\beta}^! I_{s_\beta} \cL(w^C, \lambda,\eta) \nonumber\\
		&= \rank H^p i_{w s_\beta}^! p_{1+} p_2^+ \cL(w^C, \lambda,\eta) \label{eqn:rk_pullback_step_a}\\
		&= \rank H^p (p_1|_F)_+ (p_2|_F)^! i_w^! \cL(w^C, \lambda,\eta) [-1].\label{eqn:rk_pullback_step_b}
	\end{align}
	\sloppy Here in (\ref{eqn:rk_pullback_step_a}) we did not write the twist by the line bundle $\cO_X(\rho-w\rho)$ because there is no twist on $C(w s_\beta)$. Since $\Mod_{coh}(\cD_{C(w)},N,\eta)$ is semisimple, $i_w^! \cL(w^C, \lambda,\eta)$ is a direct sum of $\cO_{C(w)}^\eta$ at different degrees. So $(p_2|_F)^! i_w^! \cL(w^C, \lambda,\eta)$ is a direct sum of $\cO_F^\eta$ at different degrees by the fact that $p_2|_F$ is an isomorphism, and the rank at degree $p$ being $\rank H^p i_w^! \cL(w^C, \lambda,\eta)$. Hence it is enough to compute $(p_1|_F)_+ \cO_F^\eta$, for which we use the fact that a map of homogeneous spaces of a unipotent group is isomorphic to a coordinate projection of affine spaces, that is, we have the following commutative diagram where all maps are $N$-equivariant, for some $N$-actions on $\BA^1 \times \BA^{\ell(w s_\beta)}$ and $\BA^{\ell(w s_\beta)}$:
	\begin{equation*}
		\begin{tikzcd}
			F \ar[d,"\cong"'] \ar[r, "p_1|_F"] & C(w s_\beta) \ar[d,"\cong"]\\
			\BA^1 \times \BA^{\ell(w s_\beta)} \ar[r,"pr_1"] & \BA^{\ell(w s_\beta)}
		\end{tikzcd}.
	\end{equation*}
	So it suffices to compute $pr_{1+} \cO_{\BA^1 \times \BA^{\ell(w s_\beta)}}^\eta$. Since $pr_1$ is a coordinate projection, $pr_1^+ \cO_{\BA^{\ell(w s_\beta)}}^\eta = \cO_{\BA^1} \boxtimes \cO_{\BA^{\ell(w s_\beta)}}^\eta$ (we remark that, without the assumption of $D s_\beta \neq D$, $w s_\beta$ and $w$ can be in the same right $W_\Theta$-coset, in which case $\cO_{\BA^{\ell(w s_\beta)}}^\eta$ does not exist). On the other hand, by functoriality of $\eta$-twisted Harish-Chandra sheaves, we must have $pr_1^+ \cO_{\BA^{\ell(w s_\beta)}}^\eta = \cO_{\BA^1 \times \BA^{\ell(w s_\beta)}}^\eta$. We conclude that
	\begin{equation*}
		\cO_{\BA^1 \times \BA^{\ell(w s_\beta)}}^\eta = \cO_{\BA^1} \boxtimes \cO_{\BA^{\ell(w s_\beta)}}^\eta.
	\end{equation*}
	As a result, writing $p: \BA^1 \to \{*\}$ for the unique morphism to a point, 
	\begin{align*}
		pr_{1+} \cO_{\BA^1 \times \BA^{\ell(w s_\beta)}}^\eta
		&= (p_+ \cO_{\BA^1}) \boxtimes \big( (\operatorname{Id}_{\BA^{\ell(w s_\beta)}})_+ \cO_{\BA^{\ell(w s_\beta)}}^\eta \big) \\
		&= \BC[1] \boxtimes \cO_{\BA^{\ell(w s_\beta)}}^\eta\\
		&= \cO_{\BA^{\ell(w s_\beta)}}[1].
	\end{align*}
	Therefore $(p_1|_F)_+ \cO_F^\eta = \cO_{C(w s_\beta)}^\eta [1]$ and hence
	\begin{align*}
		\rank H^p i_{w^{D s_\beta}}^! \cL(w^{C s_\beta} ,s_\beta \lambda, \eta) 
		= (\ref{eqn:rk_pullback_step_b})
		= \rank H^p i_w^! \cL(w^C, \lambda,\eta).
	\end{align*}
	
	Now consider the case $D s_\beta = D$. In this case $w^{D s_\beta} = w^D = w$. Set
	\begin{equation*}
		F = C(w) \times_{i_w, X, p_1} Z_{s_\beta} 
		= \{ (x,y) \in C(w) \times X \mid \fb_x \text{ and } \fb_y \text{ are in relative position } s_\beta \}
	\end{equation*}
	and set $S$ as in \ref{lem:S}, viewed as a subvariety of $F$. Then the following diagram commutes
	\begin{equation*} 
		\begin{tikzcd}
			&&&&[-4ex] S \ar[dl, "b_S"'] \ar[dr, "p_2|_S"] &[-5ex]\\
			& F \ar[dl, "p_1|_F"'] \ar[dr, "i_F"'] \ar[rr,equal] && F \ar[dl, "i_F"] \ar[dr, "p_2|_F"] && C(w) \ar[dl, "\iota_w"'] \ar[ddll, bend left, distance=3cm,"i_w"]\\
			C(w) \ar[dr, "i_w"'] && Z_{s_\beta} \ar[dl, "p_1"'] \ar[dr, "p_2"] && C(w) \cup C(w s_\beta) \ar[dl, "j_w"]\\
			& X && X
		\end{tikzcd}
	\end{equation*} 
	where the left-most square and the top-right square are Cartesian. Using base change, 
	\begin{align}
		i_{w^{D s_\beta}}^! \cL(w^{C s_\beta} ,s_\beta \lambda, \eta)
		&= i_w^! I_{s_\beta} \cL(w^C, \lambda,\eta) \nonumber\\
		&= (p_1|_F)_+ (p_2|_F)^! j_w^! \cL(w^C, \lambda,\eta) [-1].\label{eqn:rk_pullback_step_d}
	\end{align}
	By \ref{lem:two_cells}, $j_w^! \cL(w^C, \lambda,\eta) = \iota_{w+} \iota_w^! j_w^! \cL(w^C, \lambda,\eta)$.	Hence
	\begin{align}
		(\ref{eqn:rk_pullback_step_d})
		&= (p_1|_F)_+ (p_2|_F)^! \iota_{w+} \iota_w^! j_w^! \cL(w^C, \lambda,\eta) [-1] \nonumber\\
		&= (p_1|_F)_+ b_{S+} (p_2|_S)^+ i_w^! \cL(w^C, \lambda,\eta). \label{eqn:rk_pullback_step_e}
	\end{align}
	Here $p_1|_F \circ b_S = p_1|_S$. Also $i_w^! \cL(w^C, \lambda,\eta)$ is a direct sum of $\cO_{C(w)}^\eta$ in various degrees. Hence by \ref{lem:S}, 
	\begin{equation*}
		\rank H^p i_{w^{D s_\beta}}^! \cL(w^{C s_\beta} ,s_\beta \lambda, \eta)
		= \rank H^p (\ref{eqn:rk_pullback_step_e})\\
		= \rank H^p i_w^! \cL(w^C, \lambda,\eta). \qedhere
	\end{equation*}
\end{proof}
\section{Main algorithm}\label{sec:KL}

In this section, we formulate an algorithm for computing a set of polynomials in $q$ indexed by pairs of right $W_\Theta$-cosets whose evaluation at $q = -1$ lead to the character formula for irreducible modules. This is in the same spirit as ordinary Kazhdan-Lusztig polynomials for the highest weight category. The algorithm we will prove is suggested by Mili\v ci\'c.

In \textsection\ref{subsec:KL_poly}, we define the Whittaker Kazhdan-Lusztig polynomials, the module $\cH_\Theta$, and related notations. The statement of the algorithm is contained in \textsection\ref{subsec:algorithm}. Proof of the algorithm is divided into subsections that follow.

\subsection{Whittaker Kazhdan-Lusztig polynomials}\label{subsec:KL_poly}

Recall the sets $A_{\Theta,\lambda} = A_\lambda \cap (w_\Theta {}^\Theta W)$ and ${\Theta(u,\lambda)} \subseteq \Pi_\lambda$ defined in \textsection\ref{subsec:db_coset_xsec} and \textsection\ref{subsec:int_model}. Recall also that we have a partial order on $W_\Theta \backslash W$ inherited from the Bruhat order on ${}^\Theta W$, denoted by $\le$ (\textsection \ref{subsec:WTheta_prelim}). Similarly, we have a partial order on $W_{\lambda,{\Theta(u,\lambda)}} \backslash W_\lambda$ which we denote by $\le_{u,\lambda}$. 

Let $\cH_\Theta$ be the free $\BZ[q,q\inv]$-modules with basis $\delta_C$, $C \in W_\Theta \backslash W$. For any $\alpha \in \Pi$, define a $\BZ[q,q\inv]$-linear operator on $\cH_\Theta$ by
\begin{equation*}
	T_\alpha (\delta_{C}) =
	\begin{cases}
		q \delta_{C} + \delta_{C s_\alpha} & \text{if } C s_\alpha > C;\\
		0 & \text{if } C s_\alpha = C;\\
		q\inv \delta_{C} + \delta_{C s_\alpha} & \text{if } C s_\alpha < C.
	\end{cases}
\end{equation*}
The operators turn $\cH_\Theta$ into a (right) module of the Hecke algebra $\cH$, isomorphic to the \textit{antispherical module}, where $T_\alpha$ encodes the action of the Kazhdan-Lusztig basis element $C_{s_\alpha} \in \cH$ \cite[\textsection 6.2]{Romanov:Whittaker}. This interpretation is not needed for us. We refer the reader to \cite[\textsection 6.1]{Romanov:Whittaker} for a precise definition of $\cH$, and to the introduction of this paper \textsection\ref{subsec:geom_idea} for an explanation of the role of $\cH$ in the Kazhdan-Lusztig algorithm.

For an element $u$ in $A_{\Theta,\lambda}$, let $\cH_{\Theta(u,\lambda)}$ be the free $\BZ[q,q\inv]$-module with basis $\delta_{E}$, $E \in W_{\lambda,{\Theta(u,\lambda)}} \backslash W_\lambda$. For any $\alpha \in \Pi_\lambda$ we define the operator $T_\alpha^{u,\lambda}$ in the same way as $T_\alpha$:
\begin{equation*}
	T_\alpha^{u,\lambda} (\delta_E) =
	\begin{cases}
		q \delta_E + \delta_{E s_\alpha} & \text{if } E s_\alpha >_{u,\lambda} E;\\
		0 & \text{if } E s_\alpha = E;\\
		q\inv \delta_E + \delta_{E s_\alpha} & \text{if } E s_\alpha <_{u,\lambda} E.
	\end{cases}
\end{equation*}
Then $\cH_{\Theta(u,\lambda)}$ becomes a right module of the Hecke algebra $\cH_\lambda = \cH(W_\lambda)$ of the integral Weyl group.

We will use a left action of $W$ on $\cH_\Theta$ defined by $w \cdot \delta_C = \delta_{wC}$. Similarly, a right action of $W$ on $\cH_\Theta$ is defined by $\delta_C \cdot w = \delta_{C w}$. We will simply write $w \delta_C$, $\delta_C w$ for the actions, omitting the dots. $w(-)w\inv$ then denotes the simultaneous action of $w$ on the left and $w\inv$ on the right. By \ref{lem:Is_right_coset}, $s_\beta (-) s_\beta$ defines a bijection
\begin{equation*}
	s_\beta(-) s_\beta : W_{\lambda,\Theta(u,\lambda)} \backslash W_\lambda \bijects W_{s_\beta \lambda, \Theta(r, s_\beta\lambda)} \backslash W_{s_\beta \lambda}
\end{equation*}
where $r$ is the unique element in $A_{\Theta,s_\beta \lambda} \cap W_\Theta u s_\beta W_{s_\beta \lambda}$. We extend this to an isomorphism
\begin{equation*}
	s_\beta (-) s_\beta : \cH_{\Theta(u,\lambda)} \bijects \cH_{\Theta(r, s_\beta\lambda)},\quad
	\delta_{E} \mapsto \delta_{s_\beta E s_\beta}.
\end{equation*}

Recall that we have a bijection
\begin{equation*}
	(-)|_\lambda : W_\Theta \backslash W \to \bigcup_{u \in A_{\Theta,\lambda}} W_{\lambda, {\Theta(u,\lambda)}} \backslash W_\lambda
\end{equation*}
defined in \ref{not:right_coset_partition}. We extend $(-)|_\lambda$ to a map
\begin{equation*}
	(-)|_\lambda: \cH_\Theta \bijects \bigoplus_{u \in A_{\Theta,\lambda}} \cH_{\Theta(u,\lambda)},\quad
	\delta_C \mapsto \delta_{C|_\lambda}.
\end{equation*}

The following theorem, proven in \cite[Theorem 11]{Romanov:Whittaker}, defines a set of polynomials indexed by pairs of right cosets called \textit{Whittaker Kazhdan-Lusztig polynomials}. 
For a right coset $E \in W_\Theta \backslash W$, we write $(W_\Theta \backslash W)_{\le E}$ for the set of those cosets $F$ such that $F \le E$.

\begin{deftheorem}[Whittaker Kazhdan-Lusztig polynomials for $(W,\Pi,\Theta)$]\label{def:parabolic_KL_poly_Theta}
	For any $E \in W_\Theta \backslash W$, there exists a unique set of polynomials $\{P_{CD}\} \subset q \BZ[q]$ indexed by
	\begin{equation*} 
		\{(C,D) \mid C,D \in (W_\Theta \backslash W)_{\le E} ; D < C\}
	\end{equation*}
	such that the function
	\begin{equation*}
		\psi: (W_\Theta \backslash W)_{\le E} \aro \cH_\Theta, \quad
		C \mapsto \delta_{C} + \sum_{D < C} P_{CD} \delta_{D}
	\end{equation*}
	satisfies the following property: for any $C \in W_\Theta \backslash W$ with $C \neq W_\Theta$, there exist $\alpha \in \Pi$ and $c_{D} \in \BZ$ such that $C s_\alpha < C$ and
	\begin{equation*}
		T_\alpha (\psi(C s_\alpha)) = \sum_{D \le C} c_{D} \psi(D).
	\end{equation*}	
	Moreover, the polynomials $P_{CD}$ do not depend on the choice of $E$. The polynomials $P_{CD}$ are called \textbf{Whittaker Kazhdan-Lusztig polynomials}, and the elements $\psi(F)$ are called \textbf{Kazhdan-Lusztig basis elements} of $\cH_\Theta$.\footnote{Romanov actually denotes the map by $\varphi$. We reserve the notation $\varphi$ to be used in the main algorithm \ref{thm:KL_alg}.}
\end{deftheorem}

It is verified in \cite[\textsection 6.3 Remark 4]{Romanov:Whittaker} that the Whittaker Kazhdan-Lusztig polynomials $P_{CD}$ agree with the parabolic Kazhdan-Lusztig polynomials $n_{y,x}$ in \cite[\textsection 3 Remark 3.2]{Soergel:KL_Tilting} for $x = w_\Theta w^C$ and $y = w_\Theta w^D$ (recall that $w_\Theta$ is the longest element in $W_\Theta$). \cite[\textsection 6]{Romanov:Whittaker} also contains comparisons of $P_{CD}$ with various polynomials defined in other sources.

We apply the same definition to $(W_\lambda,\Pi_\lambda,\Theta(u,\lambda))$:

\begin{deftheorem}[Whittaker Kazhdan-Lusztig polynomials for $(W_\lambda,\Pi_\lambda,\Theta(u,\lambda))$]\label{def:parabolic_KL_poly}
	For any $E \in W_{\lambda,{\Theta(u,\lambda)}} \backslash W_\lambda$, there exists a unique set of polynomials $\{P_{FG}^{u,\lambda}\} \subset q \BZ[q]$ indexed by
	\begin{equation*} 
		\{(F,G) \mid F,G \in (W_{\lambda,\Theta(u,\lambda)} \backslash W_\lambda)_{\le_{u,\lambda} E} ; G <_{u,\lambda} F\}
	\end{equation*}
	such that the function
	\begin{equation*}
		\psi_{u,\lambda}: (W_{\lambda,{\Theta(u,\lambda)}} \backslash W_\lambda)_{\le_{u,\lambda} E} \aro \cH_{\Theta(u,\lambda)},\quad
		F \mapsto \delta_{F} + \sum_{G <_{u,\lambda} F} P_{FG}^{u,\lambda} \delta_{G}
	\end{equation*}
	satisfies the following property: for any $F \in W_{\lambda,{\Theta(u,\lambda)}} \backslash W_\lambda$ with $F \neq W_{\lambda,\Theta(u,\lambda)}$, there exist $\alpha \in \Pi_\lambda$ and $c_{G} \in \BZ$ such that $F s_\alpha <_{u,\lambda} F$ and
	\begin{equation}\label{eqn:parabolic_KL_poly}
		T_\alpha^{u,\lambda} (\psi_{u,\lambda}(F s_\alpha)) = \sum_{G \le_{u,\lambda} F} c_{G} \psi_{u,\lambda}(G).
	\end{equation}	
	Moreover, the polynomials $P_{FG}^{u,\lambda}$ do not depend on the choice of $E$. The polynomials $P_{FG}^{u,\lambda}$ are called \textbf{Whittaker Kazhdan-Lusztig polynomials}, and the elements $\psi_{u,\lambda}(F)$ are called \textbf{Kazhdan-Lusztig basis elements} of $\cH_{\Theta(u,\lambda)}$.
\end{deftheorem}

We will write $P_{CD}^{u,\lambda}$ instead of $P_{C|_\lambda, D|_\lambda}^{u,\lambda}$ for convenience. We set $P_{EE}^{u,\lambda} = 1$ for all $E$.

If we apply the above definitions to the special case $\eta = 0$, we recover the ordinary Kazhdan-Lusztig polynomials $P_{wv}$ for $W$ and $P_{wv}^\lambda$ for $W_\lambda$, respectively. They are related to the polynomials $P_{v,w}$ in \cite{Kazhdan-Lusztig:Hecke_Alg} by $P_{wv}(q) = q^{\ell(w)-\ell(v)} P_{v,w}(q^{-2})$, and our Kazhdan-Lusztig basis element $\psi(w)$ is the same as $\overline{C_w}$ in \textit{op. cit.}


\subsection{Main algorithm}\label{subsec:algorithm}

Recall from \textsection \ref{subsec:geom_prelim} that $i_{w^D}: C(w^D) \to X$ is the inclusion map of the Schubert cell $C(w^D)$. Recall also that the category $\Mod_{coh}(\cD_{C(w^D)},N,\eta)$ is semisimple and $\cO_{C(w^D)}^\eta$ is the unique irreducible object. Therefore, any complex $\cV^\bullet$ of modules in this category is a direct sum of $\cO_{C(w^D)}^\eta$ at various degrees. We write $\chi_q \cV^\bullet$ for its generating function in variable $q$, i.e.
\begin{equation}\label{eqn:defn_of_chi_q}
	\chi_q \cV^\bullet = \sum_{p \in \BZ} \big( \rank H^p \cV^\bullet \big) q^p.
\end{equation}
Define the map
\begin{equation}\label{eqn:defn_of_nu}
	\nu: \operatorname{Obj} \Mod_{coh}(\cD_\lambda,N,\eta) \aro \cH_\Theta,\quad
	\cF \mapsto \sum_{D \in W_\Theta \backslash W} \big( \chi_q i_{w^D}^! \cF \big) \delta_D.
\end{equation}
Clearly, this map can be defined for any complexes of $\cD_\lambda$-modules with $\eta$-twisted $N$-equivariant cohomologies. The following easy property of $\nu$ is immediate:

\begin{lemma}\label{lem:nu_std}
	\begin{equation*}
		\nu(\cI(w^C,\lambda,\eta)) = \delta_C.
	\end{equation*}
\end{lemma}

\begin{proof}
	Let $D \in W_\Theta \backslash W$. Then $i_{w^D}^! \cI(w^C,\lambda,\eta) = i_{w^D}^! i_{w^C+} \cO_{C(w^C)}^\eta$. By base change theorem, this is $\cO_{C(w^C)}^\eta$ if $C = D$ and is $0$ otherwise. Hence the claim follows by the definition of $\nu$.
\end{proof}

\begin{theorem}[Kazhdan-Lusztig Algorithm for Whittaker modules]\label{thm:KL_alg}
	Fix a character $\eta: \fn \to \BC$. For any $\lambda \in \fh^*$, there exists a unique map
	\begin{equation*}
		\varphi_\lambda: W_\Theta \backslash W \aro \cH_\Theta
	\end{equation*}
	such that for any $C \in W_\Theta \backslash W$, if we write $u$ for the unique element in $A_{\Theta,\lambda}$ such that $C$ is contained in $W_\Theta u W_\lambda$, the following conditions hold:
	\begin{enumerate}
		\item for some $P_{CD}^{u,\lambda} \in q \BZ[q]$, 
		\begin{equation*}
			\varphi_\lambda(C) = \delta_C + 
			\sum_{\substack{D \in W_\Theta \backslash W_\Theta u W_\lambda\\%
					D <_{u,\lambda} C}} P_{CD}^{u,\lambda} \delta_D.
		\end{equation*}
		
		\item for any $\alpha \in \Pi \cap \Pi_\lambda$ with $C s_\alpha < C$, there exist $c_D \in \BZ$ such that
		\begin{equation*}
			T_\alpha( \varphi_\lambda (C s_\alpha) ) =		
			\sum_{\substack{D \in W_\Theta \backslash W_\Theta u W_\lambda \\ D \le_{u,\lambda} C}} c_D \varphi_\lambda( D)
		\end{equation*}
		
		\item for any $\beta \in \Pi - \Pi_\lambda$ such that $C s_\beta < C$,
		\begin{equation*}
			\varphi_{s_\beta \lambda} (C s_\beta) = \varphi_\lambda (C) s_\beta
		\end{equation*}
		(recall that the action $\cH_\Theta \racts W$ is given by $\delta_C \cdot w = \delta_{C w}$).
		
		\item The polynomials $P_{CD}^{u,\lambda}$ are parabolic Kazhdan-Lusztig polynomials for $(W_\lambda,\Pi_\lambda, {\Theta(u,\lambda)})$.
	\end{enumerate}
	Moreover, the map $\varphi_\lambda$ is given by
	\begin{equation*}
		\varphi_\lambda(C) = \nu(\cL(w^C,\lambda,\eta)).
	\end{equation*}
\end{theorem}

If $\lambda$ is integral, this reduces to the main theorem of Romanov \cite[Theorem 11]{Romanov:Whittaker}. Readers can go back to \textsection\ref{subsec:geom_idea} for an explanation of the meaning of the algorithm and an outline of the proof.

Let us begin the proof of the theorem. Uniqueness is determined by (1), (4), and the uniqueness of Whittaker Kazhdan-Lusztig polynomials. For existence, we will show that $\varphi_\lambda(C) = \nu(\cL(w^C,\lambda))$ satisfies the requirements (1)-(4) by induction on $\ell(w^C)$. 

Consider the base case $\ell(w^C) = \ell(w_\Theta)$, that is, $C = W_\Theta$, $w^C = w_\Theta$. The argument for this case in the same as in \cite{Romanov:Whittaker}. Any composition factor of the standard module $\cI(w_\Theta,\lambda,\eta)$ is supported on cells $C(w)$ in the closure of $C(w_\Theta)$. But any such $w$ are in $W_\Theta$ with $w \le w_\Theta$. In particular, $w$ is not the longest element in its right $W_\Theta$-coset unless $w = w_\Theta$. So there is no module supported on $C(w)$ unless $w = w_\Theta$. Hence the only composition factors are supported on $C(w_\Theta)$. By pulling back to $C(w_\Theta)$, we see that there is only one such factor, namely $\cL(w_\Theta,\lambda,\eta)$. Thus $\cI(w_\Theta,\lambda,\eta) = \cL(w_\Theta,\lambda,\eta)$. As a result 
\begin{equation*}
	\nu(\cL(w_\Theta,\lambda,\eta))
	= \nu(\cI(w_\Theta,\lambda,\eta))
	= \delta_{W_\Theta}
\end{equation*}
by \ref{lem:nu_std}. Therefore, the function $\varphi_\lambda(C)$ satisfies (1) for $C = W_\Theta$. The conditions (2) and (3) are void, and (4) is trivially true. This completes the base case.

Now consider the case $\ell(w^C) = k > \ell(w_\Theta)$. The verification of (1)-(4) for $C$ is divided into subsections.

\subsection{Verification of \ref{thm:KL_alg}(3) for $\ell(w^C) = k$}\label{subsec:(3)}

Assume $\beta \in \Pi - \Pi_\lambda$ is such that $C s_\beta < C$. By definition,
\begin{align*}
	\varphi_\lambda(C) s_\beta 
	&= \left( \sum_{D \in W_\Theta \backslash W}  \big( \chi_q i_{w^D}^! \cL(w^C,\lambda,\eta) \big)  \delta_D \right) s_\beta\\
	&= \sum_{D \in W_\Theta \backslash W} \big( \chi_q i_{w^D}^! \cL(w^C,\lambda,\eta) \big) \delta_{D s_\beta}
\end{align*}
and
\begin{align*}
	\varphi_{s_\beta \lambda} (C s_\beta) 
	&= \sum_{D \in W_\Theta \backslash W} \big( \chi_q i_{w^D}^! \cL(w^{Cs_\beta}, s_\beta\lambda,\eta) \big) \delta_D\\
	&= \sum_{D \in W_\Theta \backslash W} \big( \chi_q i_{w^{Ds_\beta}}^! \cL(w^{Cs_\beta}, s_\beta\lambda,\eta) \big)  \delta_{D s_\beta}
\end{align*}
where in the last equality we rearranged the sum by the bijection $W_\Theta \backslash W \bij W_\Theta \backslash W$, $D \mapsto D s_\beta$. Hence it suffices to show 
\begin{equation*}
	\chi_q i_{w^D}^! \cL(w^C,\lambda,\eta) = \chi_q i_{w^{Ds_\beta}}^! \cL(w^{Cs_\beta}, s_\beta\lambda,\eta)
\end{equation*}
for any $D \in W_\Theta \backslash W$, or equivalently
\begin{equation*}
	\rank H^p i_{w^D}^! \cL(w^C,\lambda,\eta)
	= \rank H^p i_{w^{D s_\beta}}^! \cL(w^{C s_\beta},s_\beta\lambda, \eta)
\end{equation*}
for any $D$ and any $p \in \BZ$. This follows by Proposition \ref{thm:Is_pullback}.

\subsection{Verification of \ref{thm:KL_alg}(2) for $\ell(w^C) = k$}\label{subsec:alg_(2)}

This part of the argument is modified from \cite[Chapter 5 \textsection 2]{Milicic:Localization} and is almost identical to \cite[\textsection 5]{Romanov:Whittaker} with slight modifications. Instead of reproving all the details, we briefly review Romanov's argument and point out the main change for our situation.

Let us assume for the moment that $\lambda$ is integral. Suppose $\alpha \in \Pi$ and $C s_\alpha < C$. Writing $\varphi_\lambda(C) = \nu(\cL(w^C,\lambda,\eta))$, \ref{thm:KL_alg}(2) reads
\begin{equation}\label{eqn:T_condition_rephrased_geom}
	(T_\alpha \circ \nu)(\cL(w^C s_\alpha,\lambda,\eta)) = \sum_{D \le C} c_D \; \nu(\cL(w^D,\lambda,\eta)).
\end{equation}
In order to prove this, let $X_\alpha$ be the partial flag variety with natural projection $p_\alpha: X \to X_\alpha$. Define the functor 
\begin{equation*}
	U_\alpha := p_\alpha^+ p_{\alpha_+}.
\end{equation*}
By the decomposition theorem \cite{Mochizuki:Decomp} and a careful study of (integral) intertwining functors, one can show that, for some integers $c_D$,
\begin{equation}\label{eqn:U_semisimple}
	U_\alpha \cL(w^C s_\alpha, \lambda,\eta) = \bigoplus_{D \le C} \cL(w^D,\lambda,\eta)^{\oplus c_D}
\end{equation}
(\cite[Lemma 17]{Romanov:Whittaker}; cf. \cite[Chapter 5 Lemma 2.7]{Milicic:Localization}). Then one shows that
\begin{equation}\label{eqn:U_lifts_T}
	(T_\alpha \circ \nu)(\cL(w^C s_\alpha,\lambda,\eta)) = (\nu \circ U_\alpha)(\cL(w^C s_\alpha,\lambda,\eta)).
\end{equation}
In other words, $U_\alpha$ is the ``geometric lift'' of the combinatorially defined operator $T_\alpha$. Once this equality is established, (\ref{eqn:U_semisimple}) and (\ref{eqn:U_lifts_T}) together lead to (\ref{eqn:T_condition_rephrased_geom}), which proves \ref{thm:KL_alg}(2). Therefore it remains to prove the equality (\ref{eqn:U_lifts_T}). To do this, first notice that $T_\alpha$ restricts to an operation on the $\BZ[q^{\pm1}]$-submodule spanned by $\delta_D$ and $\delta_{D s_\alpha}$ and $U_\alpha$ operate on the abelian subgroup of $K \Mod_{coh}(\cD_\lambda,N,\eta)$ spanned by the classes of $\cI(w^D,\lambda,\eta)$ and $\cI(w^D s_\alpha,\lambda,\eta)$ (in fact one can easily show that $T_\alpha \circ \nu$ and $\nu \circ U_\alpha$ agree on standard modules). Therefore it is natural to restrict $U_\alpha \cL(w^C s_\alpha,\lambda,\eta)$ to the subvariety $X_O := C(w^D) \cup C(w^D s_\alpha)$. Since $X_O$ is the union of two Schubert cells, one can apply the distinguished triangle for local cohomology to $U_\alpha \cL(w^C s_\alpha,\lambda,\eta)|_{X_O}$ and take the long exact sequence on cohomologies. Using a parity degree vanishing property of Whittaker Kazhdan-Lusztig polynomials 
\begin{equation*}
	P_{CD} \in \BZ[q^{\pm 2}] \cdot q^{\ell(w^C) - \ell(w^D)}
\end{equation*}
(\cite[Lemma 15]{Romanov:Whittaker}), this long exact sequence splits into short exact sequences, producing a description of $U_\alpha \cL(w^C s_\alpha,\lambda,\eta)|_{C(w^D)}$ in terms of the ranks of $\cL(w^C s_\alpha,\lambda,\eta)|_{C(w^D)}$. (\ref{eqn:U_lifts_T}) is proven by plugging these descriptions into the definition of $(\nu \circ U_\alpha)(\cL(w^C s_\alpha,\lambda,\eta))$ and comparing with the left side.

Now let $\lambda$ be general, and let $\alpha \in \Pi \cap \Pi_\lambda$ so that $C s_\alpha < C$. In this case, the $U_\alpha$ functor still exists. In more detail, the definition of $U_\alpha$ requires the existence of a twisted sheaf of differential operators $\cD_{X_\alpha,\lambda}$ on $X_\alpha$ whose pullback to $X$ is $\cD_\lambda$. Such existence is equivalent to $\alpha^\vee(\lambda) = -1$. Since $\alpha$ is assumed to be integral to $\lambda$, the condition $\alpha^\vee(\lambda) = -1$ can be achieved by twisting $\lambda$ by an integral weight, which can be done by twisting $\cD$-modules by a line bundle.

Roughly the same argument applies to this situation, except now \ref{thm:KL_alg}(2) becomes
\begin{equation*}
	(T_\alpha \circ \nu)(\cL(w^C s_\alpha,\lambda,\eta)) = \sum_{\substack{D \in W_\Theta \backslash W_\Theta u W_\lambda \\ D \le_{u,\lambda} C}} c_D \; \nu(\cL(w^D,\lambda,\eta))
\end{equation*}
which has fewer terms on the right side compared to (\ref{eqn:T_condition_rephrased_geom}). Namely, the summand $\nu(\cL(w^D,\lambda,\eta))$ does not appear if either $D$ is not in the same double coset as $C$, or $D \not\le_{u,\lambda} C$. To get this restricted sum, one uses the induction assumption \ref{thm:KL_alg}(1) on $\nu(\cL(w^C s_\alpha,\lambda,\eta))$ and shows that
\begin{equation}\label{eqn:U_semisimple_restricted}
	U_\alpha \cL(w^C s_\alpha, \lambda,\eta) = \bigoplus_{\substack{D \in W_\Theta \backslash W_\Theta u W_\lambda \\ D \le_{u,\lambda} C}} \cL(w^D,\lambda,\eta)^{\oplus c_D}
\end{equation}
during the course of proving the lifting (\ref{eqn:U_lifts_T}). Thus 
\begin{align*}
	T_\alpha( \varphi_\lambda(C s_\alpha))
	&= \nu( U_\alpha \cL(w^C s_\alpha,\lambda,\eta) )\\
	&= \nu\Big( \bigoplus_{\substack{D \in W_\Theta \backslash W_\Theta u W_\lambda\\ D \le_{u,\lambda} C}} \cL(w^D, \lambda,\eta)^{\oplus c_D} \Big) \\
	&= \sum_{\substack{D \in W_\Theta \backslash W_\Theta u W_\lambda\\ D \le_{u,\lambda} C}} c_D \nu (\cL(w^D, \lambda,\eta)) \\
	&= \sum_{\substack{D \in W_\Theta \backslash W_\Theta u W_\lambda\\ D \le_{u,\lambda} C}} c_D \varphi_\lambda(D).
\end{align*}
and \ref{thm:KL_alg}(2) is verified for $C$.

\subsection{Verification of \ref{thm:KL_alg}(1) for $\ell(w^C) = k$}\label{subsec:(1)}

The idea of this proof (and the proof of \ref{thm:KL_alg}(4)) is to find a simple reflection $s$ so that $C s< C$ and deduce information of $\cL(w^C,\lambda,\eta)$ from that of $\cL(w^{Cs},\lambda',\eta)$. If $s$ is non-integral, we can use non-integral intertwining functor $I_s$ to translate properties of $C s$ to $C$. If $s$ is integral, we then use information about the $U$-functor.

Suppose there exists $\beta \in \Pi - \Pi_\lambda$ such that $C s_\beta < C$. Then the non-integral intertwining functor $I_{s_\beta}$ sends $\cL(w^C,\lambda,\eta)$ to $\cL(w^{C s_\beta},s_\beta\lambda,\eta)$, allowing us to translate induction assumption for the latter module to the former. Since the induction hypothesis applies for $C s_\beta$, from \ref{thm:KL_alg}(1) for $C s_\beta$ and $s_\beta \lambda$, we obtain
\begin{equation*}
	\varphi_{s_\beta \lambda} (C s_\beta) 
	= \delta_{C s_\beta} 
	+ \sum_{\substack{	D \in W_\Theta \backslash W_\Theta r W_{s_\beta \lambda}\\  D  <_{r,s_\beta \lambda} C s_\beta}} 
	Q_{D} \delta_D,
\end{equation*}
for some polynomials $Q_{D} \in q \BZ[q]$, where $r$ is the unique element in $A_{\Theta,s_\beta \lambda}$ such that $C s_\beta \in W_\Theta \backslash W_\Theta r W_{s_\beta \lambda}$. Applying \ref{thm:KL_alg}(3) for $C$ (which was already proven), 
\begin{equation*}
	\varphi_\lambda(C) 
	= \varphi_{s_\beta \lambda} (C s_\beta) s_\beta
	= \delta_C 
	+ \sum_{\substack{	D \in W_\Theta \backslash W_\Theta r W_{s_\beta \lambda}\\ D  <_{r,s_\beta \lambda} C s_\beta}} 
	Q_{D} \delta_{D s_\beta}.
\end{equation*}
We want to rewrite the subscript of the sum. By Lemma \ref{lem:Is_pres_lowest_db_coset} and its corollary, there exists $w \in W_\Theta$ with $w r = u s_\beta$. Hence
\begin{align*}
	W_\Theta r W_{s_\beta \lambda} 
	&= W_\Theta w r s_\beta W_\lambda s_\beta\\
	&= W_\Theta u s_\beta s_\beta W_\lambda s_\beta\\
	&= W_\Theta u W_\lambda s_\beta,
\end{align*}
and we see that $D \in W_\Theta \backslash W_\Theta r W_{s_\beta \lambda}$ if and only if $D s_\beta \in W_\Theta \backslash W_\Theta u W_\lambda$. By Proposition \ref{lem:Is_right_coset},
\begin{equation*}
	D <_{r,s_\beta \lambda} C s_\beta \iff D s_\beta <_{u,\lambda} C.
\end{equation*}
Hence
\begin{align*}
	\varphi_\lambda(C) &= \delta_C 
	+ \sum_{\substack{	D s_\beta \in W_\Theta \backslash W_\Theta u W_\lambda \\ D s_\beta <_{u,\lambda}C}}
	Q_{D} \delta_{D s_\beta}\\
	&= \delta_C 
	+ \sum_{\substack{	E \in W_\Theta \backslash W_\Theta u W_\lambda \\ E <_{u,\lambda}C}}
	Q_{D} \delta_E
\end{align*}
for some $Q_{D} \in q \BZ[q]$, and \ref{thm:KL_alg}(1) holds for $C$ in this case.

If such $\beta$ does not exist, then there exists a simple integral root $\alpha$ with $C s_\alpha < C$. From (\ref{eqn:U_semisimple_restricted}), we know that  $U_\alpha \cL(C s_\alpha,\lambda,\eta)$ is a direct sum of various irreducible modules supported in the closure of $C(w^C)$. Moreover, since the support of $U_\alpha \cL(C s_\alpha,\lambda,\eta)$ is the closure of ${p_\alpha\inv(p_\alpha(C(w^{Cs_\alpha})))}$ which is \textit{equal} to the closure of $C(w^C)$, $\cL(w^C,\lambda,\eta)$ has to be a direct summand of $U_\alpha \cL(C s_\alpha,\lambda,\eta)$. So the coefficients of the polynomial $\chi_q i_{w^D}^! \cL(w^C,\lambda,\eta)$ (which are non-negative integers) must be dominated by those of $\chi_q i_{w^D}^! U_\alpha \cL(w^C s_\alpha,\lambda)$. On the other hand, we know from (\ref{eqn:U_semisimple_restricted}) that the latter polynomial vanishes if $D$ is not in $W_\Theta u W_\lambda$ or if $D \not\le_{u,\lambda} C$. So the polynomial $\chi_q i_{w^D}^! \cL(w^C,\lambda,\eta)$ also vanishes for those $D$. Hence
\begin{align*}
	\varphi_\lambda(C) 
	&= \sum_{D \in W_\Theta \backslash W} \big( \chi_q i_{w^D}^! \cL(w^C,\lambda,\eta) \big) \delta_D\\
	&= \sum_{\substack{	D \in W_\Theta \backslash W_\Theta u W_\lambda \\D \le_{u,\lambda} C}} \big( \chi_q i_{w^D}^! \cL(w^C,\lambda,\eta) \big) \delta_D.
\end{align*}

It suffices to compute the remaining coefficients. Consider $D = C$. Since $\cI(w^C,\lambda,\eta)$ is a direct image, it contains no section supported in $\partial C(w^C)$ except $0$. The same holds for $\cL(w^C,\lambda,\eta)$ being a submodule of $\cI(w^C,\lambda,\eta)$. Hence $\cL(w^C,\lambda,\eta)|_{X - \partial C(w^C)}$ is a nonzero submodule of $\cI(w^C,\lambda,\eta)|_{X - \partial C(w^C)}$. But $\cI(w^C,\lambda,\eta)|_{X - \partial C(w^C)}$ is irreducible by Kashiwara's equivalence of categories for the closed immersion $C(w^C) \inj X - \partial C(w^C)$, so $\cL(w^C,\lambda,\eta)|_{X - \partial C(w^C)} = \cI(w^C,\lambda,\eta)|_{X - \partial C(w^C)}$, and their further pullback to $C(w^C)$ is $\cO_{C(w^C)}^\eta$. Hence the coefficient of $\delta_C$ is $1$. For $D < C$, we know $H^0 i_{w^D}^!$ takes sections supported in $C(w^D)$. Since $\cL(w^C,\lambda,\eta)$ has no section supported in $\partial C(w^C) \supset C(w^D)$, $H^0 i_{w^D}^! \cL(w^C,\lambda,\eta) = 0$ and the coefficient of $\delta_D$ has no constant term. Thus \ref{thm:KL_alg}(1) holds for $C$.

\subsection{Verification of \ref{thm:KL_alg}(4) for $\ell(w^C) = k$}\label{subsec:(4)}

Based on our definition of Whittaker Kazhdan-Lusztig polynomials \ref{def:parabolic_KL_poly}, we need to find $\alpha \in \Pi_\lambda$ such that $C s_\alpha <_{u,\lambda} C$ and equation (\ref{eqn:parabolic_KL_poly}) holds for the function
\begin{equation*}
	\psi_{u,\lambda}(C|_\lambda) := \varphi_\lambda(C)|_\lambda.
\end{equation*}
See \textsection\ref{subsec:lem_induction} for an explanation of the geometric idea behind this proof.

If $\alpha$ can be chosen to be in $\Pi \cap \Pi_\lambda$, then by the following lemma, (\ref{eqn:parabolic_KL_poly}) follows from \ref{thm:KL_alg}(2) for $C$. 

\begin{lemma}
	Let $\alpha \in \Pi \cap \Pi_\lambda$. Then for each $u \in A_{\Theta,\lambda}$ 
	\begin{equation*}
		(-)|_\lambda \circ T_\alpha = T_\alpha^{u,\lambda} \circ (-)|_\lambda
	\end{equation*}
	as maps from $\ind_\lambda \cH_{\Theta(u,\lambda)} \subseteq \cH_\Theta$ to $\cH_{\Theta(u,\lambda)}$. In other words, the following diagram commutes
	\begin{equation*}
		\begin{tikzcd}[column sep=10ex]
			\cH_\Theta \ar[r, "T_\alpha"] \ar[d, "(-)|_\lambda"'] & \cH_\Theta \ar[d, "(-)|_\lambda"]\\
			{\displaystyle \bigoplus_{u \in A_{\Theta,\lambda}} \cH_{\Theta(u,\lambda)}} \ar[r, "\bigoplus_u T_\alpha^{u,\lambda}"] & {\displaystyle \bigoplus_{u \in A_{\Theta,\lambda}} \cH_{\Theta(u,\lambda)}}
		\end{tikzcd}
	\end{equation*}
	(recall the definitions of $T_\alpha^{u,\lambda}$ and $(-)|_\lambda$ in \textsection \ref{subsec:KL_poly}).
\end{lemma}

The proof is straightforward. It consist of unwrapping definitions and using the facts that $\ind_\lambda$ (defined in Corollary \ref{lem:right_coset_partition} and in \textsection \ref{subsec:KL_poly}) preserves partial orders on right cosets \ref{thm:int_Whittaker_model}.

If such $\alpha$ cannot be found, we will need to use non-integral intertwining functors to move $\alpha$ to some simple root $s_{\beta_s} \cdots s_{\beta_1} \alpha = z\inv \alpha$ and move $\cL(w^C,\lambda,\eta)$ to some irreducible module supported on a smaller orbit, then translate the induction assumption there back. The translation requires the following lemma. The proof is similar to the previous lemma, using \ref{lem:Is_right_coset} instead of \ref{thm:int_Whittaker_model}.

\begin{lemma}
	Let $\alpha \in \Pi_\lambda$, $\beta \in \Pi - \Pi_\lambda$. For any $u \in A_{\Theta,\lambda}$, let $r \in A_{\Theta,s_\beta\lambda}$ be the unique element such that $W_\Theta u s_\beta W_{s_\beta\lambda} = W_\Theta r W_{s_\beta \lambda}$. Then
	\begin{equation*}
		(s_\beta (-) s_\beta ) \circ T_\alpha^{u,\lambda} = T_{s_\beta \alpha}^{r,s_\beta\lambda} \circ (s_\beta (-) s_\beta)
	\end{equation*}
	as maps from $\cH_{\Theta(u,\lambda)}$ to $\cH_{\Theta(r,s_\beta\lambda)}$, where $s_\beta (-) s_\beta$ denotes conjugation by $s_\beta$. In other words, the following diagram commutes
	\begin{equation*}
		\begin{tikzcd}[column sep=10ex]
			\cH_{\Theta(u,\lambda)} \ar[r, "T_\alpha^{u,\lambda}"] \ar[d, "s_\beta(-)s_\beta"'] & \cH_{\Theta(u,\lambda)} \ar[d, "s_\beta(-)s_\beta"]\\
			\cH_{\Theta(r,s_\beta\lambda)} \ar[r, "T_{s_\beta\alpha}^{r,s_\beta\lambda}"] & \cH_{\Theta(r,s_\beta\lambda)}.			
		\end{tikzcd}
	\end{equation*}
\end{lemma}

Choose $\alpha \in \Pi_\lambda$, $s \ge 0$ and $\beta_1,\ldots,\beta_s \in \Pi$ such that if we write $z_0 = 1$, $z_i = s_{\beta_1} \cdots s_{\beta_i}$ and $z = z_s$, the following conditions hold:
\begin{enumerate}[label=(\alph*)]
	\item for any $0 \le i \le s-1$, $\beta_{i+1}$ is non-integral to $z_i\inv \lambda$;
	\item $z\inv \alpha \in \Pi \cap \Pi_{z\inv \lambda}$;
	\item $C s_\alpha <_{u,\lambda}  C$;
	\item if $s > 0$, $C z < C$;
	\item $C s_\alpha z = C z s_{z\inv \alpha} < C z$.
\end{enumerate}
Such a choice exists by \ref{lem:decrease_of_length}. Combining the lemmas with the diagram (\ref{eqn:Is_right_coset}), we obtain a commutative diagram
\begin{equation}\label{eqn:three_comm_diag}
	\begin{tikzcd}[column sep=10ex]
		& \cH_\Theta \ar[r, "T_{z\inv\alpha}"] \ar[d, "(-)|_{z\inv\lambda}"'] \ar[dl, "(-)z\inv"'] 
		& \cH_\Theta \ar[d, "(-)|_{z\inv\lambda}"]\\[3ex]
		\cH_\Theta \ar[d, "(-)|_\lambda"']
		& {\displaystyle \bigoplus_{r \in A_{\Theta,z\inv\lambda}} \cH_{\Theta(r,z\inv\lambda)}} \ar[r, "{\bigoplus_{r} T_{z\inv\alpha}^{r,z\inv\lambda}}"] \ar[dl, "z(-)z\inv"']
		& {\displaystyle \bigoplus_{r \in A_{\Theta,z\inv\lambda}} \cH_{\Theta(r,z\inv\lambda)}} \ar[dl, "z(-)z\inv"]\\
		{\displaystyle \bigoplus_{u \in A_{\Theta,\lambda}} \cH_{\Theta(u,\lambda)}} \ar[r, "{\bigoplus_u T_\alpha^{u,\lambda}}"]
		& {\displaystyle \bigoplus_{u \in A_{\Theta,\lambda}} \cH_{\Theta(u,\lambda)}}		
	\end{tikzcd}.
\end{equation}

Since $Cz < C$, the induction assumption applies to $C z$ for $z\inv \lambda$. In particular, if we apply \ref{thm:KL_alg}(2) to $C z s_{z\inv\alpha} < C z$ and $z\inv \lambda$, we obtain the equation
\begin{equation}\label{eqn:(4)_(assumption)}
	T_{z\inv \alpha} \big( \varphi_{z\inv \lambda}( C z s_{z\inv \alpha}) \big) 
	= \sum_{\substack{D\in W_\Theta \backslash W_\Theta r W_{z\inv \lambda}\\  D \le_{r,z\inv\lambda} Cz}} c_D \varphi_{z\inv \lambda}(D)
\end{equation}
where $r$ is the unique element in $A_{\Theta,z\inv\lambda}$ such that $Cz \in W_\Theta \backslash W_\Theta r W_{z\inv \lambda}$. We apply to both sides $(-)|_{z\inv \lambda}$ followed by $z(-)z\inv$.

If we view $\varphi_{z\inv \lambda}( C z s_{z\inv \alpha})$ as an element in the middle $\cH_\Theta$ in the diagram, the left side of (\ref{eqn:(4)_(assumption)}) lands in $\cH_{\Theta(u,\lambda)}$ in the bottom middle of the diagram through the rightmost path after applying $(-)|_{z\inv \lambda}$ and $z(-)z\inv$. Going through the leftmost path instead, this element in $\cH_{\Theta(u,\lambda)}$ becomes
\begin{equation*}
	T_\alpha^{u,\lambda}\big( \varphi_{z\inv \lambda}( C z s_{z\inv \alpha} )z\inv|_\lambda \big).
\end{equation*}
Rewrite $Cz s_{z\inv\alpha} = Cs_\alpha z$ and use \ref{thm:KL_alg}(3) for $C s_\alpha$, the above quantity becomes
\begin{equation*}
	T_\alpha^{u,\lambda} \big( \varphi_\lambda( C s_\alpha)|_\lambda \big)
	= T_\alpha^{u,\lambda} (\psi_{u,\lambda}(C|_\lambda) ).
\end{equation*}

Viewing the right side of (\ref{eqn:(4)_(assumption)}) as an element in the middle $\cH_\Theta$ in the diagram, $(-)|_{z\inv \lambda}$ and $z(-)z\inv$ sends it to $\cH_{\Theta(u,\lambda)}$ at the bottom-left along the middle path. Going through the leftmost path instead, this element becomes
\begin{equation*}
	\sum_{\substack{D\in W_\Theta \backslash W_\Theta r W_{z\inv \lambda}\\ D \le_{r,z\inv\lambda} Cz}} c_D  \varphi_\lambda(Dz\inv)|_\lambda
	= \sum_{\substack{D\in W_\Theta \backslash W_\Theta r W_{z\inv \lambda}\\ D \le_{r,z\inv\lambda} Cz}} c_D \psi_{u,\lambda}( (D z\inv)|_\lambda).
\end{equation*}
As in the first part of \textsection\ref{subsec:(1)}, we can rewrite the subscript of the sum. There is an element $w \in W_\Theta$ such that $wr = u z$ by \ref{lem:Is_pres_lowest_db_coset}. Hence
\begin{align*}
	W_\Theta r W_{z\inv \lambda} 
	&= W_\Theta wr z\inv W_\lambda z\\
	&= W_\Theta uz z\inv W_\lambda z\\
	&= W_\Theta u W_\lambda z,
\end{align*}
and $D \in W_\Theta \backslash W_\Theta r W_{z\inv \lambda}$ if and only if $D z\inv \in W_\Theta \backslash W_\Theta u W_\lambda$. Moreover, by \ref{lem:Is_right_coset},
\begin{equation*}
	D \le_{r,z\inv\lambda} Cz \iff  D z\inv \le_{u,\lambda} C.
\end{equation*}
Hence (\ref{eqn:(4)_(assumption)}) becomes
\begin{equation*}
	T_\alpha^{u,\lambda} (\psi_{u,\lambda}(C|_\lambda) )
	= \sum_{\substack{ Dz\inv \in W_\Theta \backslash W_\Theta u W_\lambda\\ D z\inv \le_{u,\lambda} C}} c_D \psi_{u,\lambda}( (D z\inv)|_\lambda).
\end{equation*}
Therefore, $\alpha \in \Pi_\lambda$ is such that $C s_\alpha <_{u,\lambda} C$ and  equation (\ref{eqn:parabolic_KL_poly}) holds for $C s_\alpha$. By \ref{def:parabolic_KL_poly}, the polynomials $P_{CD}^{u,\lambda}$ are parabolic Kazhdan-Lusztig polynomials for $(W_\lambda,\Pi_\lambda,\Theta(u,\lambda))$. Thus \ref{thm:KL_alg}(4) holds for $C$.

This completes the proof of the algorithm \ref{thm:KL_alg}.
\section{Character formula for irreducible modules}\label{sec:character_formula}

\subsection{Regular case}\label{subsec:character_formula}

By standard arguments, the algorithm \ref{thm:KL_alg} leads to a character formula for irreducible Whittaker modules for regular infinitesimal characters: one takes \ref{thm:KL_alg}(1)(4) for $-\lambda$ dominant regular (so that $\lambda$ is antidominant regular), precomposing with holonomic duality $\BD$ (so that standard $\cD$-modules $\cI$ becomes costandard $\cD$-modules $\cM$), descending to the Grothendieck group by specializing at $q=-1$, passing through Beilinson-Bernstein localization, and applying the character map. 

In more details, let $\lambda \in \fh^*$ be antidominant regular. As explained in \textsection\ref{subsec:geom_prelim}, Beilinson-Bernstein's localization and holonomic duality are (anti-)equivalences of categories which send Whittaker modules to $\cD$-modules. Combined with the map $\nu$, we obtain the composition
\begin{equation}\label{eqn:flowchart}
	\arraycolsep=1.4pt
	\begin{array}{ccccccccc}
		\cN_{\theta,\eta} & \xleftarrow{\Gamma(X,-)} &\Mod_{coh}(\cD_\lambda,N,\eta) &\xrightarrow{\BD} &\Mod_{coh}(\cD_{-\lambda},N,\eta) &\xrightarrow{\nu}
		&\cH_\Theta &\xrightarrow{(-)|_{-\lambda}}
		&{\displaystyle \bigoplus_{u \in A_{\Theta,-\lambda}} 	\cH_{\Theta(u,-\lambda)}},\\
		L(w^C\lambda,\eta) &\mapsfrom &\cL(w^C,\lambda,\eta) &\mapsto &\cL(w^C,-\lambda,\eta) &\mapsto &\varphi_{-\lambda}(C) &\mapsto &\varphi_\lambda(C)|_{-\lambda}\\
		M(w^C\lambda,\eta) &\mapsfrom &\cM(w^C,\lambda,\eta) &\mapsto &\cI(w^C,-\lambda,\eta) &\mapsto &\delta_C &\mapsto &\delta_{C|_{-\lambda}}.
	\end{array}
\end{equation}
At $q=-1$, the coefficients $\chi_q i_{w^D}^! \cF$ in the definition of $\nu$ is additive with respect to short exact sequences. Hence $\nu$ factors through the Grothendieck group
\begin{equation*}
	\nu|_{q=-1}: K\Mod_{coh}(\cD_{-\lambda},N,\eta) \aro \cH_\Theta|_{q=-1}
\end{equation*}
which is an isomorphism by \ref{lem:nu_std}. Therefore we have an isomorphism of abelian groups
\begin{equation*}
	\arraycolsep=1.4pt
	\begin{array}{ccc}
		K\cN_{\theta,\eta} & \xrightarrow{\cong} & {\displaystyle \bigoplus_{u \in A_{\Theta,-\lambda}} \cH_{\Theta(u,-\lambda)}|_{q=-1}}\\
		{[L(w^C\lambda,\eta)]} & \mapsto & \varphi_\lambda(C)|_{-\lambda}|_{q=-1}\\
		{[M(w^C\lambda,\eta)]} & \mapsto & \delta_{C|_{-\lambda}}|_{q=-1},
	\end{array}	
\end{equation*}
where $[-]$ takes the class in the Grothendieck group. Hence Theorem \ref{thm:KL_alg}(1) and (4) imply
\begin{equation*}
	[L(w^C\lambda,\eta)] = 
	\sum_{\substack{D \in W_\Theta \backslash W_\Theta u W_{-\lambda}\\	D \le_{u,-\lambda} C}} 
	P_{CD}^{u,-\lambda}(-1) [M(w^D \lambda,\eta)]
\end{equation*}
in $K \cN_{\theta,\eta}$. Note that $\Sigma_\lambda = \Sigma_{-\lambda}$ as subsets of $\Sigma$ and $W_\lambda = W_{-\lambda}$ as subgroups of $W$. Hence all the combinatoric structures defined based on $\lambda$ and $-\lambda$ are canonically identified. Further applying the character map, we thus obtain

\begin{theorem}[Character formula: Regular case] \label{thm:multiplicity}
	Let $\lambda \in \fh^*$ be antidominant and regular. Let $\eta: \fn \to \BC$ be any character. For any $C \in W_\Theta \backslash W$, let $u \in A_{\Theta,\lambda}$ be the unique element such that $C \subseteq W_\Theta u W_\lambda$. Then
	\begin{equation}
		\ch L(w^C\lambda,\eta) = 
		\sum_{\substack{D \in W_\Theta \backslash W_\Theta u W_\lambda\\	D \le_{u,\lambda} C}} 
		P_{CD}^{u,\lambda}(-1) \ch M(w^D \lambda,\eta),
	\end{equation}
	where the polynomials $P_{CD}^{u,\lambda}$ are Whittaker Kazhdan-Lusztig polynomials for $(W_\lambda,\Pi_\lambda,\Theta(u,\lambda))$ as defined in \ref{def:parabolic_KL_poly} and $P_{CC}^{u,\lambda} = 1$.
\end{theorem}

When $\lambda$ is integral, we have a simpler description, which we state separately.

\begin{corollary}[Character formula: Regular integral case]
	Let $\lambda \in \fh^*$ be antidominant, regular, and integral. Let $\eta: \fn \to \BC$ be any character. For any $C \in W_\Theta \backslash W$,
	\begin{equation*}
		\ch L(w^C\lambda,\eta) = 
		\sum_{\substack{D \in W_\Theta \backslash W\\ D \le C}}
		P_{CD}(-1) \ch M(w^D \lambda,\eta),
	\end{equation*}
	where the polynomials $P_{CD}$ are Whittaker Kazhdan-Lusztig polynomials for $(W,\Pi,\Theta)$ as defined in \ref{def:parabolic_KL_poly_Theta} and $P_{CC}=1$.
\end{corollary}

Inverting the matrix $(P_{CD}(-1))_{C,D}$, we recover the description in \cite{Milicic-Soergel:Whittaker_algebraic} and \cite{Romanov:Whittaker} of multiplicities of irreducible Whittaker modules in standard Whittaker modules with regular integral infinitesimal characters.

At another extreme, when $\eta = 0$ (i.e. $\Theta = \varnothing$), we recover the well-known non-integral Kazhdan-Lusztig conjecture for highest weight modules (see, for example, \cite[Chapter 1]{Lusztig:Char_finite_field}, \cite[\textsection2.5 Theorem 11]{Soergel:V}, \cite[Theorem 0.1]{Kashiwara-Tanisaki:Non-int_KL}).

\begin{corollary}[Kazhdan-Lusztig conjecture for Verma modules]\label{thm:Verma_multiplicity}
	Let $\lambda \in \fh^*$ be antidominant and regular. For any $w \in W$, let $u \in A_\lambda$ be the unique element so that $w \in u W_\lambda$ (\textsection \ref{subsec:Bruhat_W/Wlambda}). For any $v \in u W_\lambda$, we write $v \le_\lambda w$ if $u\inv v = v|_\lambda \le_\lambda w|_\lambda = u\inv w$. Then
	\begin{equation*}
		\ch L(w\lambda) = 
		\sum_{\substack{v \in u W_\lambda\\ v \le_\lambda w}} P_{wv}^\lambda(-1) \ch M(v\lambda),
	\end{equation*}
	where the polynomials $P_{wv}^\lambda$ are Kazhdan-Lusztig polynomials for $(W_\lambda,\Pi_\lambda, \varnothing)$ as defined in \ref{def:parabolic_KL_poly}, $P_{ww}=1$, $M(v\lambda)$ is the Verma module of highest weight $v\lambda - \rho$, and $L(w \lambda)$ is the unique irreducible quotient of $M(w \lambda)$ (recall that $\rho$ is the half sum of roots in $\fn$).
\end{corollary}

As we have remarked at the end of \textsection \ref{subsec:KL_poly}, our (ordinary) Kazhdan-Lusztig polynomials $P_{wv}$ is related to the ones $P_{v,w}$ defined in \cite{Kazhdan-Lusztig:Hecke_Alg} by $P_{wv}(q) = q^{\ell(w)-\ell(v)} P_{v,w}(q^{-2})$. Therefore, if we write $P_{v,w}^\lambda$ for the polynomials in \cite{Kazhdan-Lusztig:Hecke_Alg} defined for $(W_\lambda,\Pi_\lambda)$, the coefficients $P_{wv}^\lambda(-1)$ in the above corollary then becomes $(-1)^{\ell_\lambda(w) - \ell_\lambda(v)} P_{v,w}^\lambda(1)$, which agrees with the coefficient appearing in \cite[Theorem 0.1]{Kashiwara-Tanisaki:Non-int_KL}.

\begin{remark}\label{rmk:Verma_case_old_proofs}
	As is mentioned in the introduction, our argument provides a new proof of the non-integral Kazhdan-Lusztig conjecture for Verma modules. Here let us briefly recall the classical approaches to this conjecture.
	
	After their resolution of the integral Kazhdan-Lusztig conjecture, Beilinson and Bernstein also treated the case where the infinitesimal character is rational (unpublished). They interpreted the multiplicities of irreducible modules in Verma modules topologically in terms of local intersection cohomologies of line bundles over Schubert varieties. Based on this interpretation, Lusztig \cite[Chapter 1]{Lusztig:Char_finite_field} gave explicit formulas for these intersection cohomology groups in positive characteristic. Since these groups can be identified with the corresponding groups in characteristic $0$, Lusztig's formulas resolves the rational Kazhdan-Lusztig conjecture. Once the rational case is treated, the general (regular) case follows by a Zariski density argument.
	
	Soergel gave an alternative proof to the conjecture by showing that the multiplicities we care about only depends on the integral Weyl group $W_\lambda$ \cite[Theorem 11]{Soergel:V}. In particular, the non-integral Kazhdan-Lusztig conjecture is reduced to the integral conjecture for $W_\lambda$. His proof involves the study of coinvariant algebras and Soergel modules, which again goes through the study of intersection cohomology complexes of Schubert varieties.
	
	In comparison, our approach (by using intertwining operators) is uniform for any (possibly non-integral) regular infinitesimal character (in that it does not require first treating the rational case) and is a $\cD$-module theoretic argument. This is important for us: localizations of Whittaker modules have irregular singularities and do not correspond to perverse sheaves (an example of such a module is contained in \cite{Milicic-Soergel:Whittaker_geometric} at the end of \textsection 4), so existing methods do not apply to our situation. Moreover, since non-integral intertwining functors are equivalences of categories of all quasi-coherent $\cD$-modules (Theorem \ref{lem:non-int_I}), an argument similar to ours should work for other non-integral Kazhdan-Lusztig problems for other categories of Lie algebra representations, provided that the integral situation is already treated.
\end{remark}

\subsection{Singular case}\label{subsec:characer_formula_singular}

The singular case can be deduced from the regular case easily.

Let $\lambda \in \fh^*$ be antidominant and singular. We still have the maps (\ref{eqn:flowchart}), but the exact functor $\Gamma(X,-)$ is no longer an equivalence of categories and only descends to a surjection $K \Mod_{coh}(\cD_\lambda,N,\eta) \surj K \cN_{\theta,\eta}$ on Grothendieck groups. However, the identification $\Gamma(X,\cM(w^D,\lambda,\eta)) = M(w^D \lambda,\eta)$ still holds. Therefore, the argument for regular case produces the equality
\begin{equation}\label{eqn:character_formula_pre}
	\ch \Gamma(X,\cL(w^C,\lambda,\eta)) = 
	\sum_{\substack{D \in W_\Theta \backslash W_\Theta u W_\lambda\\%
			D \le_{u,\lambda} C}} 
	P_{CD}^{u,\lambda}(-1) \ch M(w^D \lambda,\eta).
\end{equation}
Here $\Gamma(X,\cL(w^C,\lambda,\eta))$ could be zero, and the standard modules $M(w^D \lambda,\eta)$ could coincide for different $D$. Therefore, it suffices to describe which $M(w^D \lambda,\eta)$ coincide and which $\Gamma(X,\cL(w^C,\lambda,\eta))$ are zero. 

The first question has an easy answer. Recall from \textsection\ref{subsec:Wh_prelim} that for $C,D \in W_\Theta \backslash W$, $M(w^D \lambda,\eta) = M(w^C \lambda,\eta)$ if and only if $W_\Theta w^D \lambda = W_\Theta w^C \lambda$. Let $W^\lambda$ be the stabilizer of $\lambda$ in $W$. Then the above condition is equivalent to $W_\Theta w^D W^\lambda = W_\Theta w^C W^\lambda$, i.e. that $C$ and $D$ are in the same double $(W_\Theta,W^\lambda)$-coset.

\begin{lemma}\label{lem:std_mods_coincide}
	Let $\lambda \in \fh^*$ be antidominant and let $\eta :\fn \to \BC$ be a character. The following are equivalent:
	\begin{enumerate}[label=(\alph*)]
		\item $M(w^C \lambda,\eta) = M(w^D \lambda,\eta)$;
		\item $\Gamma(X,\cM(w^C,\lambda,\eta)) = \Gamma(X,\cM(w^D,\lambda,\eta))$;
		\item $C$ and $D$ are in the same double $(W_\Theta,W^\lambda)$-coset.
	\end{enumerate}
\end{lemma}

Therefore, for a fixed standard Whittaker module $M$, there is a unique double coset $W_\Theta v W^\lambda$ such that  $\Gamma(X,\cM(w^D,\lambda,\eta)) = M$ for all $D \in W_\Theta \backslash W_\Theta v W^\lambda$.

The following proposition answers the second question.

\begin{proposition}\label{lem:irred_vanishing_singular}
	Let $\lambda \in \fh^*$ be antidominant and let $\eta :\fn \to \BC$ be a character. Let $v \in W$. Then the set $W_\Theta \backslash W_\Theta v W^\lambda$ of right $W_\Theta$-cosets contains a unique smallest element $C$. Furthermore,
	\begin{enumerate}[label=(\alph*)]
		\item $\Gamma(X, \cL(w^C,\lambda,\eta)) = L(w^C \lambda,\eta) \neq 0$; and
		
		\item $\Gamma(X,\cL(w^D,\lambda,\eta)) = 0$ for any $D \in W_\Theta \backslash W_\Theta v W^\lambda$ not equal to $C$.
	\end{enumerate}
\end{proposition}

In other words, for a fixed standard Whittaker module $M$, among all the costandard $\cD_\lambda$-modules that realize $M$, the irreducible quotient of the one with the smallest support realizes the unique irreducible submodule of $M$.

\begin{proof}
	Write $M = M(v\lambda,\eta)$ and $L = L(v\lambda,\eta)$.
	
	First, there is one and at most one $D$ in $W_\Theta \backslash W$ with $\Gamma(X,\cL(w^D,\lambda,\eta)) = L$. This is because there is a unique irreducible $\cD_\lambda$-module $\cV$ with $\Gamma(X,\cV) = L$ (see \cite[Chapter 3 \textsection 5 Proposition 5.2]{Milicic:Localization}; in fact, $\cV$ is the unique irreducible quotient of $\cD_\lambda \dotimes_{\cU_\theta} L$). By the classification of irreducible twisted Harish-Chandra sheaves, $\cV$ is isomorphic to $\cL(w^D,\lambda,\eta)$ for a single $D \in W_\Theta \backslash W$.
	
	Since $\cL(w^D,\lambda,\eta)$ is the unique irreducible quotient of $\cM(w^D,\lambda,\eta)$ and $\Gamma(X,-)$ is exact on $\cD_\lambda$-modules, $\Gamma(X,\cL(w^D,\lambda,\eta))$ is equal to the unique irreducible quotient $L(w^D\lambda,\eta)$ of $M(w^D\lambda,\eta)$. Therefore, the equality $L(v \lambda,\eta) = \Gamma(X,\cV) = \Gamma(X,\cL(w^D,\lambda,\eta)) = L(w^D \lambda,\eta)$ implies $M(w^C \lambda,\eta) = M(w^D \lambda,\eta)$ which forces our $D$ to be in the double coset $W_\Theta v W^\lambda$.
	
	It remains to show that such $D$ is minimum in $W_\Theta \backslash W_\Theta v W^\lambda$. Let $C$ be a minimal element in $W_\Theta \backslash W_\Theta v W^\lambda$. The composition factors of $\cM(w^C,\lambda,\eta)$ consist of some $\cL(w^E,\lambda,\eta)$'s with $E \le C$. Taking global sections, we see that the composition factors of $M = \Gamma(X,\cM(w^C,\lambda,\eta))$ consist of some $\Gamma(X,\cL(w^E,\lambda,\eta))$ that are nonzero and with $E \le C$. On the other hand, $L = \Gamma(X,\cL(w^D,\lambda,\eta))$ is a composition factor of $M$. Hence $\Gamma(X, \cL(w^D,\lambda,\eta)) = \Gamma(X,\cL(w^E,\lambda,\eta))$ for some $E \le C$. By the same uniqueness statement appeared in the preceding paragraph, $\cL(w^D,\lambda,\eta) = \cL(w^E,\lambda,\eta)$ and hence $D = E \le C$. By the minimality of $C$, $D = C$. Thus $C=D$ is the minimum element in $W_\Theta \backslash W_\Theta v W^\lambda$ and $\Gamma(X,\cL(w^C,\lambda,\eta)) = L$.
\end{proof}

There exists a number $c \in \BC$ so that $W^\lambda = W_{c\lambda}$. Hence by \ref{thm:cross-section_db_coset}, the set
\begin{equation*}
	A_\Theta^\lambda := A_{c\lambda} \cap (w_\Theta {}^\Theta W)
\end{equation*}
is a cross-section of $W_\Theta \backslash W / W^\lambda$ consisting of the unique shortest elements in each double coset. \ref{lem:irred_vanishing_singular} can be rephrased as follows.

\begin{corollary}\label{lem:irred_vanishing_singular'}
	Let $\lambda \in \fh^*$ be antidominant and let $\eta :\fn \to \BC$ be a character. Let $C \in W_\Theta \backslash W$. The following are equivalent:
	\begin{enumerate}[label=(\alph*)]
		\item $C = W_\Theta v$ for some $v \in A_\Theta^\lambda$;
		\item $\Gamma(X,\cL(w^C,\lambda,\eta)) \neq 0$;
		\item $\Gamma(X,\cL(w^C,\lambda,\eta)) = L(w^C \lambda,\eta)$.
	\end{enumerate}
\end{corollary}

Using these observations, we can write down a character formula for general infinitesimal characters.

\begin{theorem}[Character formula: General case] \label{thm:multiplicity_singular}
	Let $\lambda \in \fh^*$ be antidominant. Let $\eta: \fn \to \BC$ be any character. For any $v \in A_\Theta^\lambda$, let $C = W_\Theta v$, and let $u \in A_{\Theta,\lambda}$ be the unique element such that $C \subseteq W_\Theta u W_\lambda$. Then
	\begin{equation}
		\ch L(v\lambda,\eta) = \ch L(w^C\lambda,\eta) = 
		\sum_{z \in A_\Theta^\lambda \cap (W_\Theta u W_\lambda)}
		\left(
		\sum_{\substack{%
				D \in W_\Theta \backslash W_\Theta z W^\lambda\\%
				D \le_{u,\lambda} C}} 
		P_{CD}^{u,\lambda}(-1) 
		\right)
		\ch M(z \lambda,\eta),
	\end{equation}
	where the polynomials $P_{CD}^{u,\lambda}$ are Whittaker Kazhdan-Lusztig polynomials for $(W_\lambda,\Pi_\lambda,\Theta(u,\lambda))$ as defined in \ref{def:parabolic_KL_poly}. As $v$ ranges over $A_\Theta^\lambda$, $L(v \lambda,\eta)$ exhausts all irreducible objects in $\cN_{\theta,\eta}$.
\end{theorem}

\begin{proof}
	The right hand side is obtained by grouping the right side of (\ref{eqn:character_formula_pre}) based on \ref{lem:std_mods_coincide}. In more detail, the cosets $W_\Theta v W^\lambda$ that are contained in $W_\Theta u W_\lambda$ partition $W_\Theta u W_\lambda$, and $A_\Theta^\lambda \cap (W_\Theta u W_\lambda)$ is a cross-section for this partition. We are simply grouping those standard modules within the same $(W_\Theta,W^\lambda)$-cosets together. The left hand side and the last statement (that those $L(v\lambda,\eta)$'s exhaust all irreducibles) follows from \ref{lem:irred_vanishing_singular'} and the exactness of $\Gamma(X,-)$.
\end{proof}

\section{An example in $A_3$}\label{sec:examples}

\allowdisplaybreaks
The $A_3$ root system (pictured below) is the smallest example in which all nontrivial phenomena appear. To make the picture more readable, only the positive roots are connected to the origin. Here $\lambda$ can be chosen to be $\lambda = -m \rho + c(-\alpha + 2\beta + \gamma)$ for any nonzero number $c$ transcendental over $\BQ$ and any large enough integer $m$ so that $\lambda$ is antidominant regular ($-\alpha + 2\beta +\gamma$ is a vector perpendicular to the plane spanned by $\alpha+\beta$ and $\gamma$).

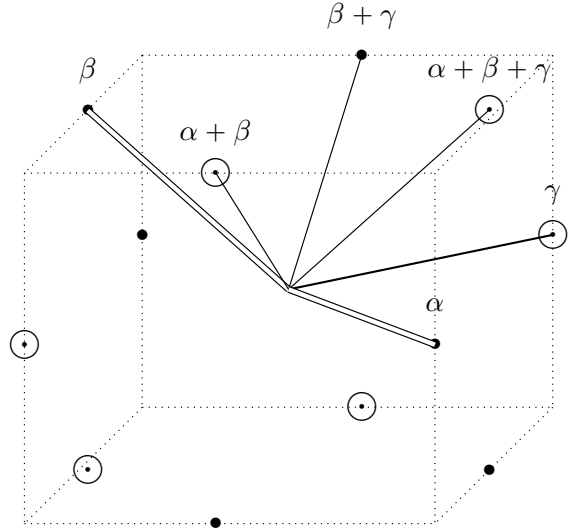
\begin{equation*}
	\begin{tikzcd}[start anchor = real center, end anchor = real center, column sep = 2ex, row sep = 1ex]
		&[-1.5ex] &[-1.5ex] \phantom{\bullet} \ar[rrrrrr, dotted, no head] \ar[ddll, dotted, no head] \ar[dddddd, dotted, no head] & & & \bullet \ar[loop, phantom, "\beta + \gamma", distance = 1cm] & &[-1.5ex] &[-1.5ex] \phantom{\bullet} \ar[ddll, dotted, no head] \ar[dddddd, dotted, no head]\\
		& \bullet \ar[loop, phantom, "\beta", distance = 1cm] & & & & & & \bigodot \ar[loop, phantom, "\alpha + \beta + \gamma", distance = 1cm]\\
		\phantom{\bullet} \ar[rrrrrr, dotted, no head] \ar[dddddd, dotted, no head] & & & \bigodot \ar[loop, phantom, "\alpha + \beta", distance = 1cm] & & & \phantom{\bullet} \ar[dddddd, dotted, no head]\\
		& & \bullet & & & & & & \bigodot \ar[loop, phantom, "\gamma", distance = 1cm]\\
		& & & & \phantom{\bullet} \ar[llluuu, dash, equal] \ar[luu, dash] \ar[ruuuu, dash] \ar[rrruuu, dash] \ar[rrrru, dash, thick] \ar[drr, dash, equal]\\
		\bigodot & & & & & & \bullet \ar[loop, phantom, "\alpha", distance = 1cm]\\
		& & \phantom{\bullet} \ar[rrrrrr, dotted, no head] \ar[ddll, dotted, no head] & & & \bigodot & & & \phantom{\bullet} \ar[ddll, dotted, no head]\\
		& \bigodot & & & & & & \bullet\\
		\phantom{\bullet} \ar[rrrrrr, dotted, no head] & & & \bullet & & & \phantom{\bullet}
	\end{tikzcd}
\end{equation*}

In the above diagram, $\{\alpha,\beta,\gamma\}$ are simple roots, $\Theta = \{\alpha,\beta\}$ which is indicated by double lines in the picture, roots in $\Sigma_\lambda$ are marked by $\bigodot$, and those not in $\Sigma_\lambda$ are marked by $\bullet$. 

Below is a diagram of the Weyl group, arranged in a way so that elements in the same right $W_\Theta$-coset are grouped together and are connected by double lines. Elements surrounded by shapes are the longest elements in right $W_\Theta$-cosets. Elements that are crossed out are those in $W_\lambda$. There are two double $(W_\Theta,W_\lambda)$-cosets: elements in $W_\Theta s_\gamma s_\beta W_\lambda$ are underlined; elements in $W_\Theta W_\lambda$ are those that are not underlined. 

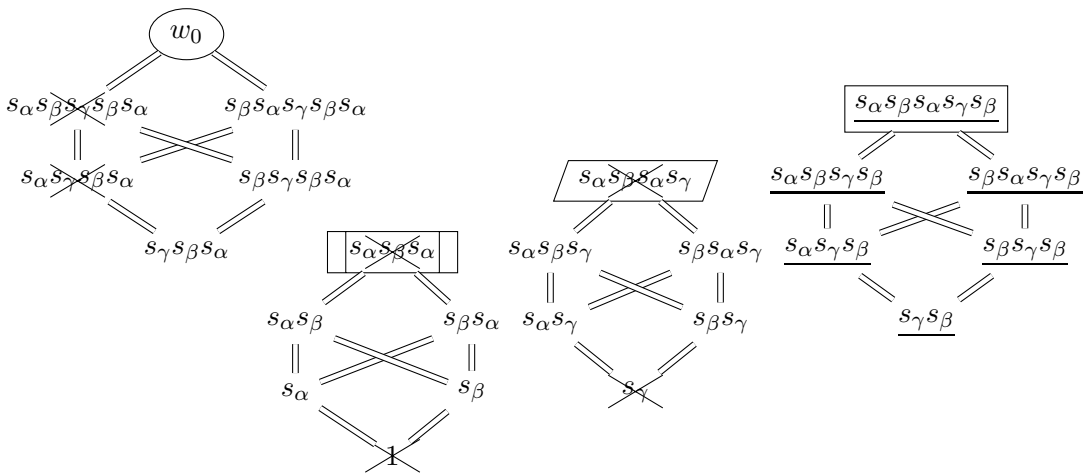
\begin{equation}\label{diag:W(A3)}
	\begin{tikzpicture}
		\matrix[column sep = 1ex, row sep = 2ex]{
			&[-3ex] \node[ellipse, draw] (abarba) {$w_0$}; &[-3ex] &[-5ex] &[-3ex] &[-2ex] &[-5ex] &[-5ex] &[-2ex] &[-5ex] &[-5ex]\\
			\node (abrba) {$s_\alpha s_\beta s_\gamma s_\beta s_\alpha$}; \node[cross out, draw] {\phantom{sss}}; & & 
			\node (barba) {$s_\beta s_\alpha s_\gamma s_\beta s_\alpha$}; &&&&&&&
			\node[rectangle, draw] (abarb) {$\uline{s_\alpha s_\beta s_\alpha s_\gamma s_\beta}$};\\
			\node (arba) {$s_\alpha s_\gamma s_\beta s_\alpha$}; \node[cross out, draw] {\phantom{sss}}; & &
			\node (brba) {$s_\beta s_\gamma s_\beta s_\alpha$}; & & & &
			\node[trapezium, trapezium left angle=70, trapezium right angle=-70, draw] (abar) {$s_\alpha s_\beta s_\alpha s_\gamma$}; \node[cross out, draw] {\phantom{sss}}; & &
			\node (abrb) {$\uline{s_\alpha s_\beta s_\gamma s_\beta}$}; & &
			\node (barb) {$\uline{s_\beta s_\alpha s_\gamma s_\beta}$};\\
			& \node (rba) {$s_\gamma s_\beta s_\alpha$}; & & 
			\node[rectangle split, rectangle split horizontal, rectangle split parts=3, draw, minimum height=1.2\baselineskip, inner sep=1pt] (aba) {\nodepart{two}$s_\alpha s_\beta s_\alpha$}; \node[cross out, draw] {\phantom{sss}}; & &
			\node (abr) {$s_\alpha s_\beta s_\gamma$}; & &
			\node (bar) {$s_\beta s_\alpha s_\gamma$}; & 
			\node (arb) {$\uline{s_\alpha s_\gamma s_\beta}$}; & &
			\node (brb) {$\uline{s_\beta s_\gamma s_\beta}$};\\
			& & \node (ab) {$s_\alpha s_\beta$}; & &
			\node (ba) {$s_\beta s_\alpha$}; &
			\node (ar) {$s_\alpha s_\gamma$}; & &
			\node (br) {$s_\beta s_\gamma$}; & &
			\node (rb) {$\uline{s_\gamma s_\beta}$};\\
			& & \node (a) {$s_\alpha$}; & &
			\node (b) {$s_\beta$}; & &
			\node (r) {$s_\gamma$}; \node[cross out, draw] {\phantom{sss}};\\
			& & & \node (1) {$1$}; \node[cross out, draw] {\phantom{sss}};\\
		};
		\draw [double equal sign distance] (1) to (a);
		\draw [double equal sign distance] (1) to (b);
		\draw [double equal sign distance] (a) to (ab);
		\draw [double equal sign distance] (b) to (ba);
		\draw [double equal sign distance] (a) to (ba);
		\draw [double equal sign distance] (b) to (ab);
		\draw [double equal sign distance] (ab) to (aba);
		\draw [double equal sign distance] (ba) to (aba);
		\draw [double equal sign distance] (r) to (ar);
		\draw [double equal sign distance] (r) to (br);
		\draw [double equal sign distance] (ar) to (abr);
		\draw [double equal sign distance] (br) to (bar);
		\draw [double equal sign distance] (ar) to (bar);
		\draw [double equal sign distance] (br) to (abr);
		\draw [double equal sign distance] (abr) to (abar);
		\draw [double equal sign distance] (bar) to (abar);
		\draw [double equal sign distance] (rb) to (arb);
		\draw [double equal sign distance] (rb) to (brb);
		\draw [double equal sign distance] (arb) to (abrb);
		\draw [double equal sign distance] (brb) to (barb);
		\draw [double equal sign distance] (arb) to (barb);
		\draw [double equal sign distance] (brb) to (abrb);
		\draw [double equal sign distance] (abrb) to (abarb);
		\draw [double equal sign distance] (barb) to (abarb);
		\draw [double equal sign distance] (rba) to (arba);
		\draw [double equal sign distance] (rba) to (brba);
		\draw [double equal sign distance] (arba) to (abrba);
		\draw [double equal sign distance] (brba) to (barba);
		\draw [double equal sign distance] (arba) to (barba);
		\draw [double equal sign distance] (brba) to (abrba);
		\draw [double equal sign distance] (abrba) to (abarba);
		\draw [double equal sign distance] (barba) to (abarba);
	\end{tikzpicture}
\end{equation}

Let's first look at the double coset $W_\Theta s_\gamma s_\beta W_\lambda$, with $u = s_\gamma s_\beta \in A_{\Theta,\lambda}$. It equals a single left $W_\lambda$-coset $s_\gamma s_\beta W_\lambda$ and a single right $W_\Theta$-coset $W_\Theta s_\gamma s_\beta$. Hence 
\begin{equation*}
	\Theta(s_\gamma s_\beta,\lambda) = \Pi_\lambda,\quad
	W_{\lambda,\Theta(s_\gamma s_\beta,\lambda)} = W_\lambda,
\end{equation*}
and $(W_\Theta s_\gamma s_\beta)|_\lambda = W_\lambda 1$, the unique right $W_\lambda$-coset in $W_\lambda$. Therefore
\begin{align*}
	\varphi_\lambda(W_\Theta s_\gamma s_\beta) &= \delta_{W_\Theta s_\gamma s_\beta},\\
	\ch L(s_\gamma s_\beta \lambda,\eta) &= \ch M(s_\gamma s_\beta \lambda,\eta).
\end{align*}

Now let's look at the other double coset $W_\Theta  W_\lambda$, with $u = 1$ and
\begin{equation*}
	\Theta(1,\lambda) = \{\alpha+\beta\}, \quad
	W_{\lambda,	\Theta(1,\lambda)} = \{1, s_{\alpha+\beta}\}.
\end{equation*}
For convenience, we write
\begin{equation*}
	W_\bullet := W_{\lambda,	\Theta(1,\lambda)}.
\end{equation*}
The root system $\Sigma_\lambda$ and a diagram for $(W_\lambda,\Pi_\lambda,\Theta_\lambda^1)$ are
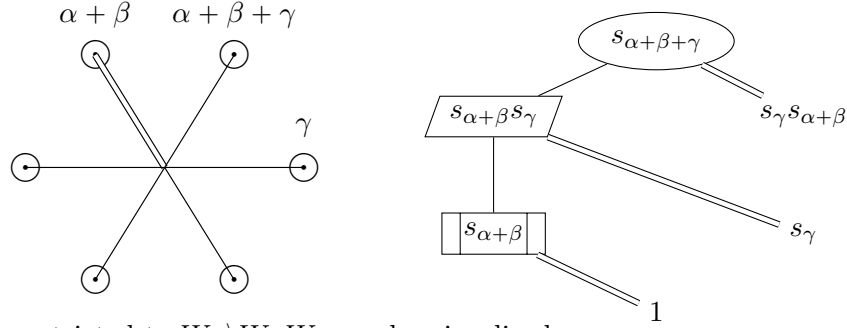
\begin{equation*}
	\begin{tikzcd}[start anchor = real center, end anchor = real center, column sep = 1ex, row sep = 5ex]
		& \bigodot \ar[dr, equal] \ar[loop, phantom, "\alpha+\beta", distance = 1cm] &  & \bigodot \ar[loop, phantom, "\alpha+\beta+\gamma", distance = 1cm] \ar[dl, dash]\\
		\bigodot \ar[rr, dash] && \phantom{\bigodot} \ar[dl, dash] \ar[dr, dash] \ar[rr, dash] && \bigodot \ar[loop, phantom, "\gamma", distance = 1cm]\\
		& \bigodot &  & \bigodot 
	\end{tikzcd}\qquad
	\begin{adjustbox}{trim = 0 2.3cm 0 0} 
		\begin{tikzpicture}
			\matrix[row sep = 1cm,column sep = 0.2cm]{
				& \node[ellipse, draw] (abr) {$s_{\alpha+\beta+\gamma}$};\\[-4ex]
				\node[trapezium, trapezium left angle=70, trapezium right angle=-70, draw] (ab-r) {$s_{\alpha+\beta} s_\gamma$}; &&%
				\node (r-ab) {$s_\gamma s_{\alpha+\beta}$};\\
				\node[rectangle split, rectangle split horizontal, rectangle split parts=3, draw, minimum height=1.2\baselineskip, inner sep=1pt] (ab) {\nodepart{two}$s_{\alpha+\beta}$}; &&%
				\node (r) {$s_\gamma$};\\[-3ex]
				& \node (1) {$1$};\\
			};
			\draw [double equal sign distance] (abr) to  (r-ab) ;
			\draw [double equal sign distance] (ab-r) to  (r) ;
			\draw [double equal sign distance] (ab) to  (1) ;
			\draw (abr) to (ab-r) ;
			\draw (ab-r) to (ab);
		\end{tikzpicture}.
	\end{adjustbox}
\end{equation*}
The map $(-)|_\lambda$ restricted to $W_\Theta \backslash W_\Theta W_\lambda$ can be visualized as
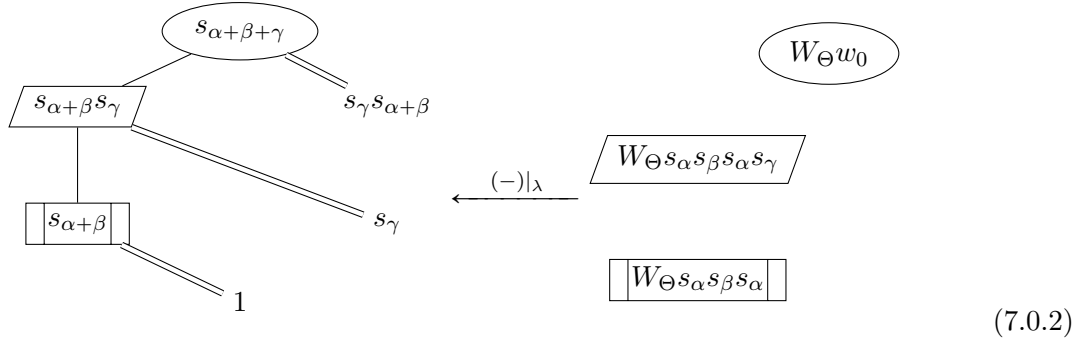
\begin{equation}\label{diag:(-)|_lambda_A3}
	\begin{tikzpicture}
		\matrix[row sep = 1cm,column sep = 0.2cm]{
			& \node[ellipse, draw] (abr) {$s_{\alpha+\beta+\gamma}$};\\[-4ex]
			\node[trapezium, trapezium left angle=70, trapezium right angle=-70, draw] (ab-r) {$s_{\alpha+\beta} s_\gamma$}; &&%
			\node (r-ab) {$s_\gamma s_{\alpha+\beta}$};\\
			\node[rectangle split, rectangle split horizontal, rectangle split parts=3, draw, minimum height=1.2\baselineskip, inner sep=1pt] (ab) {\nodepart{two}$s_{\alpha+\beta}$}; &&%
			\node (r) {$s_\gamma$};\\[-3ex]
			& \node (1) {$1$};\\
		};
		\draw [double equal sign distance] (abr) to  (r-ab) ;
		\draw [double equal sign distance] (ab-r) to  (r) ;
		\draw [double equal sign distance] (ab) to  (1) ;
		\draw (abr) to (ab-r) ;
		\draw (ab-r) to (ab);
	\end{tikzpicture}
	\raisebox{10ex}{$\xleftarrow{\quad (-)|_\lambda \quad}$}
	\begin{adjustbox}{trim = 0 1cm 0 0}
		\begin{tikzpicture}
			\matrix[row sep = 1cm,column sep = 0.2cm]{
				&[-5ex] \node[ellipse, draw] (abarba) {$W_\Theta w_0$};\\[-2ex]
				\node[trapezium, trapezium left angle=70, trapezium right angle=-70, draw] (abar) {$W_\Theta s_\alpha s_\beta s_\alpha s_\gamma$};\\
				\node[rectangle split, rectangle split horizontal, rectangle split parts=3, draw, minimum height=1.2\baselineskip, inner sep=1pt] (aba) {\nodepart{two}$W_\Theta s_\alpha s_\beta s_\alpha$};\\
				\node {};\\
			};
		\end{tikzpicture}
	\end{adjustbox}
\end{equation}
where a coset on the right hand side is sent to the coset on the left with the same shape. The Whittaker Kazhdan-Lusztig polynomials for $(W_\lambda,\Pi_\lambda,\Theta_\lambda^1)$ are
\begin{equation*}
	\begin{tabu}{c|ccc}
		P_{EF}^{1,\lambda} & W_\bullet s_{\alpha+\beta} & W_\bullet s_{\alpha+\beta} s_\gamma & W_\bullet s_{\alpha+\beta+\gamma}\\ \hline
		W_\bullet s_{\alpha+\beta} & 1 & 0 & 0\\
		W_\bullet s_{\alpha+\beta} s_\gamma & q & 1 & 0\\
		W_\bullet s_{\alpha+\beta+\gamma} & 0 & q & 1
	\end{tabu}
\end{equation*}
Hence
\begin{align*}
	\varphi_\lambda(W_\Theta s_\alpha s_\beta s_\alpha)
	&= P_{(W_\bullet s_{\alpha+\beta}), (W_\bullet s_{\alpha+\beta})}^{1,\lambda} \delta_{W_\Theta s_\alpha s_\beta s_\alpha}\\ 
	&\qquad + P_{(W_\bullet s_{\alpha+\beta}), (W_\bullet s_{\alpha+\beta} s_\gamma)}^{1,\lambda} \delta_{W_\Theta s_\alpha s_\beta s_\alpha s_\gamma}\\
	&\qquad + P_{(W_\bullet s_{\alpha+\beta}), (W_\bullet s_{\alpha+\beta+\gamma})}^{1,\lambda} \delta_{W_\Theta w_0}\\
	&= \delta_{W_\Theta s_\alpha s_\beta s_\alpha},\\
	\varphi_\lambda(W_\Theta s_\alpha s_\beta s_\alpha s_\gamma) 
	&= P_{(W_\bullet s_{\alpha+\beta} s_\gamma), (W_\bullet s_{\alpha+\beta})}^{1,\lambda} \delta_{W_\Theta s_\alpha s_\beta s_\alpha}\\ 
	&\qquad + P_{(W_\bullet s_{\alpha+\beta} s_\gamma), (W_\bullet s_{\alpha+\beta} s_\gamma)}^{1,\lambda} \delta_{W_\Theta s_\alpha s_\beta s_\alpha s_\gamma}\\
	&\qquad + P_{(W_\bullet s_{\alpha+\beta} s_\gamma), (W_\bullet s_{\alpha+\beta+\gamma})}^{1,\lambda} \delta_{W_\Theta w_0}\\
	&= q \delta_{W_\Theta s_\alpha s_\beta s_\alpha} + \delta_{W_\Theta s_\alpha s_\beta s_\alpha s_\gamma},\\
	\varphi_\lambda(W_\Theta w_0) 
	&= P_{(W_\bullet s_{\alpha+\beta+\gamma}), (W_\bullet s_{\alpha+\beta})}^{1,\lambda} \delta_{W_\Theta s_\alpha s_\beta s_\alpha}\\ 
	&\qquad + P_{(W_\bullet s_{\alpha+\beta+\gamma}), (W_\bullet s_{\alpha+\beta} s_\gamma)}^{1,\lambda} \delta_{W_\Theta s_\alpha s_\beta s_\alpha s_\gamma}\\
	&\qquad + P_{(W_\bullet s_{\alpha+\beta+\gamma}), (W_\bullet s_{\alpha+\beta+\gamma})}^{1,\lambda} \delta_{W_\Theta w_0}\\
	&= q \delta_{W_\Theta s_\alpha s_\beta s_\alpha s_\gamma} + \delta_{W_\Theta  w_0}.
\end{align*}
Specializing to $q = -1$, we get
\begin{align*}
	\ch L(s_\alpha s_\beta s_\alpha \lambda,\eta) 
	&= \ch M(s_\alpha s_\beta s_\alpha \lambda,\eta),\\
	\ch L( s_\alpha s_\beta s_\alpha s_\gamma \lambda,\eta)
	&= -\ch M(s_\alpha s_\beta s_\alpha \lambda,\eta) + \ch M( s_\alpha s_\beta s_\alpha s_\gamma \lambda,\eta),\\
	\ch L(w_0 \lambda,\eta)
	&= -\ch M( s_\alpha s_\beta s_\alpha s_\gamma \lambda,\eta) + \ch M(w_0 \lambda,\eta).
\end{align*}

\printbibliography

\end{document}